\documentclass[
reqno,
10pt,
oneside
]{article}

\usepackage[ngerman, english]{babel}
\usepackage[utf8]{inputenc}
\usepackage[normalem]{ulem}
\usepackage[babel,french=guillemets,german=swiss]{csquotes}
\usepackage[center,font=small]{caption}
\usepackage{cite}
\usepackage{bbm}
\usepackage[raggedright]{titlesec}
\usepackage{xcolor}
\usepackage{enumerate}

\titleformat{\section}{\normalfont\large\bfseries}{\thesection}{1em}{}
\titleformat{\subsection}{\normalfont\bfseries}{\thesubsection}{1em}{}

\usepackage{tocbasic}
\DeclareTOCStyleEntry[
beforeskip=.2em plus 1pt,
]{tocline}{section}

\usepackage[%
left=1.5in,
right=1.5in,
bottom=1in,
top=0.8in,
footskip=0.4in,
]{geometry}

\usepackage{color}
\definecolor{LinkColor}{rgb}{0,0,1}
\definecolor{LinkColor2}{rgb}{0,0.5,0}
\definecolor{lbcolor}{rgb}{0.85,0.85,0.85}
\definecolor{FrameColor}{rgb}{0.85,0.85,0.85}
\definecolor{rosso}{rgb}{0.8,0,0}
\definecolor{lightgray}{rgb}{0.5,0.5,0.5}
\definecolor{violet}{rgb}{0.65,0,0.65}
\definecolor{darkgreen}{rgb}{0,0.5,0}

\usepackage{enumitem}

\usepackage{graphicx}
\usepackage{placeins}
\usepackage{overpic}

\usepackage{amsmath}
\usepackage{amssymb}
\usepackage{dsfont}
\usepackage{empheq}
\allowdisplaybreaks
\usepackage{amsthm}

\usepackage[%
pdftitle={Titel},%
pdfauthor={Autor},%
pdfcreator={LaTeX, LaTeX with hyperref and KOMA-Script},
pdfsubject={Betreff}, 
pdfkeywords={Keywords}
]{hyperref} 

\hypersetup{%
	colorlinks	=true,
	linkcolor	=LinkColor,%
	anchorcolor	=LinkColor,%
	citecolor	=LinkColor2,%
	filecolor	=LinkColor,%
	menucolor	=LinkColor,%
	urlcolor	=LinkColor,%
}
\usepackage{marginnote}

\usepackage{verbatim}

\newtheorem{theorem}{Theorem}[section]
\newtheorem{lemma}[theorem]{Lemma}
\newtheorem{proposition}[theorem]{Proposition}

\newtheorem{definition}[theorem]{Definition}

\theoremstyle{definition}
\newtheorem{remark}[theorem]{Remark}

\makeatletter
\renewenvironment{proof}[1][\proofname]{%
	\par\pushQED{\qed}\normalfont%
	\topsep6\p@\@plus6\p@\relax
	\trivlist\item[\hskip\labelsep\bfseries#1\@addpunct{.}]%
	\ignorespaces
}{%
	\popQED\endtrivlist\@endpefalse
}
\makeatother

\makeatletter
\renewcommand\paragraph{\@startsection{paragraph}{4}{\z@}%
	{1ex \@plus1ex \@minus.2ex}%
	{-1em}%
	{\normalfont\normalsize\bfseries}}
\renewcommand\subparagraph{\@startsection{paragraph}{4}{\z@}%
	{1ex \@plus1ex \@minus.2ex}%
	{-1em}%
	{\normalfont\normalsize\itshape}}
\makeatother

\newcommand{\abs}[1]{\left| #1 \right|}

\newcommand{\norm}[1]{\| #1 \|}
\newcommand{\bignorm}[1]{\big\| #1 \big\|}
\newcommand{\ang}[2]{ \langle #1 , #2  \rangle}
\newcommand{\bigang}[2]{ \big< #1 , #2  \big>}
\newcommand{\scp}[2]{ \left( #1 , #2  \right)}
\newcommand{\bigscp}[2]{\big( #1 , #2 \big)}

\newcommand{\meano}[1]{{\langle #1 \rangle}_{\Omega}}
\newcommand{\meang}[1]{{\langle #1 \rangle}_{\Gamma}}

\newcommand{\R}{\mathbb R}
\newcommand{\N}{\mathbb N}
\newcommand{\n}{\mathbf{n}}

\newcommand{\intO}{\int_\Omega}

\newcommand{\intG}{\int_\Gamma}

\newcommand{\ep}{\varepsilon}

\newcommand{\dtau}{\;\mathrm d\tau}
\newcommand{\dx}{\;\mathrm{d}x}

\newcommand{\dt}{\;\mathrm dt}
\newcommand{\ds}{\;\mathrm ds}

\newcommand{\dxt}{\;\mathrm{d}x\;\mathrm{d}t}
\newcommand{\dGt}{\;\mathrm{d}\Ga\;\mathrm{d}t}
\newcommand{\dxs}{\;\mathrm{d}x\;\mathrm{d}s}
\newcommand{\dGs}{\;\mathrm{d}\Ga\;\mathrm{d}s}

\newcommand{\dG}{\;\mathrm d\Ga}

\newcommand{\h}{\mathds{h}}

\newcommand{\ddt}{\frac{\mathrm d}{\mathrm dt}}

\newcommand{\del}{\partial}
\newcommand{\delt}{\partial_{t}}

\newcommand{\deln}{\partial_\n}

\newcommand{\Grad}{\nabla}
\newcommand{\Lap}{\Delta}
\newcommand{\Div}{\textnormal{div}}
\newcommand{\Gradg}{\nabla_\Ga}
\newcommand{\Lapg}{\Delta_\Ga}
\newcommand{\Divg}{\textnormal{div}_\Ga}

\newcommand{\emb}{\hookrightarrow}

\newcommand{\suchthat}{\;\ifnum\currentgrouptype=16 \middle\fi|\;}

\newcommand{\Om}{\Omega}
\newcommand{\Ga}{\Gamma}

\newcommand{\projam}{\mathbb{P}_{\mathcal{A}_m}}
\newcommand{\projbm}{\mathbb{P}_{\mathcal{B}_m}}

\newcommand{\extended}[1]{{\color{lightgray}#1}}
\renewcommand{\extended}[1]{}

\begin{document}

%
%

\title{\bfseries 
    Well-posedness of a bulk-surface\\ 
    convective Cahn--Hilliard system\\
    with dynamic boundary conditions
\\[-1.5ex]$\;$}

\author{
	Patrik Knopf \footnotemark[1]
    \and Jonas Stange \footnotemark[1]}

\date{ }

\maketitle

\renewcommand{\thefootnote}{\fnsymbol{footnote}}

\footnotetext[1]{
    Faculty for Mathematics, 
    University of Regensburg, 
    93053 Regensburg, 
    Germany \newline
	\tt(%
        \href{mailto:patrik.knopf@ur.de}{patrik.knopf@ur.de},
        \href{mailto:jonas.stange@ur.de}{jonas.stange@ur.de}%
        ).
}

\begin{center}
	\scriptsize
	{
		\textit{This is a preprint version of the paper. Please cite as:} \\  
		P.~Knopf, J.~Stange, 
        \textit{NoDEA Nonlinear Differential Equations Appl.} \textbf{31}:82 (2024), \\
		\url{https://doi.org/10.1007/s00030-024-00970-3}
	}
\end{center}

\medskip

%
%

\begin{small}
\begin{center}
    \textbf{Abstract}
\end{center}
We consider a general class of bulk-surface convective Cahn--Hilliard systems with dynamic boundary conditions. In contrast to classical Neumann boundary conditions, the dynamic boundary conditions of Cahn--Hilliard type allow for dynamic changes of the contact angle between the diffuse interface and the boundary, a convection-induced motion of the contact line as well as absorption of material by the boundary. The coupling conditions for bulk and surface quantities involve parameters $K,L\in[0,\infty]$, whose choice declares whether these conditions are of Dirichlet, Robin or Neumann type. We first prove the existence of a weak solution to our model in the case $K,L\in (0,\infty)$ by means of a Faedo--Galerkin approach. For all other cases, the existence of a weak solution is then shown by means of the asymptotic limits, where $K$ and $L$ are sent to zero or to infinity, respectively. Eventually, we establish higher regularity for the phase-fields, and we prove the uniqueness of weak solutions given that the mobility functions are constant.
\\[1ex]
\\[1ex]	
\end{small}

\begin{small}
\setcounter{tocdepth}{2}
\hypersetup{linkcolor=black}
\tableofcontents
\end{small}

\setlength\parindent{0ex}
\setlength\parskip{1ex}
\allowdisplaybreaks
\numberwithin{equation}{section}
\renewcommand{\thefootnote}{\arabic{footnote}}

\newpage

\section{Introduction} 
\label{SECT:INTRO}

In recent times, the description of immiscible \textit{two-phase flows} in a bounded domain by diffuse-interface models has become very popular. This is especially because those systems can usually be handled more easily in terms of mathematical analysis than their sharp-interface counterparts. 
In such models, the location of the two fluids inside the container is represented by an order parameter, the so-called \textit{phase-field}. The time evolution of this phase-field function is often described by a \textit{convective Cahn--Hilliard equation}. There, the velocity field of the mixture enters via the material derivative of the phase-field. The time evolution of the velocity field is usually described by a fluid equation, e.g., the \textit{Navier--Stokes equation}. Very frequently used \textit{Navier--Stokes--Cahn--Hilliard} models for two-phase flows are the \textit{Model H} (see \cite{Gurtin1996,Hohenberg1977}), which covers the case where both fluids have the same individual density, and the \textit{Abels--Garcke--Grün model (AGG model)} (see \cite{Abels2012}), which is even capable of describing the situation where both fluids have different (i.e., \textit{unmatched}) densities.

For the Model H and the AGG model, the standard choice are homogeneous Neumann boundary conditions for the Cahn--Hilliard quantities (i.e., the phase-field and the chemical potential) as well as a no-slip boundary condition for the velocity field. However, these choices lead to some crucial limitations of these models:
\begin{itemize}[leftmargin=2em, topsep=0em, partopsep=0em, parsep=0em, itemsep=0em]
	\item The contact angle between the diffuse interface and the boundary of the domain is fixed at ninety degrees at all times, which is unrealistic for many real-world applications.
	\item The motion of the contact line, where the diffuse interface intersects the boundary of the domain, is driven only by diffusion. This means that a motion of the contact line caused by convection is not taken into account.
	\item No transfer of material between the bulk and the boundary of the domain is allowed. Therefore, any absorption of material by the boundary cannot be described.
\end{itemize}
For a more detailed discussion of these issues, we refer the reader to \cite{Giorgini2023}. 

To overcome these limitations, a new Navier--Stokes--Cahn--Hilliard model with dynamic boundary conditions was derived in \cite{Giorgini2023}. It can be regarded as an extension of the AGG model and is capable of describing two-phase flows with unmatched densities. In the case of matched densities, some first analytic results (namely the existence of weak solution and their uniqueness under certain additional assumptions) were presented in \cite{Giorgini2023}. The case of unmatched densities, however, is much more involved. This is mainly because in the Navier--Stokes equation, the density function then depends on the phase-field, and a further flux term related to the interfacial motion will appear. 
In the case of unmatched densities, at least for certain parameter choices in the dynamic boundary conditions, the existence of a weak solution to the model introduced in \cite{Giorgini2023} was shown in \cite{Gal2023a} by an implicit time discretization scheme.
A different strategy to prove the existence of weak solutions to the model proposed in \cite{Giorgini2023}, in a unified framework for all parameter choices in the dynamic boundary conditions, is to first analyze the underlying bulk-surface convective Cahn--Hilliard subsystem separately. Using this information, a weak solution to the full Navier--Stokes--Cahn--Hilliard system is then to be constructed by means of a suitable fixed point argument.

This motivates us to investigate the following \textit{bulk-surface convective Cahn--Hilliard system with dynamic boundary conditions}:
\begin{subequations}\label{CH}
    \begin{align}
        \label{CH:1}
        &\delt\phi + \Div(\phi\boldsymbol{v}) = \Div(m_\Om(\phi)\Grad\mu) && \text{in} \ Q, \\
        \label{CH:2}
        &\mu = -\epsilon \Lap\phi + \epsilon^{-1}F'(\phi)   && \text{in} \ Q, \\
        \label{CH:3}
        &\delt\psi + \Divg(\psi\boldsymbol{w}) = \Divg(m_\Ga(\psi)\Gradg\theta) - \beta m_\Om(\phi)\deln\mu && \text{on} \ \Sigma, \\
        \label{CH:4}
        &\theta = - \epsilon_{\Gamma}\kappa\Lapg\psi + \epsilon_{\Gamma}^{-1}G'(\psi) + \alpha\epsilon\deln\phi && \text{on} \ \Sigma, \\
        \label{CH:5}
        &\begin{cases} \epsilon K\deln\phi = \alpha\psi - \phi &\text{if} \ K\in [0,\infty), \\
        \deln\phi = 0 &\text{if} \ K = \infty
        \end{cases} && \text{on} \ \Sigma, \\
        \label{CH:6}
        &\begin{cases} 
        L m_\Om(\phi)\deln\mu = \beta\theta - \mu &\text{if} \  L\in[0,\infty), \\
        m_\Om(\phi)\deln\mu = 0 &\text{if} \ L=\infty
        \end{cases} &&\text{on} \ \Sigma, \\
        \label{CH:7}
        &\phi\vert_{t=0} = \phi_0 &&\text{in} \ \Om, \\
        \label{CH:8}
        &\psi\vert_{t=0} = \psi_0 &&\text{on} \ \Ga.
    \end{align}
\end{subequations}
Here, $\Om\subset\R^d$ with $d\in\{2,3\}$ is a bounded domain with boundary $\Ga\coloneqq\partial\Om$, $T>0$ is a prescribed final time, and we set $Q\coloneqq \Om\times(0,T)$ and $\Sigma\coloneqq\Ga\times(0,T)$. 
The outward unit normal vector field on $\Gamma$ is denoted by $\n$.
Moreover, $\Gradg$, $\Divg$ and $\Lapg$ denote the surface gradient, the surface divergence, and the Laplace-Beltrami operator on $\Ga$, respectively. 

The function $\phi:Q\rightarrow\R$ is the bulk phase-field. It is an order parameter, which represents the distribution of the two immiscible materials within the domain $\Omega$. Similarly, the surface phase-field $\psi:\Sigma\rightarrow\R$ represents the distribution of two materials on the surface. 
The functions $F$ and $G$ are double-well potentials, which lead to the effect that $\phi(t)$ and $\psi(t)$ will attain values close to $\pm 1$ in most parts of $\Omega$ and $\Gamma$ (which correspond to the pure phases of the materials). In the remaining intermediate regions, which are called diffuse interfaces, $\phi$ and $\psi$ are expected to exhibit a continuous transition between the values $-1$ and $1$.  
The functions $\mu:Q\rightarrow\R$ and $\theta:\Sigma\rightarrow\R$ represent the bulk chemical potential and the surface chemical potential, respectively. The non-negative, scalar functions $m_\Om(\phi)$ and $m_\Ga(\psi)$ are called mobilities. They depend on the phase-fields and describe where diffusion processes occur. Moreover, the functions $\boldsymbol{v}:Q\rightarrow\R^d$ and $\boldsymbol{w}:\Sigma\rightarrow\R^d$ are prescribed velocity fields that correspond to the flow of the two materials in the bulk and on the surface, respectively. 
In this paper, we always assume that $\boldsymbol{v} \cdot \n = 0$ and $\boldsymbol{w} \cdot \n = 0$ on $\Sigma$.
We point out that in many cases (e.g., in the model derived in \cite{Giorgini2023}), $\boldsymbol{v}$ and $\boldsymbol{w}$ are related by the condition $\boldsymbol{w} = \boldsymbol{v}\vert_\Sigma$ on $\Sigma$. 
However, as this relation will not have any impact on our mathematical analysis, we consider $\boldsymbol{v}$ and $\boldsymbol{w}$ as independent functions to keep our model as general as possible.
Furthermore, $\epsilon>0$ is a parameter related to the thickness of the diffuse interface in the bulk, whereas $\epsilon_\Gamma > 0$ corresponds to the width of the diffuse interface on the boundary. The parameter $\kappa > 0$ acts as a weight for the surface Dirichlet energy $\psi\mapsto \intG \abs{\Gradg \psi}^2 \dG$, which has a smoothing effect on the phase separation at the boundary.

The time evolution of $(\phi,\mu)$ is determined by the \textit{bulk convective Cahn--Hilliard subsystem} $\big(\eqref{CH:1},\eqref{CH:2}\big)$, whereas the evolution of $(\psi,\theta)$ is described by the \textit{surface convective Cahn--Hilliard subsystem} $\big(\eqref{CH:3},\eqref{CH:4}\big)$, which is coupled to the bulk by expressions involving the normal derivatives $\deln\phi$ and $\deln\mu$. The bulk and surface quantities are further coupled by the boundary conditions \eqref{CH:5} and \eqref{CH:6}, which depend on parameters $K,L\in [0,\infty]$ and $\alpha,\beta\in\R$.

In \eqref{CH:5} and \eqref{CH:6}, the parameters $K,L\in [0,\infty]$ are used to distinguish different cases, each corresponding to a certain solution behaviour related to a physical phenomenon. The case that has been studied most extensively in the literature is $K=0$. In this case, \eqref{CH:5} is to be interpreted as the Dirichlet type boundary condition $\phi = \alpha\psi$ on $\Sigma$.
This choice makes particular sense along with $\alpha=1$, if the materials on the boundary are simply considered as an extension of those in the bulk. Originally, the boundary condition \eqref{CH:6} was introduced in the \textit{non-convective case} (i.e., $\boldsymbol{v}\equiv 0$ and $\boldsymbol{w}\equiv 0$) in the following literature: 
\begin{itemize}[leftmargin=2em, topsep=0em, partopsep=0em, parsep=0em, itemsep=0em]
	\item The choice $L=0$ was proposed in \cite{Gal2006} and \cite{Goldstein2011}. Then, \eqref{CH:6} can be restated as the Dirichlet condition $\mu = \beta\theta$ on $\Sigma$, which means that the chemical potentials $\mu$ and $\theta$ are always in a chemical equilibrium. In this case, a rapid transfer of material between bulk and boundary can be expected (see, e.g., \cite{Knopf2021a}).
	\item The choice $L=\infty$ was introduced in \cite{Liu2019}. In this case, \eqref{CH:6} is a homogeneous Neumann boundary condition on $\mu$, which means that the mass flux between bulk and surface is zero. Consequently, no transfer of material between bulk and surface will occur.
	\item The choice $L\in (0,\infty)$ was first used in \cite{Knopf2021a} to interpolate between the extreme cases $L=0$ and $L=\infty$. In this case, a transfer of material between bulk and surface will occur, and the number $L^{-1}$ is related to the rate of absorption of bulk material by the boundary (cf.~\cite{Knopf2021a}).
\end{itemize}
\pagebreak[4]
In the \textit{convective case} (i.e., with non-trivial velocity fields $\boldsymbol{v}$ and $\boldsymbol{w}$), the boundary condition \eqref{CH:6} was also used in the Navier--Stokes--Cahn--Hilliard model derived in \cite{Giorgini2023} and in the Cahn--Hilliard--Brinkman model studied in \cite{Colli2023}. Furthermore, the model derivation in \cite{Giorgini2023} shows that the parameter $\beta$ acts as a weight for the mass flux between bulk and surface. This means that $\beta$ can also be negative and therefore, we simply assume $\beta\in\R$.

However, we also want to cover the case, where the materials on the boundary are not the same as those in the bulk. For instance, this is the case if the bulk materials are transformed at the boundary by a chemical reaction. In this case, the surface phase-field might not be proportional to the trace of the bulk phase-field, so the choice $K=0$ would not be appropriate. In our model, this is taken into account by the choice $K\in (0,\infty]$. In the case $K=\infty$, \eqref{CH:6} degenerates to a homogeneous Neumann boundary condition for the phase-field. For most applications, this might not be the preferred choice as it is known that such a Neumann boundary condition enforces the diffuse interface to always intersect the boundary at a perfect angle of ninety degrees. In fact, this is one of the aforementioned limitations we usually want to overcome by the usage of dynamic boundary conditions. However, we include the case $K=\infty$ for the sake of completeness. In the case $K\in (0,\infty)$, \eqref{CH:6} can be regarded as a Robin type boundary condition. It is suitable to describe a scenario, where $\psi$ and the trace of $\phi$ are not proportional, but it still allows for dynamical changes of the contact angle between the diffuse interface and the boundary. In the context of bulk-surface Cahn--Hilliard equations, condition \eqref{CH:6} with $K\in(0,\infty)$ was used in \cite{Knopf2020} (in the non-convective case) and in \cite{Colli2023} (in the convective case). In \cite{Knopf2020} and \cite{Colli2023}, it was further shown that the limit $K\to 0$ can be used to recover the boundary condition \eqref{CH:6} with $K=0$. 
Especially in the case that the materials on the boundary differ from those in the bulk, the parameter $\alpha$ could be any real number (even with a negative sign). Therefore, to keep the model as general as possible, we allow for $\alpha\in\R$. However, for our mathematical analysis, we will need the additional relation $\alpha\beta\abs{\Omega} + \abs{\Gamma} \neq 0$. Of course, this is trivially satisfied if $\alpha$ and $\beta$ have the same sign.

We further want to point out that including the case $K\in (0,\infty)$ also helps our mathematical analysis. This is because the existence of a weak solution to \eqref{CH} can be shown by a suitable Faedo--Galerkin scheme only in the cases, where the boundary conditions are of the same type (i.e., $K=L=0$, $K=L=\infty$ or $0<K,L<\infty$), as otherwise, the spaces of admissible test functions in the weak formulation do not match. In this paper, we will first construct a weak solution of \eqref{CH} in the case $0<K,L<\infty$. Then, we prove the existence of a weak solution in the remaining cases by sending the parameters $K$ and $L$ to zero or to infinity, respectively.

Let us now discuss some important properties of our model.
The system \eqref{CH} is associated with the total energy
\begin{align}
    \label{INTRO:ENERGY}
    \begin{split}
        E_K(\phi,\psi) &= \intO \frac{\epsilon}{2}\abs{\Grad\phi}^2 + \epsilon^{-1}F(\phi) \dx + \intG \frac{\epsilon_{\Gamma}\kappa}{2}\abs{\Gradg\psi}^2 + \epsilon_{\Gamma}^{-1}G(\psi) \dG \\
        &\quad + \h(K)\intG\frac{1}{2}\abs{\alpha\psi - \phi}^2\dG,
    \end{split}
\end{align}
where the function
\begin{align*}
    \h(K) \coloneqq
    \begin{cases}
        K^{-1}, &\text{if } K\in (0,\infty), \\
        0, &\text{if } K\in\{0,\infty\}
    \end{cases}
\end{align*}
is used to distinguish the different cases corresponding to the choice of $K$. Sufficiently regular solutions of the system \eqref{CH} satisfy the \textit{mass conversation law}
\begin{align}
    \label{INTRO:MASS}
    \begin{dcases}
        \beta\intO \phi(t)\dx + \intG \psi(t)\dG = \beta\intO \phi_0 \dx + \intG \psi_0\dG, &\textnormal{if } L\in[0,\infty), \\
        \intO\phi(t)\dx = \intO\phi_0\dx \quad\textnormal{and}\quad \intG\psi(t)\dG = \intG\psi_0\dG, &\textnormal{if } L = \infty
    \end{dcases}
\end{align}
for all $t\in[0,T]$.
As mentioned above, this means that a transfer of material between bulk and surface is allowed only in the cases $L\in [0,\infty)$. 
Moreover, sufficiently regular solutions satisfy the \textit{energy identity}
\begin{align}
    \label{INTRO:ENERGY:ID}
    \begin{split}
        \ddt E_K(\phi,\psi)  &= \intO\phi\boldsymbol{v}\cdot\Grad\mu\dx + \intG \psi\boldsymbol{w}\cdot\Gradg\theta\dG 
        - \intO m_\Om(\phi)\abs{\Grad\mu}^2\dx\\
        &\qquad
         - \intG m_\Ga(\psi)\abs{\Gradg\theta}^2\dG
        - \h(L) \intG (\beta\theta-\mu)^2\dG \\
    \end{split}
\end{align}
for all $t\in[0,T]$. We point out that in the non-convective case (i.e., $\boldsymbol{v}=0$ and $\boldsymbol{w}=0$), the right-hand side of \eqref{INTRO:ENERGY:ID} is clearly non-positive. This means that the energy dissipates over the course of time and the terms on the right-hand side of \eqref{INTRO:ENERGY:ID} can be interpreted as the dissipation rate. In this case, the system \eqref{CH} can be derived as an $H^{-1}$ type gradient flow of the free energy $E_K$ subject to the inner product $\scp{\cdot}{\cdot}_{L,\beta,*}$ that will be introduced in Section~\ref{SECT:PRELIM}, \ref{PRELIM:bulk-surface-elliptic} (cf.~\cite[Remark~2.2]{Knopf2021a} and \cite[Section~3]{Knopf2020}). Moreover, if $L=\infty$, the system \eqref{CH} can also be derived by an energetic variational approach that combines the least action principle and Onsager’s principle of maximum energy dissipation (cf.~\cite[Section~2]{Liu2019} and \cite[Appendix]{Knopf2020}).

In the convective case (i.e., $\boldsymbol{v}$ and $\boldsymbol{w}$ are non-trivial), the energy identity \eqref{INTRO:ENERGY:ID} does, in general, not imply dissipation of the energy. However, if the velocity field is not just prescribed but determined by a Navier--Stokes equation (cf.~\cite{Giorgini2023}) or a Brinkman/Stokes equation (cf.~\cite{Colli2023}), an energy dissipation law for the corresponding total energy can be obtained.

\paragraph{A brief overview of related literature.}
In the non-convective case, we refer to \cite{Garcke2020,Garcke2022,Colli2020,Colli2022,Colli2022a,Fukao2021,Miranville2020} for analytical results and to \cite{Metzger2021,Metzger2023,Harder2022,Meng2023,Bao2021,Bao2021a} for numerical results on the Cahn--Hilliard equation with a dynamic boundary condition of Cahn--Hilliard type. The nonlocal Cahn--Hilliard equation with a nonlocal dynamic boundary condition of Cahn--Hilliard type was proposed and analyzed in \cite{Knopf2021b}. Results on the Cahn--Hilliard equation with dynamic boundary conditions of second-order (e.g., of Allen--Cahn or Wentzell type) can be found, for instance, in \cite{Colli2015,Colli2014,Miranville2010,Racke2003,Wu2004,Gal2009,Gilardi2010,Cavaterra2011}.

The convective Cahn--Hilliard equation with dynamic boundary conditions was analyzed, for instance, in \cite{Colli2018,Colli2019,Gilardi2019}, and also in \cite{Gal2016,Gal2019,Giorgini2023,Gal2023a} as part of a system in which the velocity field is described by an additional fluid equation.

For more information on the Cahn--Hilliard equation with classical homogeneous Neumann boundary conditions or dynamic boundary conditions, we refer to the recent review paper \cite{Wu2022} as well as the book \cite{Miranville-Book}.

\paragraph{Structure of this paper.} In Section~\ref{SECT:PRELIM} we collect some notation, assumptions, preliminaries and important tools. After introducing the notion of weak solutions of \eqref{CH}, our main results are stated in Section~\ref{SECT:MAINRESULTS}. In Section~\ref{SECT:EXISTENCE}, we construct a weak solution in the case $(K,L)\in(0,\infty)^2$ via a Faedo--Galerkin approach. Afterwards, in Section~\ref{SECT:ASYMPTLIM}, we investigate the asymptotic limits $K\rightarrow 0$, $K\rightarrow\infty$, $L\rightarrow 0$ and $L\rightarrow\infty$, which prove the existence of weak solutions of \eqref{CH} in the limit cases. In Section~\ref{SECT:REG+CD}, under suitable additional assumptions,  we establish higher spatial regularity for the phase-fields in the context of weak solutions. In the case of constant mobilities, we eventually prove the continuous dependence of weak solutions on the velocity fields and the initial data. In particular, this continuous dependence result directly entails the uniqueness of the weak solution. 

\section{Notation, assumptions, and preliminaries} 
\label{SECT:PRELIM}

In this section, we introduce some notation, assumptions, and preliminaries that are supposed to hold throughout the remainder of this paper.

\subsection{Notation}
Let us first introduce some basic notation.

\begin{enumerate}[label=\textnormal{\bfseries(N\arabic*)}]
    \item $\N$ denotes the set of natural numbers excluding zero, whereas $\N_0 = \N\cup\{0\}$.
    
    \item Let $\Omega \subset \R^d$ with $d\in\{2,3\}$ be a bounded Lipschitz domain in $\R^d$, and let $\Gamma \coloneqq\del\Omega$ denote its boundary. For any $s\geq 0$ and $p\in[1,\infty]$, the Lebesgue and Sobolev spaces for functions mapping from $\Om$ to $\R$ are denoted as $L^p(\Om)$ and $W^{s,p}(\Om)$. We write $\norm{\cdot}_{L^p(\Om)}$ and $\norm{\cdot}_{W^{s,p}(\Om)}$ to denote the standard norms on these spaces. In the case $p=2$, we use the notation $H^s(\Om) = W^{s,2}(\Om)$. In particular, $H^0(\Om)$ can be identified with $L^2(\Om)$. The Lebesgue and Sobolev spaces on $\Ga$ are denoted by $L^p(\Ga)$ and $W^{s,p}(\Ga)$ along with the corresponding norms $\norm{\cdot}_{L^p(\Ga)}$ and $\norm{\cdot}_{W^{s,p}(\Ga)}$, respectively. 
    For vector-valued functions mapping from $\Om$ into $\R^d$, we use the notation $\mathbf{L}^p(\Om)$, $\mathbf{W}^{s,p}(\Om)$ and $\mathbf{H}^s(\Om)$. The spaces $\mathbf{L}^p(\Ga)$, $\mathbf{W}^{s,p}(\Ga)$ and $\mathbf{H}^s(\Ga)$ are defined analogously.
    For any real numbers $s\geq 0$ and $p\in[1,\infty]$ and any Banach space $X$, the Bochner spaces of functions mapping from an interval $I$ into $X$ are denoted by $L^p(I;X)$ and $W^{s,p}(I;X)$.
    Furthermore, for any interval $I$ and any Banach space $X$, the space $C(I;X)$ denotes the set of continuous functions mapping from $I$ to $X$.    
    
    \item For any Banach space $X$, its dual space is denoted by $X'$. The corresponding duality pairing of elements $\phi\in X'$ and $\zeta\in X$ is denoted by $\ang{\phi}{\zeta}_X$. If $X$ is a Hilbert space, we write $\scp{\cdot}{\cdot}_X$ to denote its inner product. 
    
    \item For any bounded domain $\Om\subset \R^d$ ($d\in\N$) with Lipschitz boundary $\Ga$, $u\in H^1(\Om)'$ and $v\in H^1(\Ga)'$, we write
    \begin{align*}
        \meano{u}\coloneqq 
        \frac{1}{\abs{\Om}}\ang{u}{1}_{H^1(\Om)},
        \qquad
        \meang{v}\coloneqq 
        \frac{1}{\abs{\Ga}}\ang{v}{1}_{H^1(\Ga)}
    \end{align*}
    to denote the generalized means of $u$ and $v$, respectively. Here, $\abs{\Om}$ denotes the $d$-dimensional Lebesgue measure of $\Omega$, whereas $\abs{\Ga}$ denotes the $(d-1)$-dimensional Hausdorff measure of $\Gamma$. If $u\in L^1(\Om)$ or $v\in L^1(\Ga)$, the generalized mean can be expressed as
    \begin{align*}
        \meano{u} = \frac{1}{\abs{\Om}} \intO u \dx,
        \qquad
        \meang{v} = \frac{1}{\abs{\Ga}} \intG v \dG,
    \end{align*}
    respectively.

    \item For any bounded domain $\Om\subset \R^d$ ($d\in\N$) with Lipschitz boundary $\Gamma:=\partial\Omega$, we introduce the spaces
    \begin{align*}
        \mathbf{L}^3_\Div(\Om) &\coloneqq \big\{\boldsymbol{v}\in\mathbf{L}^3(\Om) : \Div\;\boldsymbol{v} = 0 \ \text{in~} \Om, \ \boldsymbol{v}\cdot\mathbf{n} = 0 \ \text{on~} \Ga \big\}, \\
        \mathbf{L}^3_\tau(\Ga) &\coloneqq \big\{\boldsymbol{w}\in\mathbf{L}^3_\tau(\Ga) :  \boldsymbol{w}\cdot\mathbf{n} = 0 \ \text{on~} \Ga \big\}.
    \end{align*}
    We point out that in the definition of $\mathbf{L}^3_\Div(\Om)$, the relation $\Div\;\boldsymbol{v} = 0$ in $\Om$ has to be understood in the sense of distributions. This already implies that $\boldsymbol{v}\cdot\mathbf{n} \in H^{-1/2}(\Ga)$, and therefore, the relation $\boldsymbol{v}\cdot\mathbf{n} = 0$ on $\Ga$ is well-defined.
\end{enumerate}

\subsection{Assumptions}
We make the following general assumptions.

\begin{enumerate}[label=\textnormal{\bfseries(A\arabic*)}]
    \item  \label{ASSUMP:1} $\Omega$ is a non-empty, bounded Lipschitz domain in $\R^d$ with $d\in\{2,3\}$, whose boundary is denoted by $\Ga\coloneqq\del\Om$. Moreover, $T>0$ denotes an arbitrary final time and for brevity, we use the notation
    \begin{align*}
        Q\coloneqq \Om\times(0,T), \quad\Sigma\coloneqq\Ga\times(0,T).
    \end{align*}
    
    \item \label{ASSUMP:2} The constants in system \eqref{CH} satisfy $\epsilon, \epsilon_{\Gamma}, \kappa > 0$, and $\alpha, \beta\in \R$ with $\alpha\beta\abs{\Omega} + \abs{\Gamma} \neq 0$. (The latter condition is required to apply a certain bulk-surface Poincar\'e inequality, see \ref{PRELIM:POINCINEQ}.) Since the choice of $\epsilon_{\Gamma}, \epsilon$ and $\kappa$ has no impact on the mathematical analysis, we will simply set (without loss of generality) $\epsilon = \epsilon_{\Gamma} = \kappa = 1$ in the remainder of this paper.
    
    \item \label{ASSUMP:MOBILITY} The mobility functions $m_\Om:\R\rightarrow\R$ and $m_\Ga:\R\rightarrow\R$ are bounded, continuous and uniformly positive. This means that there exist constants $m_\Om^\ast, M_\Om^\ast, m_\Ga^\ast, M_\Ga^\ast>0$ such that
    \begin{align*}
        0 < m_\Om^\ast \leq m_\Om(s) \leq M_\Om^\ast \quad\text{and}\quad 0 < m_\Ga^\ast \leq m_\Ga(s) \leq M_\Ga^\ast \quad\text{for all $s\in\R$}.
    \end{align*}
    
    \item \label{ASSUMP:POTENTIALS:1} The potentials $F:\R\rightarrow [0,\infty)$ and $G:\R\rightarrow[0,\infty)$ are continuously differentiable and there exist exponents $p$ and $q$ satisfying
    \begin{align*}
        p\in\begin{cases}
            [2,\infty), &\text{if } d=2, \\
            [2,6], &\text{if } d=3,
        \end{cases}
        \quad\text{and}\quad q\in[2,\infty)
    \end{align*}
    as well as constants $c_{F^{\prime}},c_{G^{\prime}}\geq 0$ such that the first-order derivatives satisfy the growth conditions
    \begin{align}
        \abs{F^{\prime}(s)} &\leq c_{F^{\prime}}(1+\abs{s}^{p-1}), \label{GC:F'}\\
        \abs{G^{\prime}(s)} &\leq c_{G^{\prime}}(1+\abs{s}^{q-1}) \label{GC:G'}
    \end{align}
    for all $s\in\R$. 
    These assumptions already imply the existence of constants $c_F,c_G\geq 0$ such that $F$ and $G$ satisfy the growth conditions
    \begin{align}
        F(s) &\leq c_{F}(1+\abs{s}^p), \label{GC:F}\\
        G(s) &\leq c_{G}(1+\abs{s}^q) \label{GC:G}
    \end{align}
    for all $s\in\R$.
    \item \label{ASSUMP:POTENTIALS:2} There exist constants $a_F,a_G > 0$ and $b_F, b_G \geq 0$ such that the potentials $F$ and $G$ introduced in \ref{ASSUMP:POTENTIALS:1} satisfy 
    \begin{align}
        F(s) &\geq a_F\abs{s}^2 - b_F, \\
        G(s) &\geq a_G\abs{s}^2 - b_G
    \end{align}
    for all $s\in\R$.
    \item \label{ASSUMP:POTENTIALS:3} The potentials $F$ and $G$ introduced in \ref{ASSUMP:POTENTIALS:1} are twice continuously differentiable, and there exist constants $c_{F^{\prime\prime}}, c_{G^{\prime\prime}} \geq 0$ such that the second-order derivatives of $F$ and $G$ satisfy the growth conditions
    \begin{align}
        \abs{F^{\prime\prime}(s)} &\leq c_{F^{\prime\prime}}(1+\abs{s}^{p-2}), \label{GC:F''}\\
        \abs{G^{\prime\prime}(s)} &\leq c_{G^{\prime\prime}}(1+\abs{s}^{q-2}) \label{GC:G''}
    \end{align}
    for all $s\in\R$. 
    We point out that these assumptions already imply the growth conditions \eqref{GC:F'}--\eqref{GC:G} stated in \ref{ASSUMP:POTENTIALS:1}.
\end{enumerate}

\medskip

\begin{remark}
    A standard choice for $F$ and $G$ is the polynomial double-well potential
    \begin{equation*}
        W:\R\to\R,\quad W(s) = \tfrac 14 (s^2-1)^2.
    \end{equation*}
    It satisfies the assumptions \ref{ASSUMP:POTENTIALS:1}--\ref{ASSUMP:POTENTIALS:3} with $p=q=4$. However, singular potentials such as the logarithmic Flory--Huggins potential or the double-obstacle potential are not admissible as they do not satisfy any polynomial growth condition. 
    Nevertheless, the investigation of system \eqref{CH} with singular potentials is an interesting topic for future research.
\end{remark}

\subsection{Preliminaries}

\begin{enumerate}[label=\textnormal{\bfseries(P\arabic*)}]
    \item For any real numbers $s\geq 0$ and $p\in[1,\infty]$, we set
    \begin{align*}
        \mathcal{L}^p \coloneqq L^p(\Om)\times L^p(\Ga), \quad\text{and}\quad \mathcal{H}^s\coloneqq H^s(\Om)\times H^s(\Ga),
    \end{align*}
    provided that the boundary is sufficiently regular.
    As usual, we identify $\mathcal{L}^2$ with $\mathcal{H}^0$. Note that $\mathcal{H}^s$ is a Hilbert space with respect to the inner product
    \begin{align*}
        \bigscp{\scp{\phi}{\psi}}{\scp{\zeta}{\xi}}_{\mathcal{H}^s} \coloneqq \scp{\phi}{\zeta}_{H^s(\Om)} + \scp{\psi}{\xi}_{H^s(\Ga)} \quad\text{for all } \scp{\phi}{\psi}, \scp{\zeta}{\xi}\in\mathcal{H}^s
    \end{align*}
    and its induced norm $\norm{\cdot}_{\mathcal{H}^s} \coloneqq \scp{\cdot}{\cdot}_{\mathcal{H}^s}^{1/2}$. We recall that the duality pairing can be expressed as
    \begin{align*}
        \ang{\scp{\phi}{\psi}}{\scp{\zeta}{\xi}}_{\mathcal{H}^1} = \scp{\phi}{\zeta}_{L^2(\Om)} + \scp{\psi}{\xi}_{L^2(\Ga)}
    \end{align*}
    if $\scp{\phi}{\psi}\in \mathcal{L}^2$ and $\scp{\zeta}{\xi}\in \mathcal{H}^1$.
    
    \item\label{PRE:HKA} Let $L\in[0,\infty]$ and $\beta\in\R$. We introduce the closed linear subspace
    \begin{align*}
        \mathcal{D}_\beta \coloneqq \{(\phi,\psi)\in\mathcal{H}^1 : \phi = \beta\psi \text{ a.e.~on } \Ga\} \subset\mathcal{H}^1
    \end{align*}
    and define
    \begin{align*}
        \mathcal{H}_{L,\beta}^1 \coloneqq
        \begin{cases}
            \mathcal{H}^1, &\text{if } L \in (0,\infty] , \\
            \mathcal{D}_\beta, &\text{if } L=0.
        \end{cases}
    \end{align*}
    Endowed with the inner product $\scp{\cdot}{\cdot}_{\mathcal{H}_{L,\beta}^1} \coloneqq \scp{\cdot}{\cdot}_{\mathcal{H}^1}$ and its induced norm, the space $\mathcal{H}_{L,\beta}^1$ is a Hilbert space. Moreover, we define the product
    \begin{align*}
        \ang{\scp{\phi}{\psi}}{\scp{\zeta}{\xi}}_{\mathcal{H}_{L,\beta}^1} \coloneqq \scp{\phi}{\zeta}_{L^2(\Om)} + \scp{\psi}{\xi}_{L^2(\Ga)}
    \end{align*}
    for all $\scp{\phi}{\psi}, \scp{\zeta}{\xi}\in\mathcal{L}^2$. By means of the Riesz representation theorem, this product can be extended to a duality pairing on $(\mathcal{H}_{L,\beta}^1)^\prime\times\mathcal{H}_{L,\beta}^1$, which will also be denoted as $\ang{\cdot}{\cdot}_{\mathcal{H}_{L,\beta}^1}$.
    
    \item Let $L\in[0,\infty]$ and $\beta\in\R$. We define the closed linear subspace
    \begin{align*}
        \mathcal{V}_{L,\beta}^1 &\coloneqq \begin{cases} 
        \{\scp{\phi}{\psi}\in\mathcal{H}^1_{L,\beta} : \beta\abs{\Om}\meano{\phi} + \abs{\Ga}\meang{\psi} = 0 \}, &\text{if~} L\in[0,\infty), \\
        \{\scp{\phi}{\psi}\in\mathcal{H}^1: \meano{\phi} = \meang{\psi} = 0 \}, &\text{if~}L=\infty.
        \end{cases} 
    \end{align*}
    Note that $\mathcal{V}_{L,\beta}^1$ is a Hilbert space with respect to the inner product $\scp{\cdot}{\cdot}_{\mathcal{H}^1}$ and its induced norm.
    
    \item Let $L\in[0,\infty]$ and $\beta\in\R$. We set
    \begin{align*}
        \h(L) \coloneqq
        \begin{cases}
            L^{-1}, &\text{if } L\in(0,\infty), \\
            0, &\text{if } L\in\{0,\infty\},
        \end{cases}
    \end{align*}
    and we define a bilinear form on $\mathcal{H}^1\times\mathcal{H}^1$ by
    \begin{align*}
         \bigscp{\scp{\phi}{\psi}}{\scp{\zeta}{\xi}}_{L,\beta} \coloneqq &\intO\Grad\phi\cdot\Grad\zeta \dx + \intG\Gradg\psi\cdot\Gradg\xi \dG \\  
         &\quad + \h(L)\intG (\beta\psi-\phi)(\beta\xi-\zeta)\dG
    \end{align*}
    for all $ \scp{\phi}{\psi}, \scp{\zeta}{\xi}\in\mathcal{H}^1$. Moreover, we set 
    \begin{align*}
        \norm{\scp{\psi}{\phi}}_{L,\beta} \coloneqq \bigscp{\scp{\phi}{\psi}}{\scp{\phi}{\psi}}_{L,\beta}^{1/2}
    \end{align*}
    for all $\scp{\phi}{\psi}\in\mathcal{H}^1$. The bilinear form $\scp{\cdot}{\cdot}_{L,\beta}$ defines an inner product on $\mathcal{V}^1_{L,\beta}$, and $\norm{\cdot}_{L,\beta}$ defines a norm on $\mathcal{V}^1_{L,\beta}$, that is equivalent to the norm $\norm{\cdot}_{\mathcal{H}^1}$ (see \cite[Corol\-lary~A.2]{Knopf2021}). Moreover, the space $\big(\mathcal{V}^1_{L,\beta}, \scp{\cdot}{\cdot}_{L,\beta}, \norm{\cdot}_{L,\beta}\big)$ is a Hilbert space.
    
    \item \label{PRELIM:bulk-surface-elliptic} For any $L\in[0,\infty]$ and $\beta\in\R$, we define the space
    \begin{align*}
        \mathcal{V}_{L,\beta}^{-1} \coloneqq \begin{cases} 
        \{\scp{\phi}{\psi}\in(\mathcal{H}^1_{L,\beta})^\prime : \beta\abs{\Om}\meano{\phi} + \abs{\Ga}\meang{\psi} = 0 \}, &\text{if~} L\in[0,\infty), \\
        \{\scp{\phi}{\psi}\in(\mathcal{H}^1)^\prime: \meano{\phi} = \meang{\psi} = 0 \}, &\text{if~}L=\infty.
        \end{cases}
    \end{align*}
    Using the Lax-Milgram theorem, one can show that for any $\scp{\phi}{\psi}\in\mathcal{V}_{L,\beta}^{-1}$, there exists a unique weak solution $\mathcal{S}_{L,\beta}(\phi,\psi) = \bigscp{\mathcal{S}_{L,\beta}^\Om(\phi,\psi)}{\mathcal{S}_{L,\beta}^\Ga(\phi,\psi)}\in\mathcal{V}^1_{L,\beta}$ to the following elliptic problem with bulk-surface coupling:
    \pagebreak[2]
    \begin{subequations}\label{SYSTEM:ELLIPTIC}
        \begin{alignat}{3}
            -\Lap\mathcal{S}_{L,\beta}^\Om &= - \phi \qquad &&\text{in } \Om, \\
            -\Lapg\mathcal{S}_{L,\beta}^\Ga + \beta\deln\mathcal{S}_{L,\beta}^\Om &= -\psi &&\text{on } \Ga, \\
            L\deln\mathcal{S}_{L,\beta}^\Om &= \beta\mathcal{S}_{L,\beta}^\Ga - \mathcal{S}_{L,\beta}^\Om \qquad&&\text{on } \Ga.
        \end{alignat}
    \end{subequations}
    This means that $\mathcal{S}_{L,\beta}(\phi,\psi)$ satisfies the weak formulation
    \begin{align*}
        \bigscp{S_{L,\beta}(\phi,\psi)}{\scp{\zeta}{\xi}}_{L,\beta} = - \ang{\scp{\phi}{\psi}}{\scp{\zeta}{\xi}}_{\mathcal{H}^1_{L,\beta}}
    \end{align*}
    for all test functions $\scp{\xi}{\zeta}\in\mathcal{H}^1_{L,\beta}$. Consequently, we have
    \begin{align}\label{EST:solution-operator}
        \norm{\mathcal{S}_{L,\beta}(\phi,\psi)}_{\mathcal{H}^1} \leq C\norm{\scp{\phi}{\psi}}_{(\mathcal{H}^1_{L,\beta})^\prime}
    \end{align}
    for all $\scp{\phi}{\psi}\in\mathcal{V}^{-1}_{L,\beta}$, for a constant $C\geq 0$ depending only on $\Omega$, $L$ and $\beta$. We can thus define the solution operator
    \begin{align*}
        \mathcal{S}_{L,\beta}:\mathcal{V}_{L,\beta}^{-1}\rightarrow\mathcal{V}^1_{L,\beta}, \quad \scp{\phi}{\psi}\mapsto\mathcal{S}_{L,\beta}(\phi,\psi) = \bigscp{\mathcal{S}_{L,\beta}^\Om(\phi,\psi)}{\mathcal{S}_{L,\beta}^\Ga(\phi,\psi)},
    \end{align*}
    as well as an inner product and its induced norm on $\mathcal{V}^{-1}_{L,\beta}$ by
    \begin{align*}
        \bigscp{\scp{\phi}{\psi}}{\scp{\zeta}{\xi}}_{L,\beta,\ast} &\coloneqq \bigscp{\mathcal{S}_{L,\beta}(\phi,\psi)}{\mathcal{S}_{L,\beta}(\zeta,\xi)}_{L,\beta}, \\
        \norm{\scp{\phi}{\psi}}_{L,\beta,\ast} &\coloneqq \bigscp{\scp{\phi}{\psi}}{\scp{\phi}{\psi}}_{L,\beta,\ast}^{1/2},
    \end{align*}
    for $\scp{\phi}{\psi}, \scp{\zeta}{\xi}\in\mathcal{V}_{L,\beta}^{-1}$. This norm is equivalent to the norm $\norm{\cdot}_{(\mathcal{H}^1_{L,\beta})^\prime}$ on $\mathcal{V}_{L,\beta}^{-1}$. For the case $L\in (0,\infty)$, we refer the reader to \cite[Theorem 3.3 and Corollary 3.5]{Knopf2021} for a proof of these statements. In the other cases, the results can be proven analogously.

    \item \label{PRELIM:POINCINEQ} We further recall the following \textit{bulk-surface Poincar\'{e} inequalitiy}, which has been established in \cite[Lemma A.1]{Knopf2021}: \\[0.3em]
    Let $K\in[0,\infty)$ and $\alpha,\beta\in\mathbb{R}$ with $\alpha\beta\abs{\Om} + \abs{\Ga} \neq 0$ be arbitrary. Then there exists a constant $C_P >0$ depending only on $K, \alpha, \beta$ and $\Omega$ such that
    \begin{align}\label{EQ:BSPI1}
        \norm{\scp{\phi}{\psi}}_{\mathcal{L}^2} \leq C_P \norm{\scp{\phi}{\psi}}_{K,\alpha}
    \end{align}
    for all $\scp{\phi}{\psi}\in\mathcal{H}^1_{K,\alpha}$ satisfying $\beta\abs{\Om}\meano{\phi} + \abs{\Ga}\meang{\psi} = 0$. 
\end{enumerate}

\medskip

We conclude this section by presenting a simple inequality that will be frequently used in our mathematical analysis.
\begin{lemma}\label{LEMMA:DUALITY}
    Let $K\in[0,\infty]$ and $\alpha\in\R$. Then, it holds
    \begin{align*}
        \norm{\scp{\zeta}{\xi}}^2_{\mathcal{L}^2} \leq 2 \norm{\scp{\zeta}{\xi}}_{(\mathcal{H}^1_{K,\alpha})^\prime}\norm{\scp{\Grad\zeta}{\Gradg\xi}}_{\mathcal{L}^2} + \norm{\scp{\zeta}{\xi}}_{(\mathcal{H}^1_{K,\alpha})^\prime}^2.
    \end{align*}
    for all $\scp{\zeta}{\xi}\in\mathcal{H}^1_{K,\alpha}$.
\end{lemma}

\begin{proof}
    Let $(\zeta,\xi) \in \mathcal{H}^1_{K,\alpha}$ be arbitrary. Recalling the definition of $\mathcal{H}^1_{K,\alpha}$ (see \ref{PRE:HKA}), we have
    \begin{equation*}
        \norm{(\zeta,\xi)}_{\mathcal{H}^1_{K,\alpha}} = \norm{(\zeta,\xi)}_{\mathcal{H}^1}\le \norm{(\zeta,\xi)}_{\mathcal{L}^2} + \norm{(\Grad\zeta,\Gradg\xi)}_{\mathcal{L}^2}.
    \end{equation*}
    Using Young's inequality, we deduce
    \begin{align*}
        \norm{\scp{\zeta}{\xi}}_{\mathcal{L}^2}^2 & = \ang{\scp{\zeta}{\xi}}{\scp{\zeta}{\xi}}_{\mathcal{H}^1_{K,\alpha}} \\
        &\leq \norm{\scp{\zeta}{\xi}}_{(\mathcal{H}^1_{K,\alpha})^\prime}\norm{\scp{\zeta}{\xi}}_{\mathcal{H}^1_{K,\alpha}} \\
        &\leq  \norm{\scp{\zeta}{\xi}}_{(\mathcal{H}^1_{K,\alpha})^\prime}\norm{\scp{\zeta}{\xi}}_{\mathcal{L}^2} + \norm{\scp{\zeta}{\xi}}_{(\mathcal{H}^1_{K,\alpha})^\prime}\norm{\scp{\Grad\zeta}{\Gradg\xi}}_{\mathcal{L}^2} \\
        &\leq \tfrac{1}{2} \norm{\scp{\zeta}{\xi}}_{\mathcal{L}^2}^2 + \tfrac{1}{2}\norm{\scp{\zeta}{\xi}}_{(\mathcal{H}^1_{K,\alpha})^\prime}^2 + \norm{\scp{\zeta}{\xi}}_{(\mathcal{H}^1_{K,\alpha})^\prime}\norm{\scp{\Grad\zeta}{\Gradg\xi}}_{\mathcal{L}^2}.
    \end{align*}
    Hence, the claim directly follows.
\end{proof}

\section{Main results}
\label{SECT:MAINRESULTS}

As mentioned in \ref{ASSUMP:2}, we set $\epsilon = \epsilon_{\Gamma} = \kappa = 1$. This does not mean any loss of generality as the exact values of $\epsilon$, $\epsilon_{\Gamma}$ and $\kappa$ do not have any impact on the mathematical analysis (as long as they are positive). By this choice, the system \eqref{CH} can be restated as follows:
\begin{subequations}\label{EQ:SYSTEM}
    \begin{align}
        \label{EQ:SYSTEM:1}
        &\delt\phi + \Div(\phi\boldsymbol{v}) = \Div(m_\Om(\phi)\Grad\mu) && \text{in} \ Q, \\
        \label{EQ:SYSTEM:2}
        &\mu = -\Lap\phi + F'(\phi)   && \text{in} \ Q, \\
        \label{EQ:SYSTEM:3}
        &\delt\psi + \Divg(\psi\boldsymbol{w}) = \Divg(m_\Ga(\psi)\Gradg\theta) - \beta m_\Om(\phi)\deln\mu && \text{on} \ \Sigma, \\
        \label{EQ:SYSTEM:4}
        &\theta = - \Lapg\psi + G'(\psi) + \alpha\deln\phi && \text{on} \ \Sigma, \\
        \label{EQ:SYSTEM:5}
        &\begin{cases} K\deln\phi = \alpha\psi - \phi &\text{if} \ K\in [0,\infty), \\
        \deln\phi = 0 &\text{if} \ K = \infty
        \end{cases} && \text{on} \ \Sigma, \\
        \label{EQ:SYSTEM:6}
        &\begin{cases} 
        L m_\Om(\phi)\deln\mu = \beta\theta - \mu &\text{if} \  L\in[0,\infty), \\
        m_\Om(\phi)\deln\mu = 0 &\text{if} \ L=\infty
        \end{cases} &&\text{on} \ \Sigma, \\
        \label{EQ:SYSTEM:7}
        &\phi\vert_{t=0} = \phi_0 &&\text{in} \ \Om, \\
        \label{EQ:SYSTEM:8}
        &\psi\vert_{t=0} = \psi_0 &&\text{on} \ \Ga.
    \end{align}
\end{subequations}

The total energy associated with this system reads as
\begin{align*}
    \begin{split}
        E_K(\phi,\psi) &= \intO\frac{1}{2} \abs{\Grad\phi}^2 + F(\phi) \dx + \intG\frac{1}{2} \abs{\Gradg\psi}^2 + G(\psi) \dG 
        + \h(K) \intG \frac{1}{2} \abs{\alpha\psi - \phi}^2 \dG.
    \end{split}
\end{align*}

We now introduce the notion of a weak solution to system \eqref{EQ:SYSTEM}.

\begin{definition}[Weak solutions of system \eqref{EQ:SYSTEM} for {$K,L\in [0,\infty]$}]\label{DEF:WS}
Suppose that assumptions \ref{ASSUMP:1}--\ref{ASSUMP:POTENTIALS:1} hold. Let $K,L\in[0,\infty]$, let $(\phi_0,\psi_0)\in\mathcal{H}^1_{K,\alpha}$ be arbitrary initial data and let $\boldsymbol{v}\in L^2(0,T;\mathbf{L}_{\Div}^3(\Om))$ and $\boldsymbol{w}\in L^2(0,T;\mathbf{L}^3_\tau(\Ga))$ be given velocity fields. The quadruplet $(\phi,\psi,\mu,\theta)$ is called a weak solution of the system \eqref{EQ:SYSTEM} on $[0,T]$ if the following properties hold:
\begin{enumerate}[label=\textnormal{(\roman*)}, ref=\thetheorem(\roman*)]
    \item \label{DEF:WS:REG} The functions $\phi$, $\psi$, $\mu$ and $\theta$ have the following regularity:
    \begin{subequations}
        \begin{align}
            &\scp{\phi}{\psi} \in C([0,T];\mathcal{L}^2)\cap H^1(0,T;(\mathcal{H}_{L,\beta}^1)^\prime)\cap L^\infty(0,T;\mathcal{H}_{K,\alpha}^1), \label{REGPP}\\
            &\scp{\mu}{\theta}\in L^2(0,T,\mathcal{H}_{L,\beta}^1) \label{REGMT}.
        \end{align}
    \end{subequations}
    \item \label{DEF:WS:IC} The functions $\phi$ and $\psi$ satisfy the initial conditions
    \begin{align}
        \phi\vert_{t=0} = \phi_0 \quad\text{a.e.~in } \Omega, \quad\text{and} \quad\psi\vert_{t=0} = \psi_0 \quad\text{a.e.~on }\Gamma.
    \end{align}
    \item \label{DEF:WS:WF} The functions $\phi$, $\psi$, $\mu$ and $\theta$ satisfy the weak formulation
    \begin{subequations}\label{WF}
        \begin{align}
            &\ang{\scp{\delt\phi}{\delt\psi}}{\scp{\zeta}{\xi}}_{\mathcal{H}_{L,\beta}^1} - \intO \phi\boldsymbol{v}\cdot\Grad\zeta\dx - \intG \psi\boldsymbol{w}\cdot\Gradg\xi\dG \nonumber \\
            &\quad= - \intO m_\Om(\phi)\Grad\mu\cdot\Grad\zeta\dx - \intG  m_\Ga(\psi)\Gradg\theta\cdot\Gradg\xi\dG \label{WF:PP}\\
            &\qquad - \h(L)\intG(\beta\theta-\mu)(\beta\xi - \zeta)\dG, \nonumber\\
            &\intO \mu\,\eta\dx + \intG\theta\,\vartheta\dG 
            \nonumber\\
            &\quad = \intO\Grad\phi\cdot\Grad\eta + F'(\phi)\eta \dx + \intG\Gradg\psi\cdot\Gradg\vartheta + G'(\psi)\vartheta \dG \label{WF:MT}\\
            &\qquad + \h(K)\intG(\alpha\psi-\phi)(\alpha\vartheta - \eta) \dG, \nonumber
        \end{align}
    \end{subequations}
    a.e.~on $[0,T]$ for all $\scp{\zeta}{\xi}\in\mathcal{H}_{L,\beta}^1, \scp{\eta}{\vartheta}\in\mathcal{H}_{K,\alpha}^1$.
    \item \label{DEF:WS:MCL} The functions $\phi$ and $\psi$ satisfy the mass conservation law
    \begin{align}\label{MCL}
        \begin{dcases}
            \beta\intO \phi(t)\dx + \intG \psi(t)\dG = \beta\intO \phi_0 \dx + \intG \psi_0\dG, &\textnormal{if } L\in[0,\infty), \\
            \intO\phi(t)\dx = \intO\phi_0\dx \quad\textnormal{and}\quad \intG\psi(t)\dG = \intG\psi_0\dG, &\textnormal{if } L = \infty
        \end{dcases}
    \end{align}
    for all $t\in[0,T]$.
    \item \label{DEF:WS:WEDL} The functions $\phi, \psi, \mu$ and $\theta$ satisfy the energy inequality
    \begin{align}\label{WEDL}
        \begin{split}
            &E_K(\phi(t),\psi(t)) + \int_0^t\intO m_\Om(\phi)\abs{\Grad\mu}^2\dxs + \int_0^t\intG m_\Ga(\psi)\abs{\Gradg\theta}^2\dGs \\
            &\quad + \h(L) \int_0^t\intG (\beta\theta-\mu)^2\dGs \\
            &\quad - \int_0^t\intO\phi\boldsymbol{v}\cdot\Grad\mu\dxs - \int_0^t\intG \psi\boldsymbol{w}\cdot\Gradg\theta\dGs \\
            &\leq E_K(\phi_0,\psi_0)
        \end{split}
    \end{align}
    for all $t\in[0,T]$.
\end{enumerate}
\end{definition}

Now, we are ready to formulate the main results of this paper. The first main result provides the existence of a weak solution to system \eqref{EQ:SYSTEM} in all cases $(K,L)\in[0,\infty]^2$.

\begin{theorem}\textnormal{(Existence of weak solutions of \eqref{EQ:SYSTEM})}\label{THEOREM:EOWS}
Suppose that the assumptions \ref{ASSUMP:1}--\ref{ASSUMP:POTENTIALS:1} hold. Let $K,L\in[0,\infty]$, let $\scp{\phi_0}{\psi_0}\in\mathcal{H}^1_{K,\alpha}$ be arbitrary initial data, and let $\boldsymbol{v}\in L^2(0,T;\mathbf{L}_\Div^3(\Om))$ and $\boldsymbol{w}\in L^2(0,T;\mathbf{L}^3_\tau(\Ga))$ be given velocity fields. In the case $K\in \{0,\infty\}$, we further assume that \ref{ASSUMP:POTENTIALS:2} holds. Then, there exists a weak solution $(\phi,\psi,\mu,\theta)$ of the system \eqref{EQ:SYSTEM} in the sense of Definition \ref{DEF:WS}, which has the additional regularity
\begin{align}
    \label{REG:1}
    (\mu,\theta) \in L^4(0,T;\mathcal{L}^2) \quad \text{if } K\in(0,\infty].
\end{align}
Let us now assume that \ref{ASSUMP:POTENTIALS:3} holds, that the domain $\Om$ is of class $C^k$ for $k\in\{2,3\}$, and in the case $d=3$, we further assume that \ref{ASSUMP:POTENTIALS:1} holds with $p\leq 4$. Then, we additionally have
\begin{alignat}{2}
    \label{REG:2}
    &\scp{\phi}{\psi}\in L^4(0,T;\mathcal{H}^2) \quad &&\text{if $K\in(0,\infty]$}, \\
    \label{REG:3}
    &\scp{\phi}{\psi}\in L^2(0,T;\mathcal{H}^k) &&\text{for $k=2,3$}, \\
    \label{REG:4}
    &\scp{\phi}{\psi}\in C([0,T];\mathcal{H}^1) \quad 
    &&\text{if $(K,L)\in[0,\infty]\times (0,\infty]$ and $k=3$}
\end{alignat}
and the equations 
\begin{align*}
    &\mu = -\Lap\phi + F'(\phi) &&\text{a.e.~in } Q, \\
    &\theta = -\Lapg\psi + G'(\psi) + \alpha\deln\phi &&\text{a.e.~on } \Sigma, \\
    & \begin{cases} 
        K\deln\phi = \alpha\psi - \phi &\text{if} \ K\in [0,\infty), \\
        \deln\phi = 0 &\text{if} \ K = \infty
    \end{cases} &&  \text{a.e.~on } \Sigma
\end{align*}
are fulfilled in the strong sense.
\end{theorem}

\begin{remark}
    The continuity property \eqref{REG:4} also holds in the case $K = L = 0$ provided that $\alpha\neq 0$, $\beta\neq 0$ and the potentials $F$ and $G$ satisfy the compatibility condition
    \begin{align*}
        F(\alpha s) = \alpha\beta G(s) \quad\text{for all~} s\in\R.
    \end{align*}
    In that case, we have $\scp{F^\prime(\phi)}{G^\prime(\psi)}\in\mathcal{H}^1_{0,\beta} = \mathcal{D}_\beta$, which allows us to use Proposition~\ref{CR} similarly as in Section~\ref{SUBSEC:REG}.
\end{remark}

\begin{proof}[Proof of Theorem~\ref{THEOREM:EOWS}]
    The existence of a weak solution in the sense of Definition~\ref{DEF:WS} is established in 
    \begin{itemize}[leftmargin=2em, topsep=0em, partopsep=0em, parsep=0em, itemsep=0em]
        \item Theorem~\ref{thm:existence} if $(K,L)\in(0,\infty)\times (0,\infty)$,
        \item Theorem~\ref{THEOREM:K->0} if $(K,L)\in\{0\}\times(0,\infty)$,
        \item Theorem~\ref{THEOREM:K->inf} if $(K,L)\in\{\infty\}\times(0,\infty)$,
        \item Theorem~\ref{THEOREM:L->0} if $(K,L)\in [0,\infty] \times \{0\}$,
        \item Theorem~\ref{THEOREM:L->inf} if $(K,L)\in [0,\infty] \times \{\infty\}$.
    \end{itemize}
    Moreover, in the case $K\in (0,\infty]$, the additional regularity \eqref{REG:1} follows from the corresponding aforementioned theorems. All remaining results are shown in Theorem~\ref{THEOREM:REG}.    
\end{proof}

\begin{remark}
    We point out that Theorem~\ref{THEOREM:K->0}, Theorem~\ref{THEOREM:K->inf}, Theorem~\ref{THEOREM:L->0} and Theorem~\ref{THEOREM:L->inf} are not only useful to prove Theorem~\ref{THEOREM:EOWS}, but also provide further valuable insights about the asymptotic limits $K\to 0$, $K\to\infty$, $L\to 0$ and $L\to \infty$, respectively, on the level of weak solutions.
\end{remark}

\medskip

Our second main result shows continuous dependence of weak solutions on the velocity fields and the initial data in the case of constant mobilities. As a direct consequence, this entails uniqueness of the weak solution.

\begin{theorem}\textnormal{(Continuous dependence and uniqueness)}\label{THEOREM:CD}
    Suppose that the assumptions \ref{ASSUMP:1}--\ref{ASSUMP:POTENTIALS:1} hold, and that the mobility functions $m_\Om$ and $m_\Ga$ are constant. If $d=3$, we additionally assume that \ref{ASSUMP:POTENTIALS:1} holds with $p<6$. Let $K,L\in [0,\infty]$, let $\scp{\phi_0^1}{\psi_0^1}$ and $\scp{\phi_0^2}{\psi_0^2}\in\mathcal{H}^1_{K,\alpha}$ be two pairs of initial data, which satisfy
    \begin{equation}\label{initial-data-mean-value}
        \scp{\phi_0^1-\phi_0^2}{\psi_0^1-\psi_0^2} \in \mathcal{V}^{-1}_{L,\beta}
    \end{equation}
    and let $\boldsymbol{v}_1, \boldsymbol{v}_2\in L^2(0,T;\mathbf{L}_\Div^{3}(\Om))$ and $\boldsymbol{w}_1, \boldsymbol{w}_2\in L^2(0,T;\mathbf{L}^{3}(\Ga))$ be given velocity fields.
    Suppose that $(\phi_1,\psi_1,\mu_1,\theta_1)$ and $(\phi_2,\psi_2,\mu_2,\theta_2)$ are weak solutions in the sense of Definition~\ref{DEF:WS} corresponding to $(\phi_0^1,\psi_0^1,\boldsymbol{v}_1,\boldsymbol{w}_1)$ and $(\phi_0^2,\psi_0^2,\boldsymbol{v}_2,\boldsymbol{w}_2)$, respectively.
    Then, the continuous dependence estimate
    \begin{align}\label{EST:continuous-dependence}
        \begin{split}
            &\bignorm{\bigscp{\phi_1(t)-\phi_2(t)}{\psi_1(t)-\psi_2(t)}}^2_{L,\beta,\ast} 
            \\[0.4em]
            &\leq \bignorm{\bigscp{\phi_0^1-\phi_0^2}{\psi_0^1-\psi_0^2}}^2_{L,\beta,\ast}\exp\left(C\int_0^t\mathcal{F}(\tau)\dtau\right) 
            \\
            &\quad + \int_0^t\bignorm{\bigscp{\boldsymbol{v}_1(s)-\boldsymbol{v}_2(s)}{\boldsymbol{w}_1(s)- \boldsymbol{w}_2(s)}}^2_{\mathcal{L}^{3}}\exp\left(C\int_s^t\mathcal{F}(\tau)\dtau\right)\ds
        \end{split}
    \end{align}
    holds for almost all $t\in[0,T]$, where $\mathcal{F} \coloneqq \norm{\scp{\boldsymbol{v}_1}{\boldsymbol{w}_1}}_{\mathcal{L}^3}^2$
    and the constant $C>0$ depends only on $\Om$, the parameters of the system and the initial data. 

    In particular, if $(\phi_0^1,\psi_0^1) = (\phi_0^2,\psi_0^2)$ a.e.~in $\Omega\times \Gamma$, $\boldsymbol{v}_1 = \boldsymbol{v}_2$ a.e.~in $Q$ and $\boldsymbol{w}_1 = \boldsymbol{w}_2$ a.e.~on $\Sigma$, estimate \eqref{EST:continuous-dependence} ensures uniqueness of the corresponding weak solution.
\end{theorem}

The proof of Theorem~\ref{THEOREM:CD} will be presented in Subsection~\ref{SUBSEC:CD}.

\section{Existence of weak solutions in the case \texorpdfstring{$K,L\in(0,\infty)$}{K,L in (0,infty)}}
\label{SECT:EXISTENCE}
\begin{theorem}\label{thm:existence}
    Suppose that the assumptions \ref{ASSUMP:1}--\ref{ASSUMP:POTENTIALS:1} hold. Let $\scp{\phi_0}{\psi_0}\in\mathcal{H}^1$ be arbitrary initial data, $\boldsymbol{v}\in L^2(0,T;\mathbf{L}_\Div^3(\Om))$, $\boldsymbol{w}\in L^2(0,T;\mathbf{L}^3_\tau(\Ga))$ and $K,L\in(0,\infty)$. Then there exists a weak solution of the system \eqref{EQ:SYSTEM} in the sense of Definition~\ref{DEF:WS} which further satisfies $\scp{\mu}{\theta}\in L^4(0,T;\mathcal{L}^2)$.
\end{theorem}

\begin{proof}
    \textbf{Step 1: Discretization via a Faedo--Galerkin scheme.} 
    It is well known that the problems
    \begin{align*}
        \begin{cases}
            \begin{alignedat}{3}
                -\Lap \zeta &= \lambda_\Om \zeta, \quad&&\text{in }\Om, \\
                \deln\zeta &= 0, \quad&&\text{on }\Ga,
            \end{alignedat}
        \end{cases}
        \quad -\Lapg \xi = \lambda_\Ga \xi \quad\text{on }\Ga,
    \end{align*}
    have countable many eigenvalues, which can be written as increasing sequences $\{\lambda_\Om^j\}_{j\in\N}$ and $\{\lambda_\Ga^j\}_{j\in\N}$, respectively. 
    The associated eigenfunctions $\{\zeta_j\}_{j\in\N}\subset H^1(\Om)$ and $\{\xi_j\}_{j\in\N}\subset H^1(\Ga)$ form an orthonormal Schauder basis of $L^2(\Om)$ and $L^2(\Ga)$, respectively. In particular, we fix the eigenfunctions associated to the first eigenvalues $\lambda_\Om^1 = \lambda_\Ga^1 = 0$ as $\zeta_1 \equiv \abs{\Omega}^{-1/2}$ and $\xi_1 \equiv \abs{\Gamma}^{-1/2}$. Moreover, the eigenfunctions $\{\zeta_j\}_{j\in\N}$ and $\{\xi_j\}_{j\in\N}$ also form an orthogonal Schauder basis of $H^1(\Omega)$ and $H^1(\Gamma)$, respectively. 
    
    For any $m\in\N$, we introduce the finite-dimensional subspaces
    \begin{align*}
        \mathcal{A}_m = \textnormal{span}\{\zeta_1,\ldots,\zeta_m\}\subset H^1(\Om), \quad
        \mathcal{B}_m = \textnormal{span}\{\xi_1,\ldots,\xi_m\}\subset H^1(\Ga),
    \end{align*}
    along with the orthogonal $L^2(\Om)$-projection $\projam$, and the orthogonal $L^2(\Ga)$-projection $\projbm$. In particular, there exist constants $C_{\Omega}, C_{\Gamma}>0$ depending only on $\Omega$ and $\Gamma$, respectively, such that for all $\zeta\in H^1(\Omega)$ and $\xi\in H^1(\Gamma)$,
    \begin{align*}
        \norm{\projam \zeta}_{H^1(\Omega)} \le C_{\Omega} \norm{\zeta}_{H^1(\Omega)}
        \quad\text{and}\quad
        \norm{\projbm \xi}_{H^1(\Gamma)} \le C_{\Gamma} \norm{\xi}_{H^1(\Gamma)}.
    \end{align*}
    For any $m\in\N$, and $t\in[0,T]$, we make the ansatz
    \begin{subequations}\label{galerkin-ansatz}
        \begin{alignat*}{2}
            \phi_m(t)&\coloneqq \sum_{j=1}^m a_j^m(t)\,\zeta_j \quad\text{a.e.~in } \Om, &\qquad
            \psi_m(t)&\coloneqq \sum_{j=1}^m b_j^m(t)\,\xi_j \quad\text{a.e.~on } \Ga, \\
            \mu_m(t)&\coloneqq \sum_{j=1}^m c_j^m(t)\,\zeta_j \quad\text{a.e.~in } \Om, &\qquad
            \theta_m(t)&\coloneqq \sum_{j=1}^m d_j^m(t)\,\xi_j \quad\text{a.e.~on } \Ga.
        \end{alignat*}
        \end{subequations}
    Here, the scalar, time-dependent coefficients $a_j^m, b_j^m, c_j^m, d_j^m$, $j=1,\ldots,m$ are assumed to be continuously differentiable functions that are not yet determined. They need to be designed in a way such that the discretized weak formulation%
    \begin{subequations}\label{DWF}
        \begin{align}
            &\ang{\scp{\delt\phi_m}{\delt\psi_m}}{\scp{\zeta}{\xi}}_{\mathcal{H}^1} - \intO \phi_m\boldsymbol{v}\cdot\Grad\zeta\dx - \intG \psi_m\boldsymbol{w}\cdot\Gradg\xi\dG \nonumber \\
            &\quad= - \intO m_\Om(\phi_m)\Grad\mu_m\cdot\Grad\zeta\dx - \intG  m_\Ga(\psi_m)\Gradg\theta_m\cdot\Gradg\xi\dG \label{DWF:PP}\\
            &\qquad - L^{-1}\intG(\beta\theta_m-\mu_m)(\beta\xi - \zeta)\dG, \nonumber\\[1ex]
            &\intO \mu_m\,\zeta\dx + \intG\theta_m\,\xi\dG 
            \nonumber\\
            &\quad =  \intO\Grad\phi_m\cdot\Grad\zeta + F'(\phi_m)\zeta \dx + \intG\Gradg\psi_m\cdot\Gradg\xi + G'(\psi_m)\xi \dG \label{DWF:MT}\\
            &\qquad + K^{-1}\intG(\alpha\psi_m-\phi_m)(\alpha\xi - \zeta) \dG, \nonumber
        \end{align}
    \end{subequations}
    holds on $[0,T]$ for all $\scp{\zeta}{\xi}\in\mathcal{A}_m\times\mathcal{B}_m$, supplemented with the initial conditions
    \begin{align}\label{DWF:IC}
        \phi_m(0) =  \projam(\phi_0) \quad\text{a.e.~in }\Omega, \qquad \psi_m(0) =  \projbm(\psi_0) \quad\text{a.e.~on } \Ga.
    \end{align}
    Now, consider the corresponding coefficient vectors $\mathbf{a}^m \coloneqq (a^m_1,\ldots,a^m_m), \mathbf{b}^m \coloneqq (b^m_1,\ldots,b^m_m)$, $\mathbf{c}^m \coloneqq (c^m_1,\ldots,c^m_m)$ and $\mathbf{d}^m \coloneqq (d^m_1,\ldots,d^m_m)$. Testing \eqref{DWF:PP} with $\scp{\zeta_1}{\xi_1},\ldots,\scp{\zeta_m}{\xi_m}$, we conclude that $(\mathbf{a}^m,\mathbf{b}^m)^\top$ is determined by a system of $2m$ ordinary differential equations, whose right-hand side depends continuously on $\mathbf{a}^m, \mathbf{b}^m, \mathbf{c}^m$ and $\mathbf{d}^m$. In view of \eqref{DWF:IC}, this system is subject to the initial conditions
    \begin{alignat*}{3}
        [\mathbf{a}^m]_j(0) &= a^m_j(0) = \scp{\phi_0}{\zeta_j}, \quad\quad&&\text{for all } j\in\{1,\ldots,m\}, \\
        [\mathbf{b}^m]_j(0) &= b^m_j(0) = \scp{\psi_0}{\xi_j}, &&\text{for all } j\in\{1,\ldots,m\}.
    \end{alignat*}
    Moreover, testing \eqref{DWF:MT} with $(\zeta_1,\xi_1),\ldots,(\zeta_m,\xi_m)$, we infer that $(\mathbf{c}^m,\mathbf{d}^m)^\top$ is explicitly given by a system of $2m$ algebraic equations, whose right-hand side depends continuously on $\mathbf{a}^m$ and $\mathbf{b}^m$. Thus, replacing $(\mathbf{c}^m,\mathbf{d}^m)^\top$ in the right-hand side of the aforementioned ODE system by this algebraic description, we obtain a closed $2m$-dimensional ODE system for $(\mathbf{a}^m,\mathbf{b}^m)^\top$, whose right-hand side depends continuously on $(\mathbf{a}^m,\mathbf{b}^m)^\top$. We can thus apply the Cauchy--Peano theorem to obtain a local solution $(\mathbf{a}^m,\mathbf{b}^m)^\top:[0,T_m^\ast)\cap[0,T]\rightarrow\R^{2m}$ with $T_m^\ast>0$ to the corresponding initial value problem. Without loss of generality, we assume that $T_m^\ast \leq T$ and that $(\mathbf{a}^m,\mathbf{b}^m)^\top$ is non-extendable, i.e., $T_m^\ast$ is chosen as large as possible. Now, we can reconstruct $(\mathbf{c}^m,\mathbf{d}^m)^\top:[0,T_m^\ast)\rightarrow\R^{2m}$ by the aforementioned $2m$-dimensional system of algebraic equations. In view of \eqref{galerkin-ansatz}, we thus have shown the existence of functions
    \begin{align*}
        \scp{\phi_m}{\psi_m}\in C^1([0,T_m^\ast);\mathcal{H}^1),\quad\quad \scp{\mu_m}{\theta_m}\in C^1([0,T_m^\ast);\mathcal{H}^1)
    \end{align*}
    solving \eqref{DWF} on the time interval $[0,T_m^\ast)$ subject to the initial conditions \eqref{DWF:IC}. \\[1ex]
    \textbf{Step 2: Uniform estimates.} We establish suitable estimates for each approximate solution $(\phi_m,\psi_m,\mu_m,\theta_m)$, which are uniform with respect to the index $m$. In particular, let $T_m < T_m^\ast$ be arbitrary. In the following, let $C$ denote a generic non-negative constant depending only on the initial data and the constants introduced in \ref{ASSUMP:MOBILITY}--\ref{ASSUMP:POTENTIALS:1} including the final time $T$, but not on $m$ or $T_m$.
    
    Testing \eqref{DWF:PP} with $\scp{\mu_m}{\theta_m}$, \eqref{DWF:MT} with $-\scp{\delt\phi_m}{\delt\psi_m}$, adding the resulting equations, and integrating with respect to time from $0$ to $t$, we derive the discrete energy identity
    \begin{align}\label{EQ:DEI}
        \begin{split}
            &E_K(\phi_m(t),\psi_m(t)) + \int_0^t\intO m_\Om(\phi_m)\abs{\Grad\mu_m}^2\dxs + \int_0^t\intG m_\Ga(\psi_m)\abs{\Gradg\theta_m}^2\dGs \\
            &\quad\quad + L^{-1} \int_0^t\intG (\beta\theta_m-\mu_m)^2\dGs \\
            &\quad = E_K(\phi_m(0),\psi_m(0)) + \int_0^t\intO \phi_m\boldsymbol{v}\cdot\Grad\mu_m\dxs + \int_0^t\intG\psi_m\boldsymbol{w}\cdot\Gradg\theta_m\dGs \\
        \end{split}
    \end{align}
    for all $t\in[0,T_m]$. Now, due to the growth conditions on $F, G$ (see \ref{ASSUMP:POTENTIALS:1}), the Sobolev embeddings $H^1(\Om)\emb L^6(\Om)$ and $H^1(\Ga)\emb L^q(\Ga)$, the trace embedding $H^1(\Om)\emb L^2(\Ga)$, and the properties of the projections $\mathbb{P}_{\mathcal{A}_m}$ and $\mathbb{P}_{\mathcal{B}_m}$, we use the initial conditions to infer
    \begin{align}\label{EST:initial-energy}
        \begin{split}
            &E_K(\phi_m(0),\psi_m(0)) \\
            &\leq \frac{1}{2} \norm{\Grad\projam(\phi_0)}^2_{L^2(\Om)} + \frac{1}{2}\norm{\Gradg\projbm(\psi_0)}^2_{L^2(\Ga)} + c_F\left(\abs{\Om} +  \norm{\projam(\phi_0)}^p_{L^p(\Om)}\right) \\ 
            &\quad + c_G\left(\abs{\Ga} +  \norm{\projbm(\psi_0)}^q_{L^q(\Ga)}\right) + \frac{1}{2K}\Big(\abs{\alpha} \norm{\projbm(\psi_0)}_{L^2(\Ga)} + \norm{\projam(\phi_0)}_{L^2(\Ga)}\Big)^2 \\
            &\leq C\norm{\scp{\phi_0}{\psi_0}}_{\mathcal{H}^1}^2 +  C\left(1 +  \norm{\phi_0}^p_{H^1(\Om)}\right) + C\left(1 +  \norm{\psi_0}^q_{H^1(\Ga)}\right).
        \end{split}
    \end{align}    
     To obtain a suitable bound on the energy and the integral terms on the left-hand side of \eqref{EQ:DEI}, we have to deal with the convective terms appearing on the right-hand side. To this end, we start by deriving for all $t\in[0,T_m]$ the estimate
    \begin{align}\label{EST:convective-terms}
            &\int_0^t\intO \phi_m\boldsymbol{v}\cdot\Grad\mu_m \dxs + \int_0^t\intG \psi_m\boldsymbol{w}\cdot\Gradg\mu_m \dGs \notag\\
            &\leq \int_0^t \norm{\phi_m}_{L^6(\Om)}\norm{\boldsymbol{v}}_{L^3(\Om)}\norm{\Grad\mu_m}_{L^2(\Om)} + \norm{\psi_m}_{L^6(\Ga)}\norm{\boldsymbol{w}}_{L^3(\Ga)}\norm{\Gradg\theta_m}_{L^2(\Ga)} \ds 
            \notag\\
            \begin{split}
            &\leq \frac{m_\Om^\ast}{2} \int_0^t \norm{\Grad\mu_m}^2_{L^2(\Om)}\ds + \frac{m_\Ga^\ast}{2} \int_0^t \norm{\Gradg\theta_m}^2_{L^2(\Ga)}\ds \\
            &\quad + C \int_0^t \norm{\boldsymbol{v}}_{L^3(\Om)}^2\norm{\phi_m}_{H^1(\Om)}^2 + \norm{\boldsymbol{w}}_{L^3(\Ga)}^2\norm{\psi_m}_{H^1(\Ga)}^2 \ds 
            \end{split}
            \\
            \begin{split}
            &\leq \frac{m_\Om^\ast}{2} \int_0^t \norm{\Grad\mu_m}^2_{L^2(\Om)}\ds + \frac{m_\Ga^\ast}{2} \int_0^t \norm{\Gradg\theta_m}^2_{L^2(\Ga)}\ds \notag\\
            &\quad + C\int_0^t \left(\norm{\boldsymbol{v}}_{L^3(\Om)}^2 + \norm{\boldsymbol{w}}_{L^3(\Ga)}^2\right) \norm{\scp{\phi_m}{\psi_m}}^2_{\mathcal{H}^1} \ds, \notag
            \end{split}
    \end{align}
    using again the embeddings $H^1(\Om)\emb L^6(\Om)$ and $H^1(\Ga)\emb L^6(\Ga)$. Recalling that $F,G\geq 0$, and that the mobility functions $m_\Om$ and $m_\Ga$ are uniformly positive (see \ref{ASSUMP:MOBILITY}), we infer from \eqref{EST:initial-energy} and \eqref{EST:convective-terms} that
    \begin{align}\label{EST:PPKA}
        \begin{split}
            &\frac12\norm{\scp{\phi_m}{\psi_m}}_{K,\alpha}^2 + \frac{m_\Om^\ast}{2}\int_0^t\intO \abs{\Grad\mu_m}^2\dxs + \frac{m_\Ga^\ast}{2}\int_0^t\intG \abs{\Gradg\theta_m}^2\dGs \\
            &\quad\quad + L^{-1} \int_0^t\intG(\beta\theta_m-\mu_m)^2\dGs \\
            &\quad\leq C + C\int_0^t \left(\norm{\boldsymbol{v}}_{L^3(\Om)}^2 + \norm{\boldsymbol{w}}_{L^3(\Ga)}^2\right) \norm{\scp{\phi_m}{\psi_m}}^2_{\mathcal{H}^1} \ds \\
            &\quad\leq C + C \int_0^t \left(\norm{\boldsymbol{v}}_{L^3(\Om)}^2 + \norm{\boldsymbol{w}}_{L^3(\Ga)}^2\right) \norm{\scp{\phi_m}{\psi_m}}^2_{K,\alpha} \ds
        \end{split}
    \end{align}
    for all $t\in[0,T_m]$. Here, the last inequality follows from the bulk-surface Poincar\'{e} inequality \ref{PRELIM:POINCINEQ}, which usage is justified by the following reasoning. First, testing \eqref{DWF:PP} with $(\beta,1)$ yields the discrete mass conservation law
    \begin{align}\label{EQ:DMCL}
        \beta\intO\phi_m(t)\dx + \intG\psi_m(t)\dG = \beta\intO\phi_0\dx + \intG\psi_0\dG
    \end{align}
    for all $t\in[0,T_m]$, where we used that
    \begin{align*}
        &\beta\abs{\Om}\meano{\projam(\phi_0)} + \abs{\Ga}\meang{\projbm(\psi_0)} = \bigscp{\scp{\projam(\phi_0)}{\projbm(\psi_0)}}{\scp{\beta}{1}}_{\mathcal{L}^2} \\[0.3em]
        &\quad = \bigscp{\scp{\phi_0}{\psi_0}}{\scp{\beta}{1}}_{\mathcal{L}^2} = \beta\abs{\Om}\meano{\phi_0} + \abs{\Ga}\meang{\psi_0}.
    \end{align*}
    Hence, considering
    \begin{align*}
        c \coloneqq \frac{\beta\abs{\Om}\meano{\phi_0} + \abs{\Ga}\meang{\psi_0}}{\beta^2\abs{\Om} + \abs{\Ga}},
    \end{align*}
    and invoking \eqref{EQ:DMCL}, we infer
    \begin{align*}
        \beta\abs{\Om}\meano{\phi_m(t)-\beta c} + \abs{\Ga}\meang{\psi_m(t) - c} = 0 \quad\text{for all } t\in[0,T_m].
    \end{align*}
    This allows us to use the bulk-surface Poincar\'{e} inequality \ref{PRELIM:POINCINEQ}, which yields
    \begin{align}\label{EST:PIPP}
        \begin{split}
            \norm{\scp{\phi_m}{\psi_m}}_{\mathcal{L}^2} &\leq \norm{\scp{\phi_m-\beta c}{\psi_m-c}}_{\mathcal{L}^2} + \norm{\scp{\beta c}{c}}_{\mathcal{L}^2} \\
            &\leq C_P \norm{\scp{\phi_m}{\psi_m}}_{K,\alpha} + C.
            \end{split}
    \end{align}
    Thus, using Gronwall's lemma, we conclude from \eqref{EST:PPKA} that
    \begin{align*}
        \norm{\scp{\phi_m(t)}{\psi_m(t)}}_{K,\alpha}^2 &\leq C \exp\left(C\int_0^t\left(\norm{\boldsymbol{v}}_{L^3(\Om)}^2 + \norm{\boldsymbol{w}}_{L^3(\Ga)}^2\right)\ds\right)\\
        &\leq C\exp\left(C\int_0^T\left(\norm{\boldsymbol{v}}_{L^3(\Om)}^2 + \norm{\boldsymbol{w}}_{L^3(\Ga)}^2\right)\ds\right)
    \end{align*}
    for all $t\in[0,T_m]$. We thus readily deduce that 
    \begin{align}\label{EST:PPFULLH^1}
        \norm{\scp{\phi_m}{\psi_m}}_{L^\infty(0,T_m;\mathcal{H}^1)}^2 + K^{-1}\norm{\alpha\psi_m-\phi_m}^2_{L^\infty(0,T_m;L^2(\Ga))} \leq C.
    \end{align}
    In view of \eqref{EST:PPKA}, this additionally entails
    \begin{align}\label{EST:GRADMT+DIFF}
            \norm{\scp{\Grad\mu_m}{\Gradg\theta_m}}^2_{L^2(0,T_m;\mathcal{L}^2)} + L^{-1} \norm{\beta\theta_m-\mu_m}^2_{L^2(0,T_m;L^2(\Ga))} &\leq C.
    \end{align}
    Next, we derive a uniform estimate for $\scp{\mu_m}{\theta_m}$ in the full $\mathcal{H}^1$ norm. To this end, let $\scp{\zeta^\ast}{\xi^\ast}\in\mathcal{H}^1$ and consider $\scp{\zeta}{\xi}\coloneqq\scp{\projam(\zeta^\ast)}{\projbm(\xi^\ast)}$. By the growth conditions \eqref{GC:F'} and \eqref{GC:G'} from \ref{ASSUMP:POTENTIALS:1}, the Sobolev embeddings $H^1(\Om)\emb L^6(\Om)$ and $H^1(\Ga)\emb L^{2q}(\Ga)$, and the trace theorem $H^1(\Om)\emb L^2(\Ga)$, we have
    \begin{align}\label{COMP:mu-theta-H^1}
        \begin{split}
            \abs{\ang{\scp{\mu_m}{\theta_m}}{\scp{\zeta^\ast}{\xi^\ast}}_{\mathcal{H}^1}} &= \abs{\ang{\scp{\mu_m}{\theta_m}}{\scp{\zeta}{\xi}}_{\mathcal{H}^1}} \\
            &\leq \norm{\Grad\phi_m}_{L^2(\Om)}\norm{\Grad\zeta}_{L^2(\Om)} + \norm{F'(\phi_m)}_{L^{6/5}(\Om)}\norm{\zeta}_{L^6(\Om)} \\
            &\quad + \norm{\Gradg\psi_m}_{L^2(\Ga)}\norm{\Gradg\xi}_{L^2(\Ga)}  + \norm{G'(\psi_m)}_{L^2(\Ga)}\norm{\xi}_{L^2(\Ga)}\\
            &\quad + K^{-1}\norm{\alpha\psi_m-\phi_m}_{L^2(\Ga)}\norm{\alpha\zeta-\xi}_{L^2(\Ga)} \\
            &\leq C\Big(1 + \norm{\phi_m}_{H^1(\Om)} + \norm{\phi_m}_{L^6(\Om)}^5\Big)\norm{\zeta}_{H^1(\Om)} \\
            &\quad + C\Big(1 + \norm{\psi_m}_{H^1(\Ga)} +  \norm{\psi_m}_{L^{2q}(\Ga)}^q\Big)\norm{\xi}_{H^1(\Ga)} \\
            &\leq C\left(1 + \norm{\phi_m}^5_{H^1(\Om)} + \norm{\psi_m}^q_{H^1(\Ga)}\right)\norm{\scp{\zeta^\ast}{\xi^\ast}}_{\mathcal{H}^1}
        \end{split}
    \end{align}
    on $[0,T_m]$. Taking the supremum over all $\scp{\zeta^\ast}{\xi^\ast}\in\mathcal{H}^1$ with $\norm{\scp{\zeta^\ast}{\xi^\ast}}_{\mathcal{H}^1}\leq 1$, and exploiting \eqref{EST:PPFULLH^1}, we conclude that
    \begin{align}\label{EST:mu-theta-dual}
        \norm{\scp{\mu_m}{\theta_m}}_{L^\infty(0,T_m;(\mathcal{H}^1)')} \leq C.
    \end{align}
    Now, due to Lemma~\ref{LEMMA:DUALITY}, we have
    \begin{align}\label{EST:dual}
        \norm{\scp{\mu_m}{\theta_m}}_{\mathcal{L}^2}^2 \leq 2\norm{\scp{\mu_m}{\theta_m}}_{(\mathcal{H}^1)^\prime}\norm{\scp{\Grad\mu_m}{\Gradg\theta_m}}_{\mathcal{L}^2} + \norm{\scp{\mu_m}{\theta_m}}_{(\mathcal{H}^1)^\prime}^2
    \end{align}
    on $[0,T_m]$. Utilizing \eqref{EST:GRADMT+DIFF} and \eqref{EST:mu-theta-dual}, we deduce from \eqref{EST:dual} that
    \begin{align}\label{EST:mu-theta-L^4-L^2}
        \norm{\scp{\mu_m}{\theta_m}}_{L^4(0,T_m;\mathcal{L}^2)} \leq C.
    \end{align}
    In view of \eqref{EST:GRADMT+DIFF} we additionally infer
    \begin{align}\label{EST:full-mu-theta}
        \norm{\scp{\mu_m}{\theta_m}}_{L^2(0,T_m;\mathcal{H}^1)} \leq C.
    \end{align}
    Lastly, we derive a uniform estimate of the time derivatives. Therefore, we consider again $\scp{\zeta}{\xi}\coloneqq \scp{\projam(\zeta^\ast)}{\projbm(\xi^\ast)}$ for an arbitrary $\scp{\zeta^\ast}{\xi^\ast}\in\mathcal{H}^1$. Recalling that the mobility functions $m_\Om$ and $m_\Ga$ are bounded (see \ref{ASSUMP:MOBILITY}), we use Hölder's inequality as well as the continuous embeddings $H^1(\Om)\emb L^6(\Om)$ and $H^1(\Ga)\emb L^6(\Ga)$ to infer that
    \begin{align*}
        \begin{split}
            \abs{\ang{\scp{\delt\phi_m}{\delt\psi_m}}{\scp{\zeta^\ast}{\xi^\ast}}_{\mathcal{H}^1}} &= \abs{\ang{\scp{\delt\phi_m}{\delt\psi_m}}{\scp{\zeta}{\xi}}_{\mathcal{H}^1}} \\
            &\leq\norm{\phi_m}_{H^1(\Om)}\norm{\boldsymbol{v}}_{L^3(\Om)}\norm{\zeta}_{H^1(\Om)} + C\norm{\mu_m}_{H^1(\Om)}\norm{\zeta}_{H^1(\Om)} \\
            &\quad + \norm{\psi_m}_{H^1(\Ga)}\norm{\boldsymbol{w}}_{L^3(\Ga)}\norm{\xi}_{H^1(\Ga)} + C\norm{\theta_m}_{H^1(\Ga)}\norm{\xi}_{H^1(\Ga)} \\
            &\quad + L^{-1}\norm{\beta\theta_m-\mu_m}_{L^2(\Ga)}\norm{\beta\xi-\zeta}_{L^2(\Ga)} \\
        \end{split}
    \end{align*}
    on $[0,T_m]$. Hence, after taking the supremum over all $\scp{\zeta^\ast}{\xi^\ast}\in\mathcal{H}^1$ satisfying $\norm{\scp{\zeta^\ast}{\xi^\ast}}_{\mathcal{H}^1} \leq 1$, taking the square of the resulting inequality and integrating in time over $[0,T_m]$, we deduce that
    \begin{align*}
        \norm{\scp{\delt\phi_m}{\delt\psi_m}}_{L^2(0,T_m;(\mathcal{H}^1)')}^2 &\leq C\norm{\phi_m}_{L^\infty(0,T_m;H^1(\Om))}^2\norm{\boldsymbol{v}}_{L^2(0,T_m;L^3(\Om))}^2 \\
        &\qquad + C\norm{\psi_m}_{L^\infty(0,T_m;H^1(\Ga))}^2\norm{\boldsymbol{w}}_{L^2(0,T_m;L^3(\Ga))}^2 \\
        &\qquad + C\norm{\scp{\Grad\mu_m}{\Gradg\theta_m}}_{L^2(0,T_m;\mathcal{L}^2)}^2 \\
        &\qquad + C\norm{\beta\theta_m-\mu_m}_{L^2(0,T_m;L^2(\Ga))}^2.
    \end{align*}
    Exploiting the uniform estimates \eqref{EST:PPFULLH^1} and \eqref{EST:full-mu-theta} we conclude
    \begin{align}\label{EST:delt-uniform}
        \norm{\scp{\delt\phi_m}{\delt\psi_m}}_{L^2(0,T_m;(\mathcal{H}^1)')} \leq C.
    \end{align} 
    \textbf{Step 3: Extension of the approximate solution onto} $[0,T]$. Using the definition of the approximate solution \eqref{galerkin-ansatz} and the uniform estimate \eqref{EST:PPFULLH^1}, we obtain for any $T_m\in[0,T_m^\ast)$, all $t\in[0,T_m]$ and $j\in\{1,\ldots,m\}$
        \begin{align*}
        \abs{a_j^m(t)} + \abs{b_j^m(t)} &= \abs{\scp{\phi_m(t)}{\zeta_j}_{L^2(\Om)}} + \abs{\scp{\psi_m(t)}{\xi_j}_{L^2(\Ga)}} \\
        &\leq \norm{\phi_m}_{L^\infty(0,T_m;L^2(\Om))} + \norm{\psi_m}_{L^\infty(0,T_m;L^2(\Ga))} \leq C.
    \end{align*}
    This shows that the solution $(\mathbf{a}^m,\mathbf{b}^m)^\top$ is bounded on the time interval $[0,T_m^\ast)$ by a constant that is independent of $m$ and $T_m^\ast$. Hence, by classical ODE theory, this allows us to extend the solution beyond the time $T_m^\ast$. However, as $\scp{\mathbf{a}^m}{\mathbf{b}^m}^\top$ was assumed to be non-extendable, this is a contradiction. We thus infer that the solution $\scp{\mathbf{a}^m}{\mathbf{b}^m}^\top$ actually exists on $[0,T]$. Since the coefficients $\scp{\mathbf{c}^m}{\mathbf{d}^m}^\top$ can be reconstructed from $(\mathbf{a}^m,\mathbf{b}^m)^\top$, we further conclude that the coefficients $(\mathbf{c}^m,\mathbf{d}^m)^\top$ also exist on $[0,T]$. This automatically entails that the approximate solution $(\phi_m,\psi_m,\mu_m,\theta_m)$ exists on $[0,T]$ and satisfies the discretized weak formulation \eqref{DWF} on $[0,T]$ for all $m\in\N$. Additionally, as the particular choice of $T_m$ did not play any role in the derivation of the uniform estimates established in Step 2, we infer that the estimates \eqref{EST:PPFULLH^1}, \eqref{EST:GRADMT+DIFF}, \eqref{EST:mu-theta-L^4-L^2}, \eqref{EST:full-mu-theta} and \eqref{EST:delt-uniform}  remain true after replacing $T_m$ with $T$. In summary, we conclude that for each $m\in\N$, the approximate solution $(\phi_m,\psi_m,\mu_m,\theta_m)$ satisfies the uniform estimate
    \begin{align}\label{EST:uniform-final}
        \begin{split}
            &\norm{\scp{\delt\phi_m}{\delt\psi_m}}_{L^2(0,T;(\mathcal{H}^1)')} + \norm{\scp{\phi_m}{\psi_m}}_{L^\infty(0,T;\mathcal{H}^1)} 
            + \norm{\alpha\psi_m-\phi_m}_{L^\infty(0,T;L^2(\Ga))}
             \\
           &\;\; + \norm{\scp{\mu_m}{\theta_m}}_{L^2(0,T;\mathcal{H}^1)} 
           + \norm{\scp{\mu_m}{\theta_m}}_{L^4(0,T;\mathcal{L}^2)}
           + \norm{\beta\theta_m-\mu_m}_{L^2(0,T;L^2(\Ga))}  
            \leq C.
        \end{split}
    \end{align}

    \textbf{Step 4: Convergence to a weak solution.}
    In view of the uniform estimate \eqref{EST:uniform-final}, the Banach--Alaoglu theorem and the Aubin--Lions--Simon lemma imply the existence of functions $\phi, \psi, \mu$ and $\theta$ such that
    \begin{alignat}{3}
        \scp{\delt\phi_m}{\delt\psi_m} &\rightarrow \scp{\delt\phi}{\delt\psi} \quad\quad&&\text{weakly in } L^2(0,T;(\mathcal{H}^1)'),\label{CONV:delt} \\[0.3em]
        \scp{\phi_m}{\psi_m} &\rightarrow  \scp{\phi}{\psi} &&\text{weakly-star in } L^\infty(0,T;\mathcal{H}^1),\nonumber\\
        & &&\text{strongly in } C([0,T];\mathcal{H}^s) \text{ for all } s\in[0,1), \label{CONV:PP} \\[0.3em]
        \scp{\mu_m}{\theta_m} &\rightarrow \scp{\mu}{\theta} &&\text{weakly in } L^4(0,T;\mathcal{L}^2)\cap L^2(0,T;\mathcal{H}^1), \label{CONV:MT} \\[0.3em]
        \beta\theta_m - \mu_m &\rightarrow \beta\theta - \mu &&\text{weakly in } L^2(0,T;L^2(\Ga)),\label{CONV:TM:B} \\[0.3em]
        \alpha\psi_m - \phi_m &\rightarrow \alpha\psi - \phi &&\text{weakly-star in } L^\infty(0,T;L^2(\Ga)),
    \end{alignat}
    as $m\rightarrow\infty$, along a non-relabeled subsequence. From these convergences, we readily obtain the desired regularities \eqref{REGPP}--\eqref{REGMT} of the functions $\phi, \psi, \mu$ and $\theta$. Hence, Definition \ref{DEF:WS:REG} is fulfilled. 
    
    Using Sobolev's embedding theorem, we infer from \eqref{CONV:PP} that
    \begin{alignat}{3}
        \phi_m&\rightarrow\phi\quad\quad &&\text{strongly in } C([0,T];L^r(\Om)) \text{ for all } r\in[2,6), \text{ and a.e.~in } Q, \label{CONV:ph}\\
        \psi_m&\rightarrow\psi &&\text{strongly in } C([0,T];L^r(\Ga)) \text{ for all } r\in[2,\infty), \text{ and a.e.~on } \Sigma \label{CONV:ps},
    \end{alignat}
    as $m\rightarrow\infty$, after another subsequence extraction. Further, by the trace theorem, \eqref{CONV:PP} additionally implies that
    \begin{align}\label{CONV:ps-ph}
        \alpha\psi_m-\phi_m\rightarrow\alpha\psi-\phi\quad\quad \text{strongly in } C([0,T];L^2(\Ga))
    \end{align}
    as $m\rightarrow\infty$, after another subsequence extraction. Moreover, due to the growth conditions on $F'$ from \ref{ASSUMP:POTENTIALS:1}, the uniform estimate \eqref{EST:PPFULLH^1} and the Sobolev embedding $H^1(\Om)\emb L^6(\Om)$, we deduce that $F'(\phi_m)$ is bounded in $L^{6/5}(Q)$ uniformly in $m\in\N$. Hence, there exists a function $f\in L^{6/5}(Q)$ such that $F'(\phi_m)\rightarrow f$ weakly in $L^{6/5}(Q)$ as $m\rightarrow\infty$. Considering that $F'(\phi_m)\rightarrow F'(\phi)$ a.e.~in $Q$ as $m\rightarrow\infty$ due to \eqref{CONV:ph}, we infer from a convergence principle based on Egorov's theorem that $F'(\phi) = f$. We thus conclude that
    \begin{align}
        F'(\phi_m)\rightarrow F'(\phi) \quad\quad\text{weakly in } L^{6/5}(Q) \text{ and a.e.~in } Q
    \end{align}
    as $m\rightarrow\infty$. Recalling the growth conditions on $G$ and $G'$ (see \ref{ASSUMP:POTENTIALS:1}), the convergence in \eqref{CONV:ps} in combination with Lebesgue's general convergence theorem yields
    \begin{alignat}{3}
        G(\psi_m)&\rightarrow G(\psi)\quad\quad&&\text{strongly in } L^1(\Sigma), \text{ and a.e.~on } \Sigma, \label{CONV:G} \\
        G'(\psi_m)&\rightarrow G'(\psi)&&\text{strongly in } L^2(\Sigma), \text{ and a.e.~on } \Sigma \label{CONV:G'},
    \end{alignat}
    as $m\rightarrow\infty$.
    Furthermore, by means of Lebesgue's dominated convergence theorem, it follows from \eqref{CONV:ph}, \eqref{CONV:ps} and \ref{ASSUMP:MOBILITY} that
    \begin{alignat}{3}
        m_\Om(\phi_m)&\rightarrow m_\Om(\phi) \quad&&\text{strongly in } L^r(Q), \text{ and a.e.~in Q}\label{CONV:mobo}, \\
        m_\Ga(\psi_m)&\rightarrow m_\Ga(\psi) &&\text{strongly in } L^r(\Sigma), \text{ and a.e.~on }\Sigma\label{CONV:mobg},
    \end{alignat}
    for all $r\in[2,\infty)$, as $m\rightarrow\infty$. Moreover, as the mobility functions $m_\Om$ and $m_\Ga$ are bounded (see \ref{ASSUMP:MOBILITY}), we use Lebesgue's general convergence theorem along with the weak-strong convergence principle to find that
    \begin{alignat}{3}
        \sqrt{m_\Om(\phi_m)}\Grad\mu_m &\rightarrow \sqrt{m_\Om(\phi)}\Grad\mu  \quad\quad&&\text{weakly in } L^2(Q), \label{CONV:sqmob:mu} \\
        \sqrt{m_\Ga(\psi_m)}\Gradg\theta_m &\rightarrow \sqrt{m_\Ga(\psi)}\Gradg\theta &&\text{weakly in } L^2(\Sigma), \label{CONV:sqmob:theta} \\
         m_\Om(\phi_m) \Grad\mu_m &\rightarrow m_\Om(\phi)\Grad\mu&&\text{weakly in } L^2(Q), \label{CONV:mob:mu} \\
        m_\Ga(\psi_m) \Gradg\theta_m &\rightarrow m_\Ga(\psi) \Gradg\theta &&\text{weakly in } L^2(\Sigma), \label{CONV:mob:th}
    \end{alignat}
    as $m\rightarrow\infty.$
    Lastly, exploiting the convergences \eqref{CONV:MT}, \eqref{CONV:ph} and \eqref{CONV:ps}, and using 
    again the weak-strong convergence principle,
    we infer
    \begin{alignat}{3}
        \phi_m\,\boldsymbol{v}\cdot\Grad\mu_m&\rightarrow\phi\,\boldsymbol{v}\cdot\Grad\mu\quad\quad\quad&&\text{strongly in } L^1(Q), \text{ and a.e.~in } Q, \label{CONV:conv:pm} \\
        \psi_m\,\boldsymbol{w}\cdot\Gradg\theta_m&\rightarrow\psi\,\boldsymbol{w}\cdot\Gradg\theta&&\text{strongly in } L^1(\Sigma), \text{ and a.e.~on } \Sigma, \label{CONV:conv:pt}
    \end{alignat}
    as $m\rightarrow\infty$. 
    
    We now multiply both equations of the discretized weak formulation \eqref{DWF} with an arbitrary test function $\sigma\in C_c^\infty([0,T])$ and integrate in time from $0$ to $T$. Using the convergences \eqref{CONV:delt}--\eqref{CONV:conv:pt} established above, we may pass to the limit $m\rightarrow\infty$ to infer that
    \begin{subequations}\label{EQ:integrated-weak-formulation}
        \begin{align}
            &\int_0^T \ang{\scp{\delt\phi}{\delt\psi}}{\scp{\zeta_j}{\xi_j}}_{\mathcal{H}^1}\,\sigma\dt - \int_0^T\intO \phi\boldsymbol{v}\cdot\Grad\zeta_j\,\sigma\dxt - \int_0^T\intG \psi\boldsymbol{w}\cdot\Gradg\xi_j\,\sigma\dGt \nonumber \\
            \quad&= - \int_0^T\intO m_\Om(\phi)\Grad\mu\cdot\Grad\zeta_j\,\sigma\dxt - \int_0^T\intG  m_\Ga(\psi)\Gradg\theta\cdot\Gradg\xi_j\,\sigma\dGt \\
            &\quad - L^{-1}\int_0^T\intG(\beta\theta-\mu)(\beta\xi_j - \zeta_j)\,\sigma\dGt, \nonumber\\
            &\int_0^T\intO \mu\,\zeta_j\,\sigma\dxt + \int_0^T\intG\theta\,\xi_j\,\sigma\dGt  \\
            &= \int_0^T\intO\Grad\phi\cdot\Grad\zeta_j\,\sigma + F'(\phi)\zeta_j\,\sigma \dxt + \int_0^T\intG\Gradg\psi\cdot\Gradg\xi_j\,\sigma + G'(\psi)\xi_j\,\sigma\dGt \nonumber \\
            &\quad + K^{-1}\int_0^T\intG(\alpha\psi-\phi)(\alpha\xi_j - \zeta_j)\,\sigma \dGt, \nonumber
        \end{align}
    \end{subequations}
    hold for all $j\in\N$ and all test functions $\sigma\in C_c^\infty([0,T])$. Since $\text{span}\{\scp{\zeta_j}{\xi_j} : j\in\N\}$ is dense in $\mathcal{H}^1$, this proves that the quadruplet $(\phi,\psi,\mu,\theta)$ satisfies the weak formulation \eqref{WF} for all test functions $\scp{\zeta}{\xi}\in\mathcal{H}^1$. We thus conclude that Definition \ref{DEF:WS:WF} is fulfilled. 

    Since the approximate solutions $\scp{\phi_m}{\psi_m}$ satisfy the discrete mass conservation law (see \eqref{EQ:DMCL}) for all $m\in\N$, we can use the convergence \eqref{CONV:PP} to deduce that $\scp{\phi}{\psi}$ satisfies the mass conservation law \eqref{MCL} as well.
    Thus, Definition \ref{DEF:WS:MCL} is fulfilled. \\ [0.3em]
    Concerning the initial conditions, we infer from \eqref{DWF:IC}, due to the convergence properties of the projections, that
    \begin{align*}
        \scp{\phi_m(0)}{\psi_m(0)}\rightarrow\scp{\phi_0}{\psi_0}\quad\text{strongly in }\mathcal{L}^2
    \end{align*}
    as $m\rightarrow\infty$. Additionally, the convergences \eqref{CONV:ph} and \eqref{CONV:ps} yield
    \begin{align*}
        \scp{\phi_m(0)}{\psi_m(0)}\rightarrow\scp{\phi(0)}{\psi(0)}\quad\text{strongly in }\mathcal{L}^2
    \end{align*}
    as $m\rightarrow\infty$. We deduce that $\phi(0) =  \phi_0$ a.e.~in $\Om$ and $\psi(0) = \psi_0$ a.e.~on $\Ga$. This proves that Definition \ref{DEF:WS:IC} is fulfilled.

    To verify the energy inequality \eqref{WEDL}, we consider an arbitrary non-negative test function $\sigma\in C_c^\infty(0,T)$. Multiplying \eqref{EQ:DEI} with $\sigma$ and integrating in time from $0$ to $T$ yields
    \begin{align}
        \begin{split}
            &\int_0^T E_K(\phi_m(t),\psi_m(t))\,\sigma(t)\dt \\
            &\quad\quad + \int_0^T\int_0^t\intO\left( m_\Om(\phi_m)\abs{\Grad\mu_m}^2 - \phi_m\boldsymbol{v}\cdot\Grad\mu_m\right)\sigma(t)\dx\ds\dt \\
            &\quad\quad + \int_0^T\int_0^t\intG\left( m_\Ga(\psi_m)\abs{\Gradg\theta_m}^2 - \psi_m\boldsymbol{w}\cdot\Gradg\theta_m\right)\sigma(t)\dG\ds\dt \\
            &\quad\quad + L^{-1} \int_0^T\int_0^t\intG (\beta\theta_m-\mu_m)^2\sigma(t)\dG\ds\dt \\
            &\ = \int_0^T E_K(\phi_m(0),\psi_m(0))\,\sigma(t)\dt.
        \end{split}
    \end{align}
    As $m\rightarrow\infty$, we infer from \eqref{CONV:G} that
    \begin{align}\label{CONV:GINT}
        \int_0^T \left(\intG G(\psi_m)\dG\right)\sigma(t)\dt \rightarrow \int_0^T\left(\intG G(\psi)\dG\right)\sigma(t)\dt.
    \end{align}
    Additionally, using \eqref{CONV:ph} and \ref{ASSUMP:POTENTIALS:1}, Fatou's lemma implies
    \begin{align}
        \int_0^T\left(\intO F(\phi)\dx\right)\sigma(t)\dt \leq\liminf_{m\rightarrow\infty} \int_0^T\left(\intO F(\phi_m)\dx\right)\sigma(t)\dt.
    \end{align}
    Moreover, from the strong convergence in \eqref{CONV:ps-ph} we obtain
    \begin{align}
        \int_0^T\left(\intG \abs{\alpha\psi_m-\phi_m}^2\dG\right)\sigma(t)\dt \rightarrow \int_0^T\left(\intG \abs{\alpha\psi-\phi}^2\dG\right)\sigma(t)\dt,
    \end{align}
    as $m\rightarrow\infty$. Lastly, in view of \eqref{CONV:PP} and the weak lower semicontinuity of norms, we have
    \begin{align}\label{CONV:LSCPP}
        \begin{split}
            &\int_0^T\left(\frac{1}{2}\norm{\Grad\phi}_{L^2(\Om)}^2 + \frac{1}{2}\norm{\Gradg\psi}_{L^2(\Ga)}^2\right)\sigma(t)\dt \\
            &\quad\leq \liminf_{m\rightarrow\infty} \int_0^T\left(\frac{1}{2}\norm{\Grad\phi_m}^2_{L^2(\Om)} + \frac{1}{2}\norm{\Gradg\psi_m}^2_{L^2(\Ga)}\right)\sigma(t)\dt.
            \end{split}
    \end{align}
    Combining \eqref{CONV:GINT}--\eqref{CONV:LSCPP} yields
    \begin{align}\label{CONV:LSC:E_K}
        \int_0^T E_K(\phi(t),\psi(t))\,\sigma(t)\dt \leq \liminf_{m\rightarrow\infty} \int_0^T E_K(\phi_m(t),\psi_m(t))\,\sigma(t) \dt.
    \end{align}
    As a consequence of \eqref{CONV:sqmob:mu}, \eqref{CONV:sqmob:theta} and the weak lower semicontinuity of norms with respect to the weak convergence, another application of Fatou's lemma entails that
    \begin{align}\label{CONV:LSCMT}
        \begin{split}
            &\int_0^T\int_0^t\intO m_\Om(\phi)\abs{\Grad\mu}^2\,\sigma(t)\dxs\dt + \int_0^T\int_0^t\intG m_\Ga(\psi)\abs{\Gradg\theta}^2\,\sigma(t)\dGs\dt \\
            &\leq\liminf_{m\rightarrow\infty}\left[ \int_0^T\int_0^t\intO m_\Om(\phi_m)\abs{\Grad\mu_m}^2\,\sigma(t)\dxs\dt \right.\\
            &\qquad\qquad\qquad\left.
            + \int_0^T\int_0^t\intG m_\Ga(\psi_m)\abs{\Gradg\theta_m}^2\,\sigma(t)\dGs\dt
            \right].
            \end{split}
    \end{align}
    Further, due to the convergences \eqref{CONV:conv:pm} and \eqref{CONV:conv:pt}, we find that
    \begin{align}
        \begin{split}
            &\int_0^T \int_0^t \intO \phi_m\boldsymbol{v}\cdot\Grad\mu_m\,\sigma(t)\dxs\dt + \int_0^T\int_0^t\intG \psi_m\boldsymbol{w}\cdot\Gradg\theta_m\,\sigma(t)\dGs\dt \\
            &\quad \longrightarrow \int_0^T \int_0^t \intO \phi\boldsymbol{v}\cdot\Grad\mu\,\sigma(t)\dxs\dt + \int_0^T\int_0^t\intG \psi\boldsymbol{w}\cdot\Gradg\theta\,\sigma(t)\dGs\dt
        \end{split}
    \end{align}
    as $m\rightarrow\infty$.
    Additionally, as convex and continuous functionals are weakly lower semicontinuous, we conclude from \eqref{CONV:TM:B} that
    \begin{align}\label{CONV:LSC:T-M}
        \int_0^T\int_0^t\intG (\beta\theta-\mu)^2 \,\sigma(t)\dGs\dt \leq \liminf_{m\rightarrow\infty} \int_0^T\int_0^t\intG (\beta\theta_m-\mu_m)^2 \,\sigma(t)\dGs\dt.
    \end{align}
    Lastly, recalling the growth conditions on $F$ and $G$ (see \ref{ASSUMP:POTENTIALS:1}), we use the initial conditions \eqref{DWF:IC} as well as the convergence properties of the projections $\projam$ and $\projbm$ together with Lebesgue's general convergence theorem to infer that
    \begin{align}\label{CONV:IE}
        \lim_{m\rightarrow\infty} E_K(\phi_m(0),\psi_m(0)) = E_K(\phi_0,\psi_0).
    \end{align}
    Collecting \eqref{CONV:LSCMT}--\eqref{CONV:IE} finally yields
    \begin{align}\label{WEDL:AAT}
        \begin{split}
            &E_K(\phi(t),\psi(t)) + \int_0^t\intO m_\Om(\phi)\abs{\Grad\mu}^2\dxs + \int_0^t\intG m_\Ga(\psi)\abs{\Gradg\theta}^2\dGs \\
            &\quad - \int_0^t\intO\phi\boldsymbol{v}\cdot\Grad\mu\dxs - \int_0^t\intG \psi\boldsymbol{w}\cdot\Gradg\theta\dGs \\
            &\quad + L^{-1} \int_0^t\intG (\beta\theta-\mu)^2\dGs \\
            &\leq E_K(\phi_0,\psi_0)
        \end{split}
    \end{align}
    for almost every $t\in [0,T]$. To prove that \eqref{WEDL:AAT} holds for every $t\in[0,T]$, first note that all integral terms in \eqref{WEDL:AAT} depend continuously on time. 
    Furthermore, due to $(\phi,\psi)\in C([0,T];\mathcal{L}^2)$, we deduce that the following functions
    \begin{align*}
        t\mapsto \norm{\Grad\phi(t)}_{L^2(\Om)}^2, \quad t\mapsto\norm{\Gradg\psi(t)}_{L^2(\Ga)}^2, \quad t\mapsto \intO F(\phi(t)) \dx, \quad t\mapsto \intG G(\psi(t)) \dG
    \end{align*}
    are lower semicontinuous on $[0,T]$. For the first two functions, this is a consequence of the lower semicontinuity of the respective norms, while for the last two functions, it follows from Fatou's lemma. This already entails that the weak solution $(\phi,\psi,\mu,\theta)$ satisfies the energy inequality for \textit{all} times $t\in[0,T]$. Thus, Definition \ref{DEF:WS:WEDL} is fulfilled. \\[0.3em]
    We have therefore shown that the quadruplet $(\phi,\psi,\mu,\theta)$ is a weak solution of system \eqref{EQ:SYSTEM} in the sense of Definition \ref{DEF:WS}. 
\end{proof}

\begin{remark} \label{REM:GAL}
    This proof should work similarly in the cases $K=L=0$ and $K=L=\infty$. In the first case, one uses a Faedo--Galerkin scheme as in \cite{Giorgini2023} based on eigenfunctions of a suitable bulk-surface elliptic problem with Dirichlet type coupling condition (cf.~\cite[Theorem 3.3]{Knopf2021}). In the second case, the bulk-surface Cahn--Hilliard equation \eqref{EQ:SYSTEM} reduces to two, uncoupled Cahn--Hilliard equations, one in $\Om$ and one on $\Ga$. Therefore, a Faedo--Galerkin basis as in the above proof can be used. Also note that, while we do not have the bulk-surface Poincar\'{e} inequality at our disposal anymore, we can use the standard Poincar\'{e} inequality for functions with vanishing mean, since the approximate solutions satisfy the mass conservation law \eqref{MCL} for $L=\infty$.
\end{remark}

\section{Asymptotic limits and existence of weak solutions to the limit models}
\label{SECT:ASYMPTLIM}
In this section, we investigate the asymptotic limits $K\rightarrow 0$ and $K\rightarrow\infty$, and $L\rightarrow 0$ and $L\rightarrow\infty$ of the system \eqref{EQ:SYSTEM}.

\begin{remark}
    In this section, we will need the additional assumption \ref{ASSUMP:POTENTIALS:2} when approaching the limit cases $K\in\{0,\infty\}$. This is because the constant in the bulk-surface Poincar\'{e} inequality \ref{PRELIM:POINCINEQ} depends on $K$ in some way, but we do not know this dependence explicitly. In particular, it is unclear how this constant behaves if we send $K$ to zero or infinity, respectively. Therefore, we cannot rely on this Poincar\'{e} inequality to obtain suitable uniform bounds, but instead, we use \ref{ASSUMP:POTENTIALS:2} to directly obtain uniform bounds from the energy functional. 
\end{remark}

\subsection{The limit \texorpdfstring{$K\rightarrow 0$}{K -> 0} and the existence of a weak solution if \texorpdfstring{$(K,L)\in\{0\}\times(0,\infty)$}{K = 0}}

\begin{theorem}\textnormal{(Asymptotic limit $K\rightarrow 0$)}\label{THEOREM:K->0}
    Suppose that the assumptions \ref{ASSUMP:1}--\ref{ASSUMP:POTENTIALS:2} hold, let $\scp{\phi_0}{\psi_0}\in\mathcal{D}_\alpha$ be arbitrary initial data, $L\in(0,\infty)$, $\boldsymbol{v}\in L^2(0,T;\mathbf{L}_\Div^3(\Om))$ and $\boldsymbol{w}\in L^2(0,T;\mathbf{L}^3_\tau(\Ga))$. For any $K\in(0,\infty)$, let $(\phi_K,\psi_K,\mu_K,\theta_K)$ denote a weak solution of the system \eqref{EQ:SYSTEM} in the sense of Definition~\ref{DEF:WS} with initial data $\scp{\phi_0}{\psi_0}$. 
    Then, there exists a quadruplet $(\phi_\ast,\psi_\ast,\mu_\ast,\theta_\ast)$ with $\phi_\ast = \alpha\psi_\ast$ a.e.~on $\Sigma$ such that
    \begin{alignat*}{3}
        \scp{\delt\phi_K}{\delt\psi_K} &\rightarrow \scp{\delt\phi_\ast}{\delt\psi_\ast} \quad\quad&&\text{weakly in } L^2(0,T;(\mathcal{H}^1)'), \\[0.3em]
        \scp{\phi_K}{\psi_K} &\rightarrow  \scp{\phi_\ast}{\psi_\ast} &&\text{weakly-star in } L^\infty(0,T;\mathcal{H}^1),\nonumber\\
        & &&\text{strongly in } C([0,T];\mathcal{H}^s) \text{ for all } s\in[0,1),\\[0.3em]
        \scp{\mu_K}{\theta_K} &\rightarrow \scp{\mu_\ast}{\theta_\ast} &&\text{weakly in } L^2(0,T;\mathcal{H}^1), 
        \\[0.3em]
        \alpha\psi_K - \phi_K &\rightarrow 0 &&\text{strongly in } L^\infty(0,T,L^2(\Ga)),
    \end{alignat*}
    as $K\rightarrow 0$, up to subsequence extraction, with
    \begin{align}
    \label{EST:K->0}
        \norm{\alpha\psi_K-\phi_K}_{L^\infty(0,T;L^2(\Ga))} \leq C\sqrt{K},
    \end{align}
    and the limit $(\phi_\ast,\psi_\ast,\mu_\ast,\theta_\ast)$ is a weak solution to the system \eqref{EQ:SYSTEM} in the sense of Definition~\ref{DEF:WS} with $K=0$.
\end{theorem}

\begin{remark}
    \begin{enumerate}[label=\textnormal{(\alph*)},leftmargin=*]
    \item As the right-hand side in \eqref{EST:K->0} tends to zero as $K\to 0$, this explains why the Dirichlet type boundary condition $\phi_* = \alpha\psi_*$ a.e.~on $\Sigma$ appears in the limit model corresponding to $K=0$.
    \item The result of Theorem~\ref{THEOREM:K->0} remains valid if we replace the initial data $\scp{\phi_0}{\psi_0}\in\mathcal{D}_\alpha$ by a sequence $\scp{\phi_{0,K}}{\psi_{0,K}}\in\mathcal{H}^1$ satisfying
    \begin{align*}
        E_K(\phi_{0,K},\psi_{0,K}) \rightarrow E_0(\phi_0,\psi_0)
        \quad\text{as $K\rightarrow 0$}
    \end{align*}
    for some pair $\scp{\phi_0}{\psi_0}\in\mathcal{H}^1$ (see also \cite[Theorem~2.3]{Knopf2020}). In this case, one can show that $\phi_0 = \alpha\psi_0$ a.e.~on $\Ga$, and that $\scp{\phi_\ast}{\psi_\ast}$ is a weak solution corresponding to the initial data $\scp{\phi_0}{\psi_0}$.
    \end{enumerate}
\end{remark}

\begin{proof}[Proof of Theorem~\ref{THEOREM:K->0}]
    We consider an arbitrary sequence $(K_m)_{m\in\N}\subset(0,\infty)$ such that $K_m\rightarrow 0$ as $m\rightarrow\infty$ and a corresponding sequence of weak solutions $(\phi_{K_m},\psi_{K_m},\mu_{K_m},\theta_{K_m})$ to the initial data $\scp{\phi_0}{\psi_0}$.
    In this proof, we use the letter $C$ to denote generic positive constants independent of $K_m$ and $m$. In order to prove suitable uniform estimates, we make use of the energy inequality and the additional growth assumptions made on $F$ and $G$ (see \ref{ASSUMP:POTENTIALS:2}). 
    Let now $m\in\N$ be arbitrary.    
    First, note that $\phi_0 = \alpha\psi_0$ a.e.~on $\Ga$ directly implies 
    \begin{align}\label{EQ:EKM=E0}
        E_{K_m}(\phi_0,\psi_0) = E_0(\phi_0, \psi_0) \leq C.
    \end{align}
    Next, due to the growth condition \ref{ASSUMP:POTENTIALS:2}, we find that
    \begin{align*}
        \begin{split}
            &\frac12\intO \abs{\Grad\phi_{K_m}}^2\dx + \intO F(\phi_{K_m})\dx + \frac12\intG \abs{\Gradg\psi_{K_m}}^2\dG + \intG G(\psi_{K_m})\dG \\
            &\geq \frac12\intO\abs{\Grad\phi_{K_m}}^2\dx + a_F\intO \abs{\phi_{K_m}}^2\dx - b_F\abs{\Om} \\
            &\qquad + \frac12\intG\abs{\Gradg\psi_{K_m}}^2\dG + a_G\intG\abs{\psi_{K_m}}^2\dG - b_G\abs{\Ga} \\
            &\geq c\norm{\scp{\phi_{K_m}}{\psi_{K_m}}}^2_{\mathcal{H}^1} - b_F\abs{\Om} - b_G\abs{\Ga},
        \end{split}
    \end{align*}
    where $c = \min\{\frac12, a_F, a_G\}$.
    Then, using the energy inequality \eqref{WEDL}, and bounding the convective terms as in \eqref{EST:convective-terms}, we deduce that\vspace{-1ex}
    \begin{align*}
        \begin{split}
            &c\norm{\scp{\phi_{K_m}}{\psi_{K_m}}}_{\mathcal{H}^1}^2 
            + \frac{1}{2} K_m^{-1}\intG \abs{\alpha\psi_{K_m}-\phi_{K_m}}^2\dG
            + \frac{m_\Omega^*}{2} \int_0^t\intO \abs{\Grad\mu_{K_m}}^2\dxs \\
            &\quad + \frac{m_\Gamma^*}{2}\int_0^t\intG\abs{\Gradg\theta_{K_m}}^2\dGs + L^{-1}\int_0^t\intG (\beta\theta_{K_m}-\mu_{K_m})^2 \dGs \\
            &\leq C + C\int_0^t\left(\norm{\boldsymbol{v}}^2_{L^3(\Om)} + \norm{\boldsymbol{w}}_{L^3(\Ga)}^2\right)\norm{\scp{\phi_{K_m}}{\psi_{K_m}}}_{\mathcal{H}^1}^2 \ds
        \end{split}
    \end{align*}
    a.e.~on $[0,T]$. Thus, using Gronwall's lemma, we readily infer that
    \begin{align}\label{est:unif:pre:K->0}
        \begin{split}
            &\norm{\scp{\phi_{K_m}}{\psi_{K_m}}}_{L^\infty(0,T;\mathcal{H}^1)} + K_m^{-1/2}\norm{\alpha\psi_{K_m} - \phi_{K_m}}_{L^\infty(0,T;L^2(\Gamma))} \\
            &\quad + \norm{\scp{\Grad\mu_{K_m}}{\Gradg\theta_{K_m}}}_{L^2(0,T;\mathcal{L}^2)} + \norm{\beta\theta_{K_m} - \mu_{K_m}}_{L^2(0,T;L^2(\Gamma))} \leq C.
        \end{split}
    \end{align}
    This already entails \eqref{EST:K->0}.
    We now test the weak formulation \eqref{WF:MT} for $(\mu_{K_m},\theta_{K_m})$ with $(\eta,\vartheta)\in\mathcal{D}_\alpha$. As $\eta = \alpha\vartheta$ a.e.~on $\Ga$, the corresponding boundary term involving $K_m$ vanishes. Therefore, arguing as in the proof of Theorem~\ref{THEOREM:EOWS}, we deduce
    \begin{align}\label{jonas4}
        \norm{\scp{\mu_{K_m}}{\theta_{K_m}}}_{L^\infty(0,T;\mathcal{D}_\alpha^\prime)} \leq C.
    \end{align}
    Now, using Lemma~\ref{LEMMA:DUALITY}, we find that
    \begin{align*}
        \norm{\scp{\mu_{K_m}}{\theta_{K_m}}}_{\mathcal{L}^2}^2 \leq 2\norm{\scp{\mu_{K_m}}{\theta_{K_m}}}_{\mathcal{D}_\alpha^\prime}\norm{\scp{\Grad\mu_{K_m}}{\Gradg\theta_{K_m}}}_{\mathcal{L}^2} + \norm{\scp{\mu_{K_m}}{\theta_{K_m}}}_{\mathcal{D}_\alpha^\prime}^2.
    \end{align*}
    In combination with \eqref{est:unif:pre:K->0} and \eqref{jonas4}, this yields
    \begin{align}\label{EST:MT:FULL:KM:0}
        \norm{\scp{\mu_{K_m}}{\theta_{K_m}}}_{L^2(0,T;\mathcal{H}^1)}\leq C.
    \end{align}
    Lastly, proceeding similarly as in the derivation of \eqref{EST:delt-uniform} in the proof of Theorem~\ref{THEOREM:EOWS}, it follows from the weak formulation \eqref{WF:PP} and estimate \eqref{est:unif:pre:K->0} that
    \begin{align}\label{EST:DELT:KM:0}
        \norm{\scp{\delt\phi_{K_m}}{\delt\psi_{K_m}}}_{L^2(0,T;(\mathcal{H}^1)^\prime)} \leq C.
    \end{align}
    In summary, combining \eqref{est:unif:pre:K->0}, \eqref{EST:MT:FULL:KM:0} and \eqref{EST:DELT:KM:0}, we have thus shown that
    \begin{align}
        \begin{split}
            &\norm{\scp{\delt\phi_{K_m}}{\delt\psi_{K_m}}}_{L^2(0,T;(\mathcal{H}^1)^\prime)} 
            + \norm{\scp{\phi_{K_m}}{\psi_{K_m}}}_{L^\infty(0,T;\mathcal{H}^1)} \\
            &\quad + \norm{\scp{\mu_{K_m}}{\theta_{K_m}}}_{L^2(0,T;\mathcal{H}^1)} + \norm{\beta\theta_{K_m}-\mu_{K_m}}_{L^2(0,T;L^2(\Ga))} \leq C.
        \end{split}
    \end{align}
    Using the Banach--Alaoglu theorem and the Aubin--Lions--Simon lemma, we deduce the existence of functions $(\phi_\ast,\psi_\ast,\mu_\ast,\theta_\ast)$ such that
    \begin{alignat}{3}
        \scp{\delt\phi_{K_m}}{\delt\psi_{K_m}} &\rightarrow \scp{\delt\phi_\ast}{\delt\psi_\ast} \quad\quad&&\text{weakly in } L^2(0,T;(\mathcal{H}^1)^\prime), \label{CONV:DELT:KM:0}\\[0.3em]
        \scp{\phi_{K_m}}{\psi_{K_m}} &\rightarrow  \scp{\phi_\ast}{\psi_\ast} &&\text{weakly-star in } L^\infty(0,T;\mathcal{H}^1),\nonumber\\
        & &&\text{strongly in } C([0,T];\mathcal{H}^s) \text{ for all } s\in[0,1), \label{CONV:PP:KM:0}\\[0.3em]
        \scp{\mu_{K_m}}{\theta_{K_m}} &\rightarrow \scp{\mu_\ast}{\theta_\ast} &&\text{weakly in } L^2(0,T;\mathcal{H}^1),  
        \label{CONV:MTDIFF:KM:0}
    \end{alignat}
    as $m\rightarrow\infty$, along a non-relabeled subsequence. Additionally, due to \eqref{CONV:PP:KM:0} and the trace theorem, we infer
    \begin{align}
        \alpha\psi_{K_m}-\phi_{K_m} \rightarrow \alpha\psi_\ast-\phi_\ast \quad\quad\text{strongly in } C([0,T];L^2(\Ga))
    \end{align}
    as $m\rightarrow\infty$, which, in combination with \eqref{est:unif:pre:K->0}, already entails $\phi_\ast = \alpha\psi_\ast$ a.e.~on $\Sigma$, i.e., $\scp{\phi_\ast(t)}{\psi_\ast(t)}\in\mathcal{D}_\alpha$ for almost all $t\in[0,T]$. It is then clear that from the convergence properties \eqref{CONV:DELT:KM:0}--\eqref{CONV:MTDIFF:KM:0} the quadruplet $(\phi_\ast,\psi_\ast,\mu_\ast,\theta_\ast)$ has the desired regularity as stated in Definition~\ref{DEF:WS:REG}. \\[0.3em]
    Further, due to \eqref{CONV:PP:KM:0}, we infer that 
    \begin{align}
        \scp{\phi_{K_m}(0)}{\psi_{K_m}(0)}\rightarrow \scp{\phi_\ast(0)}{\psi_\ast(0)} \quad\text{strongly in } \mathcal{L}^2
    \end{align}
    as $m\rightarrow\infty$, and therefore $\phi_\ast(0) = \phi_0$ a.e.~in $\Om$ and $\psi_\ast(0) = \psi_0$ a.e.~on $\Ga$. Thus, Definition~\ref{DEF:WS:IC} is fulfilled. \\[0.3em]
    Regarding the weak formulation, we consider $\scp{\zeta}{\xi}\in\mathcal{H}^1$ and $\scp{\eta}{\vartheta}\in\mathcal{D}_\alpha$ as test functions, multiply the resulting equations with $\sigma\in C_c^\infty([0,T])$ and integrate in time from $0$ to $T$. We obtain
    \begin{subequations}\label{WF:KM:0}
        \begin{align}
        \begin{split}
            &\int_0^T\ang{\scp{\delt\phi_{K_m}}{\delt\psi_{K_m}}}{\scp{\zeta}{\xi}}_{\mathcal{H}^1}\,\sigma \dt - \int_0^T\intO \phi_{K_m}\boldsymbol{v}\cdot\Grad\zeta\,\sigma \dxt 
            \\
            &\quad - \int_0^T\intG \psi_{K_m}\boldsymbol{w}\cdot\Gradg\xi\,\sigma \dGt \\
            &= - \int_0^T\intO m_\Om(\phi_{K_m})\Grad\mu_{K_m}\cdot\Grad\zeta\,\sigma \dxt 
            \\
            &\quad - \int_0^T\intG  m_\Ga(\psi_{K_m})\Gradg\theta_{K_m}\cdot\Gradg\xi\,\sigma \dGt \label{WF:PP:KM:0}
            \\
            &\quad  - L^{-1}\int_0^T\intG(\beta\theta_{K_m}-\mu_{K_m})(\beta\xi - \zeta)\,\sigma\dGt,
        \end{split}
        \\
        \begin{split}
            &\int_0^T\intO \mu_{K_m}\,\eta\,\sigma\dxt + \int_0^T\intG\theta_{K_m}\,\vartheta\,\sigma\dGt 
            \\
            & =  \int_0^T\intO\Grad\phi_{K_m}\cdot\Grad\eta\,\sigma + F'(\phi_{K_m})\eta\,\sigma \dxt\label{WF:MT:KM:0} 
            \\
            &\quad + \int_0^T\intG\Gradg\psi_{K_m}\cdot\Gradg\vartheta\,\sigma + G'(\psi_{K_m})\vartheta\,\sigma \dGt .
        \end{split}
        \end{align}
    \end{subequations}
    Proceeding similarly as in the proof of Theorem~\ref{THEOREM:EOWS}, the convergences \eqref{CONV:DELT:KM:0}--\eqref{CONV:MTDIFF:KM:0} allow us to pass to the limit $m\rightarrow\infty$, and deduce that $(\phi_\ast,\psi_\ast,\mu_\ast,\theta_\ast)$ satisfies the desired weak formulation \eqref{WF}. Thus, Definition~\ref{DEF:WS:WF} is fulfilled. \\[0.3em]
    The mass conservation law follows by passing to the limit $m\rightarrow\infty$ in the mass conservation law for $\scp{\phi_{K_m}}{\psi_{K_m}}$. Alternatively, one can test the weak formulation \eqref{WF:PP} with $\scp{\beta}{1}\in \mathcal{H}^1$, integrate with respect to time from $0$ to $t$, and employ the fundamental theorem of calculus to infer \eqref{MCL}.
    We conclude that Definition~\ref{DEF:WS:MCL} is satisfied. \\[0.3em]
    As the boundary term involving $K_m$ in the energy is non-negative, we note that
    \begin{align}
        E_0(\phi_{K_m}(t),\psi_{K_m}(t)) \leq E_{K_m}(\phi_{K_m}(t),\psi_{K_m}(t))
    \end{align}
    for all $t\in[0,T]$, and thus
    \begin{align}
        \label{IEQ:EKM}
        \begin{split}
            \int_0^T E_0(\phi_\ast(t),\psi_\ast(t))\,\sigma(t)\dt &\leq \liminf_{m\rightarrow\infty} \int_0^T E_0(\phi_{K_m}(t),\psi_{K_m}(t))\,\sigma(t)\dt 
            \\
            &\leq \liminf_{m\rightarrow\infty} \int_0^T E_{K_m}(\phi_{K_m}(t),\psi_{K_m}(t))\,\sigma(t)\dt
        \end{split}
    \end{align}
    for all non-negative test functions $\sigma\in C_c^\infty(0,T)$. Here, the first inequality follows by proceeding similarly as in the proof of Theorem~\ref{THEOREM:EOWS}. 
    We now use the energy inequality \eqref{WEDL} written for $(\phi_{K_m},\psi_{K_m},\mu_{K_m},\theta_{K_m})$ to further bound the right-hand side of \eqref{IEQ:EKM}.
    Using lower semicontinuity arguments (similar to those in the proof of Theorem~\ref{THEOREM:EOWS}), and recalling \eqref{EQ:EKM=E0},    
    passing to the limit $m\to\infty$ leads to the corresponding energy inequality for $(\phi_\ast,\psi_\ast,\mu_\ast,\theta_\ast)$. \\[0.3em]
    This proves that the quadruplet $(\phi_\ast,\psi_\ast,\mu_\ast,\theta_\ast)$ is a weak solution of \eqref{EQ:SYSTEM} in the sense of Definition~\ref{DEF:WS} with $K=0$.
\end{proof}

\subsection{The limit \texorpdfstring{$K\rightarrow \infty$}{K -> inf} and the existence of a weak solution if \texorpdfstring{$(K,L)\in\{\infty\}\times(0,\infty)$}{K = inf}}

\begin{theorem}\textnormal{(Asymptotic  limit $K\rightarrow\infty$)}\label{THEOREM:K->inf}
    Suppose that the assumptions \ref{ASSUMP:1}--\ref{ASSUMP:POTENTIALS:2} hold, let $\scp{\phi_0}{\psi_0}\in\mathcal{H}^1$ be arbitrary initial data, $L\in(0,\infty)$, $\boldsymbol{v}\in L^2(0,T;\mathbf{L}_\Div^3(\Om))$ and $\boldsymbol{w}\in L^2(0,T;\mathbf{L}^3_\tau(\Ga))$. For any $K\in(0,\infty)$,  let $(\phi_K,\psi_K,\mu_K,\theta_K)$ denote a weak solution of the system \eqref{EQ:SYSTEM} in the sense of Definition~\ref{DEF:WS} with initial data $\scp{\phi_0}{\psi_0}$. Then, there exists a quadruplet $(\phi^\ast,\psi^\ast,\mu^\ast,\theta^\ast)$ such that
    \begin{alignat*}{3}
        \scp{\delt\phi^K}{\delt\psi^K} &\rightarrow \scp{\delt\phi^\ast}{\delt\psi^\ast} \quad\quad&&\text{weakly in } L^2(0,T;(\mathcal{H}^1)'), \\[0.3em]
        \scp{\phi^K}{\psi^K} &\rightarrow  \scp{\phi^\ast}{\psi^\ast} &&\text{weakly-star in } L^\infty(0,T;\mathcal{H}^1),\nonumber\\
        & &&\text{strongly in } C([0,T];\mathcal{H}^s) \text{ for all } s\in[0,1),\\[0.3em]
        \scp{\mu^K}{\theta^K} &\rightarrow \scp{\mu^\ast}{\theta^\ast} &&\text{weakly in } L^4(0,T;\mathcal{L}^2)\cap L^2(0,T;\mathcal{H}^1), 
        \\[0.3em]
        \frac{1}{K}\left(\alpha\psi^K - \phi^K\right) &\rightarrow 0 &&\text{strongly in } L^\infty(0,T,L^2(\Ga)),
    \end{alignat*}
    as $K\rightarrow\infty$, up to subsequence extraction, with
    \begin{align}
        \label{EST:DNPHIKM}
        \frac{1}{K}\norm{\alpha\psi^K-\phi^K}_{L^\infty(0,T;L^2(\Gamma))} \leq \frac{C}{\sqrt{K}},
    \end{align}
    and the limit $(\phi^\ast,\psi^\ast,\mu^\ast,\theta^\ast)$ is a weak solution to the system \eqref{EQ:SYSTEM} in the sense of Definition~\ref{DEF:WS} with $K=\infty$, which additionally satisfies $\scp{\mu^\ast}{\theta^\ast}\in L^4(0,T;\mathcal{L}^2)$.
\end{theorem}

\begin{remark}
    Suppose that for any $K\in(0,\infty)$, the phase-field $\phi^K$ is sufficiently regular such that the boundary condition \eqref{EQ:SYSTEM:5} holds in the strong sense, that is
        \begin{align*}
        K \deln \phi^K = \alpha \psi^K - \phi^K
        \quad
        \text{a.e.~on $\Sigma$}.
    \end{align*}    
    This is actually fulfilled under additional assumptions on the regularity of $\Gamma$ and the parameter $p$ (if $d=3$), see Theorem~\ref{THEOREM:REG}.
    Then, estimate \eqref{EST:DNPHIKM} can be reformulated as
    \begin{align*}
        \norm{\deln\phi^K}_{L^2(\Sigma)} \leq \frac{C}{\sqrt{K}}.
    \end{align*}
    As the right-hand side of this inequality tends to zero as $K\to\infty$, this explains why the homogeneous Neumann boundary condition $\deln \phi^* = 0$ a.e.~on $\Sigma$ appears in the limit model corresponding to $K=\infty$.
\end{remark}

\begin{proof}[Proof of Theorem~\ref{THEOREM:K->inf}]
    We consider an arbitrary sequence $(K_m)_{m\in\N}\subset(0,\infty)$ such that $K_m\rightarrow\infty$ as $m\rightarrow\infty$. Without loss of generality, we assume that $K_m\in[1,\infty)$ for all $m\in\N$. For any $m\in\N$, let $(\phi^{K_m},\psi^{K_m},\mu^{K_m},\theta^{K_m})$ denote a weak solution of the system \eqref{EQ:SYSTEM} in the sense of Definition~\ref{DEF:WS} with initial data $\scp{\phi_0}{\psi_0}$ corresponding to the parameter $K_m$.
    In this proof, we use the letter $C$ to denote generic positive constants independent of $K_m$ and $m$. Let now $m\in\N$ be arbitrary.
    
    As we have seen in the proof of Theorem~\ref{THEOREM:EOWS} (see \ref{EST:initial-energy}), the initial energy satisfies
    \begin{align}
        E_{K_m}(\phi^{K_m}(0),\psi^{K_m}(0)) \leq C(1+K_m^{-1}) \leq C
    \end{align}
    since $K_m \geq 1$ for all $m\in\N$. This allows us to use the same argumentation as in the proof of Theorem~\ref{THEOREM:K->0}, and we infer
    \begin{align}
        \begin{split}
        &\norm{\scp{\phi^{K_m}}{\psi^{K_m}}}_{L^\infty(0,T;\mathcal{H}^1)} + K_m^{-1/2}\norm{\alpha\psi^{K_m}-\phi^{K_m}}_{L^\infty(0,T;L^2(\Ga))}  \label{jonas5} \\
        &\quad + \norm{\scp{\Grad\mu^{K_m}}{\Gradg\theta^{K_m}}}_{L^2(0,T;\mathcal{L}^2)} + \norm{\beta\theta^{K_m}-\mu^{K_m}}_{L^2(0,T;L^2(\Ga))} \leq C.
        \end{split}
    \end{align}
    This already implies \eqref{EST:DNPHIKM}.
    We now test the weak formulation \eqref{WF:MT} with $\scp{\zeta}{\xi}\in\mathcal{H}^1$. Using again that $K_m \geq 1$, we find that
    \begin{align}\label{jonas7}
        \norm{\scp{\mu^{K_m}}{\theta^{K_m}}}_{L^\infty(0,T;(\mathcal{H}^1)^\prime)} \leq C.
    \end{align}
    This estimate allows us to apply Lemma~\ref{LEMMA:DUALITY}, and using \eqref{jonas5} and \eqref{jonas7}, we conclude 
    \begin{align}\label{EST:FULLMT:KM:INF}
        \norm{\scp{\mu^{K_m}}{\theta^{K_m}}}_{L^4(0,T;\mathcal{L}^2)} \leq C.
    \end{align}
    In combination with \eqref{jonas5} we thus deduce
    \begin{align}\label{EST:MT:FULL:KM:INF}
        \norm{\scp{\mu^{K_m}}{\theta^{K_m}}}_{L^2(0,T;\mathcal{H}^1)} \leq C.
    \end{align}
    Lastly, exploiting \eqref{jonas5}, we conclude from the weak formulation \eqref{WF:PP} by a comparison argument that
    \begin{align}\label{EST:DELT:KM:INF}
        \norm{\scp{\delt\phi^{K_m}}{\delt\psi^{K_m}}}_{L^2(0,T;(\mathcal{H}^1)^\prime)} \leq C.
    \end{align}
    In summary, combining \eqref{jonas5} and \eqref{EST:FULLMT:KM:INF}--\eqref{EST:DELT:KM:INF}, we thus have shown that
    \begin{align}\label{EST:UNIF:KM:INF}
        \begin{split}
            &\norm{\scp{\delt\phi^{K_m}}{\delt\psi^{K_m}}}_{L^2(0,T;(\mathcal{H}^1)^\prime)} 
            + \norm{\scp{\phi^{K_m}}{\psi^{K_m}}}_{L^\infty(0,T;\mathcal{H}^1)} 
            \\
            &\quad  
            + \norm{\scp{\mu^{K_m}}{\theta^{K_m}}}_{L^4(0,T;\mathcal{L}^2)}
            + \norm{\scp{\mu^{K_m}}{\theta^{K_m}}}_{L^2(0,T;\mathcal{H}^1)}
            \leq C.
        \end{split}
    \end{align}
    In view of the uniform estimate \eqref{EST:UNIF:KM:INF}, the Banach--Alaoglu theorem and the Aubin--Lions--Simon lemma imply the existence of functions $(\phi^\ast,\psi^\ast,\mu^\ast,\theta^\ast)$ such that
    \begin{alignat}{3}
        \scp{\delt\phi^{K_m}}{\delt\psi^{K_m}} &\rightarrow \scp{\delt\phi^\ast}{\delt\psi^\ast} \quad\quad&&\text{weakly in } L^2(0,T;(\mathcal{H}^1)'), \label{CONV:DELT:KM:INF}
        \\[0.3em]
        \scp{\phi^{K_m}}{\psi^{K_m}} &\rightarrow  \scp{\phi^\ast}{\psi^\ast} &&\text{weakly-star in } L^\infty(0,T;\mathcal{H}^1),\nonumber\\
        & &&\text{strongly in } C([0,T];\mathcal{H}^s) \text{ for all } s\in[0,1), \label{CONV:PP:KM:INF}
        \\[0.3em]
        \scp{\mu^{K_m}}{\theta^{K_m}} &\rightarrow \scp{\mu^\ast}{\theta^\ast} &&\text{weakly in } L^4(0,T;\mathcal{L}^2)\cap L^2(0,T;\mathcal{H}^1), 
        \label{CONV:MTDIFF:KM:INF}
    \end{alignat}
    as $m\rightarrow\infty$, along a non-relabeled subsequence. Additionally, thanks to \eqref{jonas5} and the trace theorem, we infer
    \begin{align}\label{CONV:PPB:KM:INF}
        \frac{1}{K_m}(\alpha\psi^{K_m} - \phi^{K_m})\rightarrow 0 \quad\text{strongly in } L^\infty(0,T;L^2(\Ga))
    \end{align}
    as $m\rightarrow\infty$. We readily deduce that Definition~\ref{DEF:WS:REG} is satisfied. 
    
    Using the Sobolev embedding $\mathcal{H}^s\emb\mathcal{L}^2$ for $s\in (0,1)$ along with \eqref{CONV:PP:KM:INF}, we find that
    \begin{align}
        \scp{\phi^{K_m}(0)}{\psi^{K_m}(0)}\rightarrow \scp{\phi^\ast(0)}{\psi^\ast(0)}\quad\text{strongly in } \mathcal{L}^2,
    \end{align}
    as $m\rightarrow\infty$. In view of the initial conditions satisfied by $(\phi^{K_m},\psi^{K_m})$, we deduce that Definition~\ref{DEF:WS:IC} is fulfilled. 

    As in the proof of Theorem~\ref{THEOREM:K->0}, the convergences \eqref{CONV:DELT:KM:INF}--\eqref{CONV:PPB:KM:INF} are sufficient to pass to the limit in the weak formulations \eqref{WF} associated with $K_m$ to conclude that the limit $(\phi^*,\psi^*,\mu^*,\theta^*)$ fulfils the weak formulation \eqref{WF} associated with $K=\infty$.
    Thus, we infer that Definition~\ref{DEF:WS:WF} is satisfied. 
    
    The mass conservation law as stated in Definition~\ref{DEF:WS:MCL} also follows with the same reasoning as in the proof of Theorem~\ref{THEOREM:K->0}.
    
    For the energy inequality, note that for all non-negative $\sigma\in C_c^\infty(0,T)$, we have
    \begin{align}
        \begin{split}
            \int_0^T E_\infty(\phi^\ast(t),\psi^\ast(t))\,\sigma(t)\dt &\leq \liminf_{m\rightarrow\infty}\int_0^T E_\infty(\phi^{K_m}(t),\psi^{K_m}(t))\,\sigma(t)\dt \\
            &\le \liminf_{m\rightarrow\infty} \int_0^T E_{K_m}(\phi^{K_m}(t),\psi^{K_m}(t))\,\sigma(t)\dt.
        \end{split}
    \end{align}
    The first inequality follows the same argumentation that we already have seen before, whereas the second inequality is due to \eqref{CONV:PPB:KM:INF}. We finish the proof by mimicking the proof of the energy inequality in Theorem~\ref{THEOREM:EOWS}, which is based on the convergence results \eqref{CONV:DELT:KM:INF}--\eqref{CONV:MTDIFF:KM:INF} and the assumptions \ref{ASSUMP:MOBILITY}--\ref{ASSUMP:POTENTIALS:1}. 
    
    We thus have shown that the quadruplet $(\phi^\ast,\psi^\ast,\mu^\ast,\theta^\ast)$ is a weak solution to the system \eqref{EQ:SYSTEM} in the sense of Definition~\ref{DEF:WS} for $K=\infty$.
\end{proof}

\subsection{The limit \texorpdfstring{$L\rightarrow 0$}{L -> 0} and the existence of a weak solution if \texorpdfstring{$(K,L)\in[0,\infty]\times\{0\}$}{L = 0}}

In this section, we investigate the asymptotic limits $L\rightarrow 0$ and $L\rightarrow\infty$ of the system \eqref{EQ:SYSTEM} for fixed $K\in[0,\infty]$. 

\begin{theorem}\textnormal{(Asymptotic  limit $L\rightarrow 0$)}\label{THEOREM:L->0}
    Suppose that the assumptions \ref{ASSUMP:1}--\ref{ASSUMP:POTENTIALS:1}
    hold, let $\scp{\phi_0}{\psi_0}\in\mathcal{H}^1_{K,\alpha}$ be arbitrary initial data for $K\in[0,\infty]$, $\boldsymbol{v}\in L^2(0,T;\mathbf{L}_\Div^3(\Om))$ and $\boldsymbol{w}\in L^2(0,T;\mathbf{L}^3_\tau(\Ga))$. 
    In the case $K\in \{0,\infty\}$, we further assume that \ref{ASSUMP:POTENTIALS:2} holds.
    For any $L\in(0,\infty)$, let $(\phi_L,\psi_L,\mu_L,\theta_L)$ denote a weak solution of the system \eqref{EQ:SYSTEM} in the sense of Definition~\ref{DEF:WS} with initial data $\scp{\phi_0}{\psi_0}$. Then, there exists a quadruplet $(\phi_\ast,\psi_\ast,\mu_\ast,\theta_\ast)$ with $\mu_\ast = \beta\theta_\ast$ a.e.~on $\Sigma$ such that
    \begin{alignat*}{3}
        \scp{\delt\phi_L}{\delt\psi_L} &\rightarrow \scp{\delt\phi_\ast}{\delt\psi_\ast} \quad\quad&&\text{weakly in } L^2(0,T;\mathcal{D}_\beta^\prime), \\[0.3em]
        \scp{\phi_L}{\psi_L} &\rightarrow  \scp{\phi_\ast}{\psi_\ast} &&\text{weakly-star in } L^\infty(0,T;\mathcal{H}^1_{K,\alpha}),\nonumber\\
        & &&\text{strongly in } C([0,T];\mathcal{H}^s) \text{ for all } s\in[0,1),  \\[0.3em]
        \scp{\mu_L}{\theta_L} &\rightarrow \scp{\mu_\ast}{\theta_\ast} &&\text{weakly in } L^2(0,T;\mathcal{H}^1),  \nonumber \\
        & &&\text{weakly in } L^4(0,T;\mathcal{L}^2) \ \text{if } K\in(0,\infty], \\[0.3em]
        \beta\theta_L - \mu_L &\rightarrow 0 &&\text{strongly in } L^2(0,T;L^2(\Ga)), 
    \end{alignat*}
    as $L\rightarrow 0$, up to subsequence extraction, with
    \begin{align}
    \label{EST:L->0}
        \norm{\beta\theta_L-\mu_L}_{L^2(\Sigma)} \leq C\sqrt{L},
    \end{align}
    and the quadruplet $(\phi_\ast,\psi_\ast,\mu_\ast,\theta_\ast)$ is a weak solution to the system \eqref{EQ:SYSTEM} in the sense of Definition~\ref{DEF:WS} with $L=0$, which additionally satisfies $\scp{\mu_\ast}{\theta_\ast}\in L^4(0,T;\mathcal{L}^2)$ in the case $K\in(0,\infty]$.
\end{theorem}

\begin{remark}
    As the right-hand side in \eqref{EST:L->0} tends to zero as $L\to 0$, this explains why the Dirichlet type boundary condition $\mu_* = \beta\theta_*$ a.e.~on $\Sigma$ appears in the limit model corresponding to $L=0$.
\end{remark}

\begin{proof}[Proof of Theorem~\ref{THEOREM:L->0}]
    We consider an arbitrary sequence $(L_m)_{m\in\N}\subset(0,\infty)$ such that $L_m\rightarrow\infty$ as $m\rightarrow\infty$ and a corresponding weak solution $(\phi_{L_m},\psi_{L_m},\theta_{L_m},\mu_{L_m})$ to the system \eqref{EQ:SYSTEM} in the sense of Definition~\ref{DEF:WS} to the initial data $\scp{\phi_0}{\psi_0}$.
    In this proof, we denote by $C$ arbitrary positive constants independent of $L_m$ and $m$, which may change their value from line to line. Let now $m\in\N$ be arbitrary.
    
    If $K\in[0,\infty)$, we conclude from the energy inequality \eqref{WEDL} that
    \begin{align}
        \begin{split}
            &\norm{\scp{\phi_{L_m}}{\psi_{L_m}}}_{L^\infty(0,T;\mathcal{H}^1)} + \h(K)^{1/2}\norm{\alpha\psi_{L_m} - \phi_{L_m}}_{L^2(0,T;L^2(\Ga))}  \\
            &\quad + \norm{\scp{\Grad\mu_{L_m}}{\Gradg\theta_{L_m}}}_{L^2(0,T;\mathcal{L}^2)} + L_m^{-1/2}\norm{\beta\theta_{L_m} - \mu_{L_m}}_{L^2(0,T;L^2(\Ga))} \leq C. \label{EST:LM:0}
        \end{split}
    \end{align}
    If $K=\infty$, we do not have the bulk-surface Poincar\'{e} inequality at our disposal. Instead, we have to argue as in the proof of Theorem~\ref{THEOREM:K->0} and make use of the additional assumption \ref{ASSUMP:POTENTIALS:2} to obtain the estimate \eqref{EST:LM:0}. In particular, \eqref{EST:LM:0} yields \eqref{EST:L->0}. Arguing as in the proof of Theorem~\ref{THEOREM:EOWS}, we additionally infer that
    \begin{align}
        \norm{\scp{\mu_{L_m}}{\theta_{L_m}}}_{L^\infty(0,T;(\mathcal{H}^1_{K,\alpha})^\prime)} \leq C.
    \end{align}
    If $K\in(0,\infty]$, we may use Lemma~\ref{LEMMA:DUALITY} and further obtain
    \begin{align}\label{EST:MT:L^4-L^2:LM:0}
        \norm{\scp{\mu_{L_m}}{\theta_{L_m}}}_{L^4(0,T;\mathcal{L}^2)} \leq C,
    \end{align}
    which in combination with \eqref{EST:LM:0} yields
    \begin{align}\label{EST:MT:FULL:LM:0} 
        \norm{\scp{\mu_{L_m}}{\theta_{L_m}}}_{L^2(0,T;\mathcal{H}^1)} \leq C.
    \end{align}
    In the case $K=0$, we argue analogously as in the proof of Theorem~\ref{THEOREM:K->0} to deduce that
    \begin{align}
        \norm{\scp{\mu_{L_m}}{\theta_{L_m}}}_{L^2(0;T;\mathcal{L}^2)} \leq C,
    \end{align}
    from which we obtain \eqref{EST:MT:FULL:LM:0} as well. 
    
    For the time derivatives, we again proceed similarly as in the proof of Theorem~\ref{THEOREM:EOWS} but choose the test function space $\mathcal{D}_\beta$ instead of $\mathcal{H}^1$. We then obtain due to the weak formulation \eqref{WF:PP} and the uniform bound \eqref{EST:LM:0} that
    \begin{align}\label{EST:DELT:LM:0}
        \norm{\scp{\delt\phi_{L_m}}{\delt\psi_{L_m}}}_{L^2(0,T;\mathcal{D}_\beta^\prime)} \leq C.
    \end{align}
    In view of the uniform estimates \eqref{EST:LM:0}, \eqref{EST:MT:L^4-L^2:LM:0}, \eqref{EST:MT:FULL:LM:0} and \eqref{EST:DELT:LM:0}, the Banach--Alaoglu theorem and the Aubin--Lions--Simon lemma imply the existence of functions $\phi_\ast, \psi_\ast, \mu_\ast$ and $\theta_\ast$ such that
    \begin{alignat}{3}
        \scp{\delt\phi_{L_m}}{\delt\psi_{L_m}} &\rightarrow \scp{\delt\phi_\ast}{\delt\psi_\ast} \quad\quad&&\text{weakly in } L^2(0,T;\mathcal{D}_\beta^\prime), \\[0.3em]
        \scp{\phi_{L_m}}{\psi_{L_m}} &\rightarrow  \scp{\phi_\ast}{\psi_\ast} &&\text{weakly-star in } L^\infty(0,T;\mathcal{H}^1_{K,\alpha}),\nonumber\\
        & &&\text{strongly in } C([0,T];\mathcal{H}^s) \text{ for all } s\in[0,1),  \\[0.3em]
        \scp{\mu_{L_m}}{\theta_{L_m}} &\rightarrow \scp{\mu_\ast}{\theta_\ast} &&\text{weakly in } L^2(0,T;\mathcal{H}^1), \nonumber \\
        & &&\text{weakly in } L^4(0,T;\mathcal{L}^2) \ \text{if } K\in(0,\infty], \\[0.3em]
        \beta\theta_{L_m} - \mu_{L_m} &\rightarrow \beta\theta_\ast - \mu_\ast &&\text{weakly in } L^2(0,T;L^2(\Ga)), \label{CONV:MTB:LM:0} 
    \end{alignat}
    as $m\rightarrow\infty$, along a non-relabeled subsequence. Furthermore, we conclude from \eqref{EST:LM:0} that
    \begin{align}
        \norm{\beta\theta_{L_m} - \mu_{L_m}}_{L^2(\Sigma)} \leq C\sqrt{L_m} \rightarrow 0,
    \end{align}
    as $m\rightarrow\infty$. In combination with \eqref{CONV:MTB:LM:0} we infer that $\mu_\ast = \beta\theta_\ast$ a.e.~on $\Sigma$ due to the uniqueness of the limit. Proceeding similarly as in the case $K\rightarrow 0$ (see the proof of Theorem~\ref{THEOREM:K->0}), we eventually show that the quadruplet $(\phi_\ast,\psi_\ast,\mu_\ast,\theta_\ast)$ is a weak solution of the system \eqref{EQ:SYSTEM} in the sense of Definition~\ref{DEF:WS} for $L=0$.
\end{proof}

\subsection{The limit \texorpdfstring{$L\rightarrow \infty$}{L -> inf} and the existence of a weak solution if \texorpdfstring{$(K,L)\in[0,\infty]\times\{\infty\}$}{L = inf}}

\begin{theorem}\textnormal{(Asymptotic  limit $L\rightarrow\infty$)}\label{THEOREM:L->inf}
    Suppose that the assumptions \ref{ASSUMP:1}--\ref{ASSUMP:POTENTIALS:1} hold, let $\scp{\phi_0}{\psi_0}\in\mathcal{H}^1_{K,\alpha}$ be arbitrary initial data for $K\in[0,\infty]$, $\boldsymbol{v}\in L^2(0,T;\mathbf{L}_\Div^3(\Om))$ and $\boldsymbol{w}\in L^2(0,T;\mathbf{L}^3_\tau(\Ga))$. 
    In the case $K\in \{0,\infty\}$, we further assume that \ref{ASSUMP:POTENTIALS:2} holds.
    For any $L\in(0,\infty)$, let $(\phi^L,\psi^L,\mu^L,\theta^L)$ denote a weak solution to the system \eqref{EQ:SYSTEM} in the sense of Definition~\ref{DEF:WS} with initial data $\scp{\phi_0}{\psi_0}$. Then there exists a quadruplet $(\phi^\ast,\psi^\ast,\mu^\ast,\theta^\ast)$ such that
    \begin{alignat*}{3}
        \scp{\delt\phi^L}{\delt\psi^L} &\rightarrow \scp{\delt\phi^\ast}{\delt\psi^\ast} \quad\quad&&\text{weakly in } L^2(0,T;(\mathcal{H}^1)^\prime), \\[0.3em]
        \scp{\phi^L}{\psi^L} &\rightarrow  \scp{\phi^\ast}{\psi^\ast} &&\text{weakly-star in } L^\infty(0,T;\mathcal{H}^1_{K,\alpha}),\nonumber\\
        & &&\text{strongly in } C([0,T];\mathcal{H}^s) \text{ for all } s\in[0,1),  \\[0.3em]
        \scp{\mu^L}{\theta^L} &\rightarrow \scp{\mu^\ast}{\theta^\ast} &&\text{weakly in } L^2(0,T;\mathcal{H}^1), \nonumber \\
        & &&\text{weakly in } L^4(0,T;\mathcal{L}^2) \ \text{if } K\in(0,\infty], \\[0.3em]
        \frac{1}{L}\left(\beta\theta^L - \mu^L\right) &\rightarrow 0 &&\text{strongly in } L^2(0,T;L^2(\Ga)), 
    \end{alignat*}
    as $L\rightarrow\infty$, up to subsequence extraction, with
    \begin{align}\label{EST:DNMULM}
        \frac{1}{L}\norm{\beta\theta^L-\mu^L}_{L^2(\Sigma)} \leq \frac{C}{\sqrt{L}},
    \end{align}
    and the quadruplet $(\phi^\ast,\psi^\ast,\mu^\ast,\theta^\ast)$ is a weak solution to the system \eqref{EQ:SYSTEM} in the sense of Definition~\ref{DEF:WS} for $L=\infty$, which additionally satisfies $\scp{\mu^\ast}{\theta^\ast}\in L^4(0,T;\mathcal{L}^2)$ in the case $K\in(0,\infty]$.
\end{theorem}

\begin{remark}
    Assuming that for any $L\in(0,\infty)$, the chemical potential $\mu^L$ is sufficiently regular such that the boundary condition \eqref{EQ:SYSTEM:6} holds in the strong sense, that is
    \begin{align*}
        L m_\Omega(\phi^L) \deln \mu^L = \beta\theta^L - \mu^L
        \quad
        \text{a.e.~on $\Sigma$},
    \end{align*}
    estimate \eqref{EST:DNMULM} can be reformulated as
    \begin{align*}
        \norm{\deln\mu^L}_{L^2(\Sigma)} \leq \frac{C}{\sqrt{L}}.
    \end{align*}
    As the right-hand side of this inequality tends to zero as $L\to\infty$, this explains why the homogeneous Neumann boundary condition $\deln \mu^* = 0$ a.e.~on $\Sigma$ appears in the limit model corresponding to $L=\infty$.
\end{remark}

\begin{proof}[Proof of Theorem~\ref{THEOREM:L->inf}]
    We consider an arbitrary sequence $(L_m)\subset(0,\infty)$ such that $L_m\rightarrow\infty$ as $m\rightarrow\infty$ and a corresponding weak solution $(\phi^{L_m},\psi^{L_m},\theta^{L_m},\mu^{Lm})$ to \eqref{EQ:SYSTEM} in the sense of Definition~\ref{DEF:WS} to the initial data $\scp{\phi_0}{\psi_0}$. Since $L_m\rightarrow\infty$ as $m\rightarrow\infty$, we can assume without loss of generality that $L_m\in[1,\infty)$ for all $m\in\N$. In this proof, we use the letter $C$ to denote generic positive constants independent of $L_m$ and $m$, which may change their value from line to line. Let now $m\in\N$ be arbitrary.
    
    Using the same argumentation as in the proof of Theorem~\ref{THEOREM:L->0}, we infer
    \begin{align}
        \begin{split}
        &\norm{\scp{\phi^{L_m}}{\psi^{L_m}}}_{L^\infty(0,T;\mathcal{H}^1)} + \h(K)^{1/2}\norm{\alpha\psi^{L_m} - \phi^{L_m}}_{L^\infty(0,T;L^2(\Ga))} \\
        &\quad + \norm{\scp{\Grad\mu^{L_m}}{\Gradg\theta^{L_m}}}_{L^2(0,T;\mathcal{L}^2)} + L_m^{-1/2}\norm{\beta\theta^{L_m} - \mu^{L_m}}_{L^2(0,T;L^2(\Ga))} \leq C. 
        \end{split}
        \label{EST:LM:INF}
    \end{align}
    In particular, \eqref{EST:LM:INF} already entails \eqref{EST:DNMULM}.
    If $K\in (0,\infty]$, we again derive the estimate
    \begin{align}\label{EST:MT:L^4-L^2:LM:INF}
        \norm{\scp{\mu^{L_m}}{\theta^{L_m}}}_{L^4(0,T;\mathcal{L}^2)} \leq C,
    \end{align}
    while in the case $K=0$, we merely obtain
    \begin{align}
        \norm{\scp{\mu^{L_m}}{\theta^{L_m}}}_{L^2(0,T;\mathcal{L}^2)} \leq C.
    \end{align}
    In both cases, we thus get
    \begin{align}\label{EST:MT:FULL:LM:INF}
        \norm{\scp{\mu^{L_m}}{\theta^{L_m}}}_{L^2(0,T;\mathcal{H}^1)} \leq C.
    \end{align}
    For the time derivatives, we proceed similarly as in the proof of Theorem~\ref{THEOREM:EOWS} and obtain due to the weak formulation \eqref{WF:PP} and the uniform bounds \eqref{EST:LM:INF} that
    \begin{align}\label{EST:DELT:LM:INF}
        \norm{\scp{\delt\phi^{L_m}}{\delt\psi^{L_m}}}_{L^2(0,T;(\mathcal{H}^1)^\prime)} \leq C\left(1 + \frac{1}{\sqrt{L_m}}\right) \leq C.
    \end{align}
    Here we additionally used $L_m \geq 1$ for all $m\in\N$.
    In view of the uniform estimates \eqref{EST:LM:INF}, \eqref{EST:MT:L^4-L^2:LM:INF}--\eqref{EST:MT:FULL:LM:INF} and \eqref{EST:DELT:LM:INF}, the Banach--Alaoglu theorem and the Aubin--Lions--Simon lemma imply the existence of functions $\phi^\ast, \psi^\ast, \mu^\ast$ and $\theta^\ast$ such that
    \begin{alignat}{3}
        \scp{\delt\phi^{L_m}}{\delt\psi^{L_m}} &\rightarrow \scp{\delt\phi^\ast}{\delt\psi^\ast} \quad\quad&&\text{weakly in } L^2(0,T;(\mathcal{H}^1)^\prime), \\[0.3em]
        \scp{\phi^{L_m}}{\psi^{L_m}} &\rightarrow  \scp{\phi^\ast}{\psi^\ast} &&\text{weakly-star in } L^\infty(0,T;\mathcal{H}^1_{K,\alpha}),\nonumber\\
        & &&\text{strongly in } C([0,T];\mathcal{H}^s) \text{ for all } s\in[0,1),  \\[0.3em]
        \scp{\mu^{L_m}}{\theta^{L_m}} &\rightarrow \scp{\mu^\ast}{\theta^\ast} &&\text{weakly in } L^2(0,T;\mathcal{H}^1),  \nonumber\\
        & &&\text{weakly in } L^4(0,T;\mathcal{L}^2) \ \text{if } K \in(0,\infty], 
    \end{alignat}
    as $m\rightarrow\infty$, along a non-relabeled subsequence. 
    Due to \eqref{EST:LM:INF}, we additionally have
    \begin{align}
        \frac{1}{L_m}\norm{\beta\theta^{L_m} - \mu^{L_m}}_{L^2(\Sigma)} \leq \frac{C}{\sqrt{L_m}}\rightarrow 0
    \end{align}
    as $m\rightarrow\infty$. 
    Arguing further as in the case $K\rightarrow\infty$ (see the proof of Theorem~\ref{THEOREM:K->inf}), we eventually show that the quadruplet $(\phi^\ast,\psi^\ast,\mu^\ast,\theta^\ast)$ satisfies Definition~\ref{DEF:WS:REG}--\ref{DEF:WS:WF} and \ref{DEF:WS:WEDL}. For the mass conservation law, simply note that the weak formulations for $\phi^\ast$ and $\psi^\ast$, as well as the test functions, are not coupled, which allows us to use $\zeta\equiv 1$ and $\xi\equiv 1$ as test functions, respectively. Integrating the resulting equations in time from $0$ to $t$, and employing the fundamental theorem of calculus, we infer
    \begin{align*}
        \intO \phi^\ast(t)\dx = \intO\phi_0\dx \quad\text{and}\quad \intG\psi^\ast(t)\dG = \intG \psi_0\dG
    \end{align*}
    for all $t\in[0,T]$, which shows that $(\phi^\ast,\psi^\ast,\mu^\ast,\theta^\ast)$ is a solution of the system \eqref{EQ:SYSTEM} in the sense of Definition~\ref{DEF:WS} for $L=\infty$.
\end{proof}

We conclude this section with the following remark.

\begin{remark}
    If the uniqueness of the respective limiting models is known (e.g. if the mobility functions are constant, cf.~Theorem~\ref{THEOREM:CD}), we even obtain convergence of the \textit{full} sequence (not only a subsequence) in Theorem~\ref{THEOREM:K->0}, Theorem~\ref{THEOREM:K->inf}, Theorem~\ref{THEOREM:L->0} and Theorem~\ref{THEOREM:L->inf}, respectively, by means of a standard subsequence argument.
\end{remark}

\section{Higher regularity, continuous dependence and uniqueness}
\label{SECT:REG+CD}
In this section, we present the results on higher regularity for the phase-fields as well as the continuous dependence and uniqueness of weak solutions to the system \eqref{EQ:SYSTEM} as stated in Theorem~\ref{THEOREM:EOWS} and Theorem~\ref{THEOREM:CD}, respectively.

\subsection{Higher spatial regularity for the phase-fields}
\label{SUBSEC:REG}

\begin{theorem}\textnormal{(Higher regularity)}\label{THEOREM:REG}
    Suppose that the assumptions \ref{ASSUMP:1}--\ref{ASSUMP:POTENTIALS:3} hold, and let $K,L\in[0,\infty]$, $\scp{\phi_0}{\psi_0}\in\mathcal{H}^1_{K,\alpha}$, $\boldsymbol{v}\in L^2(0,T;\mathbf{L}_\Div^3(\Om))$ and $\boldsymbol{w}\in L^2(0,T;\mathbf{L}^3_\tau(\Ga))$. Suppose that the domain $\Om$ is of class $C^k$ for $k\in\{2,3\}$, and in the case $d=3$, we further assume that \ref{ASSUMP:POTENTIALS:1} holds with $p\leq 4$. Then, if a weak solution $(\phi,\psi,\mu,\theta)$ of the system \eqref{EQ:SYSTEM} in the sense of Definition~\ref{DEF:WS} exists, it additionally satisfies
    \begin{align}
        \label{REG:SP:1}
        \scp{\phi}{\psi}&\in L^4(0,T;\mathcal{H}^2)\quad&&\text{if~}K\in(0,\infty], \\
        \label{REG:SP:2}
        \scp{\phi}{\psi}&\in L^2(0,T;\mathcal{H}^k)
        &&\text{for $k\in\{2,3\}$},\\
        \label{REG:SP:3}
        \scp{\phi}{\psi}&\in C([0,T];\mathcal{H}^1) \quad 
        &&\text{if $(K,L)\in[0,\infty]\times (0,\infty]$ and $k=3$}
    \end{align}
    as well as
\begin{align}
    &\mu = -\Lap\phi + F'(\phi) &&\text{a.e.~in } Q, \\
    &\theta = -\Lapg\psi + G'(\psi) + \alpha\deln\phi &&\text{a.e.~on } \Sigma, \\
    &\begin{cases} 
        K\deln\phi = \alpha\psi - \phi &\text{if} \ K\in [0,\infty), \\
        \deln\phi = 0 &\text{if} \ K = \infty
    \end{cases} && \text{a.e.~on } \Sigma.
\end{align}
\end{theorem}

\begin{proof}
    Let $(\phi,\psi,\mu,\theta)$ be a weak solution to the system \eqref{EQ:SYSTEM} in the sense of Definition~\ref{DEF:WS}. We infer from the weak formulation \eqref{WF:MT} that
    \begin{align}\label{WF:REGULARITY}
        \begin{split}
            &\intO\Grad\phi(t)\cdot\Grad\eta\dx + \intG\Gradg\psi(t)\cdot\Gradg\vartheta\dG + \h(K)\intG(\alpha\psi(t)-\phi(t))(\alpha\vartheta-\eta)\dG \\
            &\quad = \intO \big(\mu(t) - F^\prime(\phi(t))\big)\eta\dx + \intG \big(\theta(t) - G^\prime(\psi(t))\big)\vartheta\dG
        \end{split}
    \end{align}
    for almost all $t\in[0,T]$ and all $\scp{\eta}{\vartheta}\in\mathcal{H}^1_{K,\alpha}$. If $K\in[0,\infty)$, this means that the pair $\scp{\phi(t)}{\psi(t)}$ is a weak solution of the bulk-surface elliptic problem
    \begin{subequations}
        \begin{alignat*}{3}
            -\Lap\phi(t) &= f(t) \qquad &&\text{in } \Om, \\
            -\Lapg\psi(t) + \alpha\deln\phi(t) &= g(t) &&\text{on } \Ga, \\
            K\deln\phi(t) &= \alpha\psi(t) - \phi(t) \quad&&\text{on } \Ga,
        \end{alignat*}
    \end{subequations}
    for almost all $t\in[0,T]$, in the sense of \cite[Definition 3.1]{Knopf2021}, where
    \begin{align*}
        f(t) = \mu(t) - F^\prime(\phi(t)), \quad\text{and}\quad g(t) = \theta(t) - G^\prime(\psi(t)).
    \end{align*}
    If $K=\infty$, we may choose $\scp{\eta}{0}\in\mathcal{H}^1$ and $\scp{0}{\vartheta}\in\mathcal{H}^1$ as test functions in \eqref{WF:REGULARITY}, respectively. This yields that $\phi(t)$ is a weak solution to the Poisson-Neumann problem
    \begin{subequations}\label{POISSON-NEUMANN}
        \begin{alignat}{3}
            -\Lap\phi(t) &= f(t) \quad&&\text{in~}\Om, \\
            \deln\phi(t) &= 0 &&\text{on~} \Ga
        \end{alignat}
    \end{subequations}
    for almost all $t\in[0,T]$, whereas $\psi(t)$ is a weak solution to the elliptic problem
    \begin{align}\label{POISSON-BOUNDARY}
        -\Lapg\psi(t) = g(t) \quad\text{on~}\Ga
    \end{align}
    for almost all $t\in[0,T]$. 
    
    Recalling that $p\leq 4$, we use the growth conditions on $F^\prime$ and $G^\prime$ (see \ref{ASSUMP:POTENTIALS:1}), the Sobolev embeddings $H^1(\Om)\emb L^6(\Om)$ and $H^1(\Ga)\emb L^{2(q-1)}(\Ga)$ and the fact that $\scp{\phi}{\psi}\in L^\infty(0,T;\mathcal{H}^1_{K,\alpha})$ to derive the estimate
    \begin{align}\label{EST:F':L^2}
        \norm{F^\prime(\phi(t))}_{L^2(\Om)} &\leq C + C\norm{\phi(t)}^{p-1}_{L^{2(p-1)}(\Om)} \leq C,
    \end{align}
    as well as
    \begin{align}\label{EST:G':L^2}
        \norm{G^\prime(\psi(t))}_{L^2(\Ga)} \leq C + C\norm{\psi(t)}^{q-1}_{L^{2(q-1)}(\Ga)} \leq C,
    \end{align}
    for almost all $t\in[0,T]$. In particular, the estimates \eqref{EST:F':L^2} and \eqref{EST:G':L^2} imply $\scp{f(t)}{g(t)}\in\mathcal{L}^2$ with
    \begin{align}
        \norm{f(t)}_{L^2(\Om)} &\leq C + \norm{\mu(t)}_{L^2(\Om)}, \label{EST:f:L^2}\\
        \norm{g(t)}_{L^2(\Ga)} &\leq C + \norm{\theta(t)}_{L^2(\Ga)} \label{EST:g:L^2},
    \end{align}
    for almost all $t\in[0,T]$.
    
    We first consider the case $k=2$. In the case $K\in[0,\infty)$, we deduce from regularity theory for elliptic problems with bulk-surface coupling (see \cite[Theorem 3.3]{Knopf2021}) that $\scp{\phi(t)}{\psi(t)}\in\mathcal{H}^2$ with
    \begin{align*}
        \begin{split}
            \norm{\scp{\phi(t)}{\psi(t)}}_{\mathcal{H}^2}^2 &\leq C\norm{\scp{f(t)}{g(t)}}_{\mathcal{L}^2}^2 \\
            &\leq C\Big( \norm{\mu(t)}_{L^2(\Om)}^2 + \norm{\theta(t)}_{L^2(\Ga)}^2 + \norm{F^\prime(\phi(t))}_{L^2(\Om)}^2 + \norm{G^\prime(\psi(t))}_{L^2(\Ga)}^2\Big) \\
            &\leq C + C\norm{\mu(t)}_{L^2(\Om)}^2 + C\norm{\theta(t)}_{L^2(\Ga)}^2
        \end{split}
    \end{align*}
    for almost all $t\in[0,T]$. Integrating this inequality in time from $0$ to $T$, we use the regularity of $\scp{\mu}{\theta}$ to infer that $\scp{\phi}{\psi}\in L^2(0,T;\mathcal{H}^2)$. In the case $K\in(0,\infty)$, we even obtain $\scp{\phi}{\psi}\in L^4(0,T;\mathcal{H}^2)$ since $\scp{\mu}{\theta}\in L^4(0,T;\mathcal{L}^2)$. 
    \\
    If $K=\infty$, we may choose
    $\scp{\eta}{\vartheta} = \scp{1}{0}\in\mathcal{H}^1$ as a test function \eqref{WF:REGULARITY}, which yields
    \begin{align*}
        \intO f(t) \dx = \intO \mu(t) - F^\prime(\phi(t)) \dx = 0
    \end{align*}
    for almost all $t\in[0,T]$. This allows us to use regularity theory for Poisson's equations with homogeneous Neumann boundary condition (see, e.g. \cite[s.5, Proposition 7.7]{Taylor}) to infer that $\phi(t)\in H^2(\Om)$ with
    \begin{align*}
        \norm{\phi(t)}_{H^2(\Om)} \leq C \norm{f(t)}_{L^2(\Om)} + C\norm{\phi(t)}_{H^1(\Om)}
    \end{align*}
    for almost all $t\in[0,T]$. In combination with \eqref{EST:f:L^2} we deduce that
    \begin{align}\label{EST:phi:H^2}
        \norm{\phi(t)}_{H^2(\Om)} \leq C + C\norm{\mu(t)}_{L^2(\Om)}
    \end{align}
    for almost all $t\in[0,T]$. 
    Next, due to \eqref{EST:g:L^2}, we can apply regularity theory for elliptic equations on submanifolds (see, e.g., \cite[s.5, Theorem 1.3]{Taylor}) for \eqref{POISSON-BOUNDARY} and infer that $\psi(t)\in H^2(\Ga)$ with
    \begin{align}\label{EST:psi:H^2}
        \norm{\psi(t)}_{H^2(\Ga)} \leq C \norm{g(t)}_{L^2(\Ga)} + C\norm{\psi(t)}_{H^1(\Ga)} \leq C + C\norm{\theta(t)}_{L^2(\Ga)}
    \end{align}
    for almost all $t\in[0,T]$. Squaring the equations \eqref{EST:phi:H^2} and \eqref{EST:psi:H^2} and integrating in time over $[0,T]$ yields $\scp{\phi}{\psi}\in L^2(0,T;\mathcal{H}^2)$ which proves the assertion in the case $K=\infty$. Additionally, as $\scp{\mu}{\theta}\in L^4(0,T;\mathcal{L}^2)$, we obtain $\scp{\phi}{\psi}\in L^4(0,T;\mathcal{H}^2)$.
    
    Let us now consider the case $k=3$. Since $p\leq 4$, we use Hölder's inequality and the Sobolev embedding $H^1(\Om)\emb L^6(\Om)$ to derive the estimate
    \begin{align*}
        \begin{split}
            \norm{F^{\prime\prime}(\phi(t))\Grad\phi(t)}_{L^2(\Om)} &\leq C\norm{\Grad\phi(t)}_{L^2(\Om)} + C\norm{\abs{\phi(t)}^4\Grad\phi(t)}_{L^2(\Om)} \\
            &\leq C\norm{\Grad\phi(t)}_{L^2(\Om)} + C\norm{\phi(t)}_{L^6(\Om)}^2\norm{\Grad\phi(t)}_{L^6(\Om)} \\
            &\leq C + C\norm{\phi(t)}_{H^2(\Om)}.
        \end{split}
    \end{align*}
    for almost all $t\in[0,T]$. In combination with \eqref{EST:F':L^2} this yields
    \begin{align}\label{EST:F':H^1}
        \norm{F^\prime(\phi(t))}_{H^1(\Om)} \leq \norm{F^\prime(\phi(t))}_{L^2(\Om)} + \norm{F^{\prime\prime}(\phi(t))\Grad\phi(t)}_{L^2(\Om)} \leq C + C\norm{\phi(t)}_{H^2(\Om)}
    \end{align}
    for almost all $t\in[0,T]$. Proceeding similarly, we derive the estimate
    \begin{align*}
        \norm{G^{\prime\prime}(\psi(t))\Gradg\psi(t)}_{L^2(\Ga)} \leq C + C\norm{\psi(t)}_{H^2(\Ga)},
    \end{align*}
    which leads to
    \begin{align}\label{EST:G':H^1}
        \norm{G^\prime(\psi(t))}_{H^1(\Ga)} \leq C + C\norm{\psi(t)}_{H^2(\Ga)}
    \end{align}
    for almost all $t\in[0,T]$. We thus obtain $\scp{f(t)}{g(t)}\in\mathcal{H}^1$. If $K\in[0,\infty)$ we may use again regularity theory for elliptic problems with bulk-surface coupling (see \cite[Theorem 3.3]{Knopf2021}) to infer $\scp{\phi(t)}{\psi(t)}\in\mathcal{H}^3$ with
    \begin{align*}
        \begin{split}
            \norm{\scp{\phi(t)}{\psi(t)}}_{\mathcal{H}^3}^2 &\leq C\norm{\scp{f(t)}{g(t)}}_{\mathcal{H}^1}^2 \\
            &\leq C\norm{\mu(t)}_{H^1(\Om)}^2 + C\norm{\theta(t)}_{H^1(\Ga)}^2 + C\norm{F^\prime(\phi(t))}_{H^1(\Om)}^2 + C\norm{G^\prime(\psi(t))}_{H^1(\Ga)}^2 \\
            &\leq C + C\norm{\mu(t)}_{H^1(\Om)}^2 + C\norm{\theta(t)}_{H^1(\Ga)}^2 + C\norm{\phi(t)}_{H^2(\Om)}^2 + C\norm{\psi(t)}_{H^2(\Ga)}^2
        \end{split}
    \end{align*}
    for almost all $t\in[0,T]$. Integrating this inequality in time over $[0,T]$ yields $\scp{\phi}{\psi}\in L^2(0,T;\mathcal{H}^3)$. \\
    If $K=\infty$, the result can be established analogously to the case $k=2$ as the regularity results cited above hold true for any integer $k\geq 0$. We again find that
    \begin{align*}
        \norm{\scp{\phi(t)}{\psi(t)}}_{\mathcal{H}^3} \leq C\norm{\scp{f(t)}{g(t)}}_{\mathcal{H}^1}
    \end{align*}
    for almost all $t\in[0,T]$, where the right-hand side is bounded due to \eqref{EST:F':H^1} and \eqref{EST:G':H^1}. 

    This means that \eqref{REG:SP:1} and \eqref{REG:SP:2} are established. If $(K,L)\in[0,\infty]\times (0,\infty]$ and $k=3$, the additional continuity property \eqref{REG:SP:3} follows from Proposition~\ref{CR}. Therefore, the proof is complete.
\end{proof}

\begin{remark}
    In the other cases, where the boundary conditions involving $K$ and $L$ are of the same type (i.e., $K=L=0$ or $K=L=\infty$), it should also be possible to obtain the regularities \eqref{REG:SP:1} and \eqref{REG:SP:2} even for $d=3$ and $p\in(4,6)$. This is because in these cases, the system \eqref{CH} can be discretized by a Faedo--Galerkin scheme (cf.~Remark~\ref{REM:GAL}) and therefore, the higher order regularity estimates can be performed on the level of the approximate solutions.
    In the case $K=L=0$, we refer the reader to the proof of \cite[Theorem 3.3]{Giorgini2023}, where this line of argument is carried out in detail.
\end{remark}

\subsection{Continuous dependence and uniqueness}\label{SUBSEC:CD}

\begin{proof}[Proof of Theorem~\ref{THEOREM:CD}]
    As the functions $m_\Om$ and $m_\Ga$ are constant, we assume, without loss of generality, that $m_\Om\equiv 1$ and $m_\Ga\equiv 1$. In the following, we use the letter $C$ to denote generic positive constants depending only on $\Om$, the parameters of the system \eqref{EQ:SYSTEM}, and the initial data and the prescribed velocity field. 
    
    We consider two weak solutions $(\phi_1,\psi_1,\mu_1,\theta_1)$ and $(\phi_2,\psi_2,\mu_2,\theta_2)$ corresponding to the initial data $\scp{\phi_0^1}{\psi_0^1}, \scp{\phi_0^2}{\psi_0^2}\in\mathcal{H}^1_{K,\alpha}$, the bulk velocity fields $\boldsymbol{v}_1, \boldsymbol{v}_2\in L^2(0,T;\mathbf{L}_\Div^{3}(\Om))$ and the surface velocity fields $\boldsymbol{w}_1,\boldsymbol{w}_2\in L^2(0,T;\mathbf{L}^{3}(\Ga))$, respectively, and set
    \begin{align*}
        (\phi,\psi,\mu,\theta, \boldsymbol{v},\boldsymbol{w}) \coloneqq (\phi_2 - \phi_1, \psi_2 - \psi_1, \mu_2 - \mu_1, \theta_2 - \theta_1, \boldsymbol{v}_2 - \boldsymbol{v}_1, \boldsymbol{w}_2 - \boldsymbol{w}_1).
    \end{align*}
    Then, due to their respective weak formulations \eqref{WF}, the quadruplet $(\phi,\psi,\mu,\theta)$ satisfies
    \begin{subequations}\label{WF:UQV}
        \begin{align}
            \label{WF:UQV:PP}
            &\ang{\scp{\delt\phi}{\delt\psi}}{\scp{\zeta}{\xi}}_{\mathcal{H}_{L,\beta}^1} - \intO \phi_2\boldsymbol{v}\cdot\Grad\zeta\dx - \intO\phi\boldsymbol{v}_1\cdot\Grad\zeta\dx \nonumber \\
            &\qquad - \intG \psi_2\boldsymbol{w}\cdot\Gradg\xi\dG - \intG\psi\boldsymbol{w}_1\cdot\Gradg\xi\dG  \\
            &\quad= - \intO \Grad\mu\cdot\Grad\zeta\dx - \intG\Gradg\theta\cdot\Gradg\xi\dG 
            - \h(L)\intG(\beta\theta-\mu)(\beta\xi - \zeta)\dG, \nonumber\\
            \begin{split}
            \label{WF:UQV:MT}
            &\intO \mu\,\eta\dx + \intG\theta\,\vartheta\dG  
            - \intO \big[F'(\phi_2) - F'(\phi_1)\big]\eta \dx 
            - \intG\big[G'(\psi_2) - G'(\psi_1)\big] \vartheta \dG \\
            &\quad = \intO\Grad\phi\cdot\Grad\eta \dx 
            + \intG\Gradg\psi\cdot\Gradg\vartheta \dG
            + \h(K)\intG(\alpha\psi-\phi)(\alpha\vartheta - \eta) \dG, 
            \end{split}
        \end{align}
    \end{subequations}
    a.e.~on $[0,T]$ for all $\scp{\zeta}{\xi}\in\mathcal{H}^1_{L,\beta}$ and $\scp{\eta}{\vartheta}\in\mathcal{H}^1_{K,\alpha}$. 
    Now, due to \eqref{initial-data-mean-value}, we have
    \begin{align*}
        \beta\abs{\Om}\meano{\phi} + \abs{\Ga}\meang{\psi} = 0 \quad\text{a.e.~on~} [0,T].
    \end{align*}
    This shows that the pair $\scp{\phi(t)}{\psi(t)}$ is an admissible right-hand side in the elliptic problem with bulk-surface coupling \eqref{SYSTEM:ELLIPTIC} for almost all $t\in[0,T]$. Further, testing the weak formulation \eqref{WF:UQV:PP} with $\scp{\zeta}{\xi} = \scp{\beta}{1}\in\mathcal{H}^1_{L,\beta}$, we infer that
    \begin{align*}
        \beta\abs{\Om}\meano{\delt\phi} + \abs{\Ga}\meang{\delt\psi} = 0 \quad\text{a.e.~on~} [0,T].
    \end{align*}
    Hence, the pair $\scp{\delt\phi(t)}{\delt\psi(t)}$ is an admissible right-hand side in \eqref{SYSTEM:ELLIPTIC} for almost all $t\in[0,T]$ as well. 
    Due to the linearity of the operator $\mathcal{S}_{L,\beta}$, we deduce that 
    $\delt\mathcal{S}_{L,\beta}(\phi,\psi) = \mathcal{S}_{L,\beta}(\delt\phi,\delt\psi)$ in $\mathcal{H}^1_{L,\beta}$ for almost all $[0,T]$.
    Thus, choosing $\scp{\zeta}{\xi} = \mathcal{S}_{L,\beta}(\phi,\psi)\in\mathcal{H}^1_{L,\beta}$ in \eqref{WF:UQV:PP}, we derive the identity
    \begin{align}
            &\ddt\frac{1}{2}\norm{\scp{\phi}{\psi}}_{L,\beta,\ast}^2 
            = \ddt\frac{1}{2} \norm{\mathcal{S}_{L,\beta}(\phi,\psi)}_{L,\beta}^2 \nonumber \\
            &\quad= \scp{\mathcal{S}_{L,\beta}(\delt\phi,\delt\psi)}{\mathcal{S}_{L,\beta}(\phi,\psi)}_{L,\beta} 
            \nonumber\\&\quad
            = - \ang{\scp{\delt\phi}{\delt\psi}}{\mathcal{S}_{L,\beta}(\phi,\psi)}_{\mathcal{H}^1_{L,\beta}} \nonumber \\
            &\quad = \bigscp{\scp{\mu}{\theta}}{\mathcal{S}_{L,\beta}(\phi,\psi)}_{L,\beta} \label{EQ:Gronwall1}
            \\
            &\qquad - \bigscp{\scp{\phi_2\boldsymbol{v} + \phi\boldsymbol{v}_1}{\psi_2\boldsymbol{w} + \psi\boldsymbol{w}_1}}{\scp{\Grad\mathcal{S}_{L,\beta}^\Om(\phi,\psi)}{\Gradg\mathcal{S}_{L,\beta}^\Ga(\phi,\psi)}}_{\mathcal{L}^2} \nonumber \\
            &\quad = -\bigscp{\scp{\mu}{\theta}}{\scp{\phi}{\psi}}_{\mathcal{L}^2} \nonumber \\
            &\qquad - \bigscp{\scp{\phi_2\boldsymbol{v} + \phi\boldsymbol{v}_1}{\psi_2\boldsymbol{w} + \psi\boldsymbol{w}_1}}{\scp{\Grad\mathcal{S}_{L,\beta}^\Om(\phi,\psi)}{\Gradg\mathcal{S}_{L,\beta}^\Ga(\phi,\psi)}}_{\mathcal{L}^2} \nonumber
    \end{align}
    a.e.~on $[0,T]$. Here, the second term on the right-hand side can be bounded by
    \begin{align}\label{jonas22}
        \begin{split}
            &\bigscp{\scp{\phi_2\boldsymbol{v} + \phi\boldsymbol{v}_1}{\psi_2\boldsymbol{w} + \psi\boldsymbol{w}_1}}{\scp{\Grad\mathcal{S}_{L,\beta}^\Om(\phi,\psi)}{\Gradg\mathcal{S}_{L,\beta}^\Ga(\phi,\psi)}}_{\mathcal{L}^2} \\
            &\leq \norm{\scp{\phi_2\boldsymbol{v} + \phi\boldsymbol{v}_1}{\psi_2\boldsymbol{w} + \psi\boldsymbol{w}_1}}_{\mathcal{L}^2} \norm{\scp{\Grad\mathcal{S}_{L,\beta}^\Om(\phi,\psi)}{\Gradg\mathcal{S}_{L,\beta}^\Ga(\phi,\psi)}}_{\mathcal{L}^2} \\
            &\leq \Big(\norm{\scp{\phi_2\boldsymbol{v}}{\psi_2\boldsymbol{w}}}_{\mathcal{L}^2} + \norm{\scp{\phi\boldsymbol{v}_1}{\psi\boldsymbol{w}_1}}_{\mathcal{L}^2}\Big)\norm{\mathcal{S}_{L,\beta}(\phi,\psi)}_{L,\beta} \\
            &\leq \Big(C\norm{\scp{\boldsymbol{v}}{\boldsymbol{w}}}_{\mathcal{L}^{3}} + \norm{\scp{\phi}{\psi}}_{\mathcal{L}^{6}}\norm{\scp{\boldsymbol{v}_1}{\boldsymbol{w}_1}}_{\mathcal{L}^{3}}\Big)\norm{\scp{\phi}{\psi}}_{L,\beta,\ast}
        \end{split}
    \end{align}
    a.e.~on $[0,T]$. Using Young's inequality in combination with the Sobolev embedding $\mathcal{H}^1\emb\mathcal{L}^6$ and the Poincar\'{e} inequality, we obtain
    \begin{align}\label{jonas24}
        \begin{split}
            &\norm{\scp{\phi}{\psi}}_{\mathcal{L}^{6}}\norm{\scp{\boldsymbol{v}_1}{\boldsymbol{w}_1}}_{\mathcal{L}^{3}}\norm{\scp{\phi}{\psi}}_{L,\beta,\ast} \\
            &\leq \ep\norm{\scp{\phi}{\psi}}^2_{K,\alpha} + C_\ep\norm{\scp{\boldsymbol{v}_1}{\boldsymbol{w}_1}}_{\mathcal{L}^{3}}^2\norm{\scp{\phi}{\psi}}_{L,\beta,\ast}^2
        \end{split}
    \end{align}
    a.e.~on $[0,T]$ for all $\ep >0$ and a constant $C_\ep$. Another application of Young's inequality yields
    \begin{align}\label{jonas25}
        \begin{split}
           C\norm{\scp{\boldsymbol{v}}{\boldsymbol{w}}}_{\mathcal{L}^{3}}\norm{\scp{\phi}{\psi}}_{L,\beta,\ast} \leq \frac12\norm{\scp{\boldsymbol{v}}{\boldsymbol{w}}}_{\mathcal{L}^{3}}^2 + C\norm{\scp{\phi}{\psi}}_{L,\beta,\ast}^2
        \end{split}
    \end{align}
    a.e.~on $[0,T]$. Plugging \eqref{jonas24} and \eqref{jonas25} back into \eqref{jonas22}, we infer
    \begin{align}\label{jonas26}
        \begin{split}
            &\bigscp{\scp{\phi_2\boldsymbol{v} + \phi\boldsymbol{v}_1}{\psi_2\boldsymbol{w} + \psi\boldsymbol{w}_1}}{\scp{\Grad\mathcal{S}_{L,\beta}^\Om(\phi,\psi)}{\Gradg\mathcal{S}_{L,\beta}^\Ga(\phi,\psi)}}_{\mathcal{L}^2} \\
            &\leq \frac12\norm{\scp{\boldsymbol{v}}{\boldsymbol{w}}}_{\mathcal{L}^{3}}^2 + \ep\norm{\scp{\phi}{\psi}}_{K,\alpha}^2 + C\norm{\scp{\boldsymbol{v}_1}{\boldsymbol{w}_1}}_{\mathcal{L}^{3}}^2\norm{\scp{\phi}{\psi}}_{L,\beta,\ast}^2
        \end{split}
    \end{align}
    a.e.~on $[0,T]$. For the first term on the right-hand side of \eqref{EQ:Gronwall1} we find
    \begin{align}\label{jonas27}
        \begin{split}
            -\bigscp{\scp{\mu}{\theta}}{\scp{\phi}{\psi}}_{\mathcal{L}^2} &= -\norm{\scp{\phi}{\psi}}_{K,\alpha}^2 - \intO [F^\prime(\phi_2) - F^\prime(\phi_1)]\phi\dx 
            \\&\qquad 
            -\intG[G^\prime(\psi_2) - G^\prime(\psi_1)]\psi\dG
        \end{split}
    \end{align}
    a.e.~on $[0,T]$. Now, exploiting the growth condition on $F^{\prime\prime}$ (see \ref{ASSUMP:POTENTIALS:1}) with $p<6$ and using the fundamental theorem of calculus as well as Hölder's inequality, we compute
    \begin{align}\label{jonas28}
        \begin{split}
            &\left\vert\intO [F^\prime(\phi_2) - F^\prime(\phi_1)]\phi\dx\right\vert 
            = \left\vert\intO\int_0^1 F^{\prime\prime}\big(\tau\phi_2 + (1-\tau)\phi_1\big)\dtau (\phi_2 - \phi_1)\phi\dx\right\vert \\
            &\quad \leq C\intO\left(1 + \abs{\phi_1}^{p-2} + \abs{\phi_2}^{p-2}\right)\abs{\phi}^2\dx \\
            &\quad\leq C\big( 1 + \norm{\phi_1}_{L^6(\Om)}^{p-2} + \norm{\phi_2}_{L^6(\Om)}^{p-2}\big)\norm{\phi}_{L^{12/(8-p)}(\Om)}^2
            \leq C\norm{\phi}_{L^{12/(8-p)}(\Om)}^2
        \end{split}
    \end{align}
    a.e.~on $[0,T]$. Recalling $H^1(\Ga)\emb L^q(\Ga)$,  we analogously infer
    \begin{align}\label{jonas29}
        \left\vert\intG[G^\prime(\psi_2) - G^\prime(\psi_1)]\psi\dG\right\vert &\leq C\norm{\psi}_{L^{12/(8-p)}(\Ga)}^2
    \end{align}
    a.e.~on $[0,T]$.
    Combining \eqref{jonas28}--\eqref{jonas29} yields
    \begin{align*}
            \left\vert\intO [F^\prime(\phi_2) - F^\prime(\phi_1)]\phi\dx\right\vert + \left\vert\intG[G^\prime(\psi_2) - G^\prime(\psi_1)]\psi\dG\right\vert 
            \leq C\norm{\scp{\phi}{\psi}}_{\mathcal{L}^{12/(8-p)}}^2
    \end{align*}
    a.e.~on $[0,T]$.
    Next, recalling again that $p<6$, which entails the compact embedding $\mathcal{H}^1\emb \mathcal{L}^{12/(8-p)}$, we deduce with Ehrling's lemma that
    \begin{align*}
        \norm{\scp{\phi}{\psi}}_{\mathcal{L}^{12/(8-p)}}^2 \leq \ep \norm{\scp{\phi}{\psi}}_{K,\alpha}^2 + C_\ep\norm{\scp{\phi}{\psi}}_{L,\beta,\ast}^2.
    \end{align*}
    Therefore,
    \begin{align*}
        \begin{split}
            &\ddt\frac12\norm{\scp{\phi}{\psi}}_{L,\beta,\ast}^2 + (1 -C\ep)\norm{\scp{\phi}{\psi}}_{K,\alpha}^2\leq \frac12\norm{\scp{\boldsymbol{v}}{\boldsymbol{w}}}_{\mathcal{L}^{3}}^2 +  C\mathcal{F}\norm{\scp{\phi}{\psi}}_{L,\beta,\ast}^2
        \end{split}
    \end{align*}
    a.e.~on $[0,T]$, where $\mathcal{F} \coloneqq \norm{\scp{\boldsymbol{v}_1}{\boldsymbol{w}_1}}_{\mathcal{L}^{3}}^2$.
    Choosing $\ep> 0$ sufficiently small, we conclude
    \begin{align*}
        \ddt\frac12\norm{\scp{\phi}{\psi}}_{L,\beta,\ast}^2 \leq  \frac12\norm{\scp{\boldsymbol{v}}{\boldsymbol{w}}}_{\mathcal{L}^3}^2 +   C\mathcal{F}\norm{\scp{\phi}{\psi}}_{L,\beta,\ast}^2.
    \end{align*}
    As $\mathcal{F}\in L^1(0,T)$, Gronwall's lemma directly implies
    \begin{align*}
        \norm{\scp{\phi(t)}{\psi(t)}}^2_{L,\beta,\ast} &\leq \norm{\scp{\phi(0)}{\psi(0)}}^2_{L,\beta,\ast}\exp\left(C\int_0^t\mathcal{F}(\tau)\dtau\right) \\
        &\quad + \int_0^t\norm{\scp{\boldsymbol{v}(s)}{\boldsymbol{w}(s)}}^2_{\mathcal{L}^3}\exp\left(C\int_s^t\mathcal{F}(\tau)\dtau\right)\ds
    \end{align*}
    for almost all $t\in[0,T]$, which proves \eqref{EST:continuous-dependence}. As a consequence, if $(\phi_0^1,\psi_0^1) = (\phi_0^2,\psi_0^2)$ a.e.~in $\Omega\times \Gamma$, $\boldsymbol{v}_1 = \boldsymbol{v}_2$ a.e.~in $Q$ and $\boldsymbol{w}_1 = \boldsymbol{w}_2$ a.e.~on $\Sigma$, the uniqueness of the corresponding weak solution immediately follows. Thus, the proof is complete.
\end{proof}

\section*{Appendix: Some calculus for bulk-surface function spaces}
\addcontentsline{toc}{section}{Appendix: Some calculus for bulk-surface function spaces}
\renewcommand\thesection{A}
\setcounter{theorem}{0}
\setcounter{equation}{0}

\begin{proposition}\label{CR}
    Let $\Om\subset\R^d$ with $d\in\{2,3\}$ be a non-empty, bounded domain with Lipschitz boundary $\Ga \coloneqq \partial\Om$. Moreover, let $T>0$, $K,L\in [0,\infty]$ and $\alpha,\beta\in\R$ be arbitrary.    
    Let $\scp{u}{v}\in C([0,T];\mathcal{L}^2) \cap L^2(0,T;\mathcal{H}^3)$ with $\scp{-\Lap u}{-\Lapg v + \alpha\deln u}\in\mathcal{H}^1_{L,\beta}$ and
    \begin{align}\label{A:boundary-cond}
        \begin{cases} K\deln u = \alpha v - u &\text{if} \ K\in [0,\infty), \\
        \deln u = 0 &\text{if} \ K = \infty
        \end{cases} \quad\text{a.e.~on } \Sigma.
    \end{align}
    We further suppose that their weak time derivative satisfies $\scp{\delt u}{\delt v}\in L^2(0,T;(\mathcal{H}^1_{L,\beta})^\prime)$. Then, the continuity property $\scp{u}{v}\in C([0,T];\mathcal{H}^1)$ holds, the mapping
    \begin{align*}
        t\mapsto \norm{\scp{u(t)}{v(t)}}^2_{K,\alpha}
    \end{align*}
    is absolutely continuous on $[0,T]$, and the chain rule formula
    \begin{align}\label{A:id:cr}
        \ddt \norm{\scp{u(t)}{v(t)}}^2_{K,\alpha} = 2\bigang{\scp{\delt u}{\delt v}(t)}{\scp{-\Lap u}{-\Lapg v + \alpha\deln u}(t)}_{\mathcal{H}^1_{L,\beta}}
    \end{align}
    holds for almost all $t\in[0,T]$.
\end{proposition}

\begin{proof}[Proof of Proposition~\ref{CR}.]
    In the cases $(K,L)\in [0,\infty)\times\{\infty\}$, a proof was already given in \cite[Proposition~A.1.(b)]{Colli2023}. There, the claim was established by approximating the occurring functions by a sequence of time-mollified functions, and eventually passing to the limit. For all other choices of $(K,L)$, the proof can be carried out analogously.
\end{proof}


\section*{Acknowledgement}
This research was funded by the Deutsche Forschungsgemeinschaft (DFG, German Research Foundation) – Project 52469428. Moreover, the authors were partially supported by the Deutsche Forschungsgemeinschaft (DFG, German Research Foundation) – RTG 2339. The support is gratefully acknowledged.

\section*{Competing Interests and funding}
The authors do not have any financial or non-financial interests that are directly or indirectly related to the work submitted for publication.

\section*{Data availability statement}
No further data is used in this manuscript.


\footnotesize

\bibliographystyle{abbrv}
\bibliography{KS_CCH}

\end{document}